\documentclass{elsarticle} \journal{arXiv}
\usepackage[a4paper, margin=4cm]{geometry}
\usepackage{amsmath,amssymb,amsthm,color,enumerate}
\usepackage[colorlinks,citecolor=magenta,linkcolor=blue]{hyperref}

\theoremstyle{plain}
\newtheorem{theorem}{Theorem}[section]

\theoremstyle{definition}
\newtheorem{definition}[theorem]{Definition}
\newtheorem{proposition}[theorem]{Proposition}
\newtheorem{remark}[theorem]{Remark}
\newtheorem{corollary}[theorem]{Corollary}
\newtheorem{lemma}[theorem]{Lemma}
\newtheorem{example}[theorem]{Example}
\renewenvironment{proof}{{\noindent \bf  Proof.}}{\qed}

\usepackage{verbatim}

\newcommand{\Mr}{D^\psi_{-,\infty}}

\newcommand{\Cl}{\partial^\psi_{+}}
\newcommand{\Cr}{\partial^\psi_{-}}

\newcommand{\RLl}{D^\psi_{+}}

\newcommand{\Crh}{\partial^\psi_{-h}}

\newcommand{\sym}{\psi}
\newcommand{\Gru}{\mathcal G^{\sym}}

\newcommand{\LTp}{\beta}

\newcommand{\skorss}{E}

\newcommand{\Gen}{G}
\newcommand{\AF}{A}

\newcommand{\Yh}[1]{ Y^{{\rm #1},h}}
\newcommand{\Y}[1]{ Y^{{\rm #1}}}

\newcommand{\Nl}[1]{{\rm N^{\text{$l$}}_{\text{$#1$}}}}
\newcommand{\Nr}[1]{{\rm N^{\text{$r$}}_{\text{$#1$}}}}
\newcommand{\Nlone}{{\rm N^{\text{$l$}}}}
\newcommand{\Nrone}{{\rm N^{\text{$r$}}}}
\newcommand{\Dl}{{\rm D^{\text{$l$}} }}
\newcommand{\Dr}{{\rm D^{\text{$r$}} }}
\newcommand{\Dlh}{{\rm \overline{D}^{\text{$l$}}_{\text{$h$}} }}
\newcommand{\Drh}{{\rm \overline{D}^{\text{$r$}}_{\text{$h$}} }}
\newcommand{\Dlzero}{{\rm \overline{D}^{\text{$l$}}_{\text{$0$}} }}
\newcommand{\Drzero}{{\rm \overline{D}^{\text{$r$}}_{\text{$0$}} }}
\newcommand{\Dlp}[1]{{\rm \overline{D}^{\text{$l$}}_{\text{$#1$}} }}
\newcommand{\Drp}[1]{{\rm \overline{D}^{\text{$r$}}_{\text{$#1$}} }}
\newcommand{\Nal}[1]{{\rm N^{*,{\text{$l$}}}_{\text{$#1$}}}}
\newcommand{\Nar}[1]{{\rm N^{*,{\text{$r$}}}_{\text{$#1$}}}}
\newcommand{\Nalone}{{\rm N^{*,{\text{$l$}}}}}
\newcommand{\Narone}{{\rm N^{*,{\text{$r$}}}}}

\newcommand{\indi}[1]{\mathbf 1_{\{ #1 \}}}

 \newcommand{\dpath}{f}
 
 \newcommand{\Grid}[1]{{\rm Grid}_{ #1  }}
  \newcommand{\Gridh}{{\rm Grid}_{ h  }}
 \newcommand{\setinftime}[2]{S_{#1,#2}}

 \newcommand{\supnorm}[2]{ \|#1\|_{#2,\infty} }
  \newcommand{\Supnorm}[2]{ \left\|#1\right\|_{#2,\infty} }

\newcommand{\Leb}{\lambda}
  \newcommand{\Tau}[2]{\tau^{#1}_{#2}}
   \newcommand{\dpathtil}{\widetilde\dpath}
        \newcommand{\stc}{\gamma}
      \newcommand{\stctil}{\widetilde\stc}

\newcommand{\dd}{ { \mathrm{d}} }

\newcommand{\R}{\mathbb {R}}

\newcommand{\BC}{\mathrm{LR}}

\makeatletter
\let\orgdescriptionlabel\descriptionlabel
\renewcommand*{\descriptionlabel}[1]{%
  \let\orglabel\label
  \let\label\@gobble
  \phantomsection 
  \edef\@currentlabel{#1}%
  \let\label\orglabel
  \orgdescriptionlabel{#1}%
}
\makeatother
\allowdisplaybreaks

\begin{document}

\begin{frontmatter}
\title{Boundary conditions for nonlocal one-sided pseudo-differential operators and the associated stochastic processes II\tnoteref{dedication}} 
\tnotetext[dedication]{Dedicated to Mark Meerschaert, you were a great friend and inspiration.}
\fntext[marsden]{Baeumer and Kov\'acs were partially funded by the Marsden Fund administered by the Royal Society of New Zealand.} 
\author{Boris Baeumer\fnref{marsden}} \address{University of Otago, New Zealand}\ead{bbaeumer@maths.otago.ac.nz} 
\author{Mih\'aly Kov\'acs\fnref{marsden}} \address{P\'azm\'any P\'eter Catholic University, Hungary} \ead{kovacs.mihaly@itk.ppke.hu} \author{Lorenzo Toniazzi\fnref{hari}} \address{University of Otago, New Zealand}\fntext[hari]{Toniazzi was fully funded by the Marsden Fund administered by the Royal Society of New Zealand.} \ead{ltoniazzi@maths.otago.ac.nz} 
\begin{abstract} We connect boundary conditions for one-sided pseudo-differential operators with the generators of modified one-sided L\'evy processes. On one hand this allows modellers to use appropriate boundary conditions with confidence when restricting the modelling domain.  On the other hand it allows for numerical techniques based on differential equation solvers to obtain fast approximations of densities or other statistical properties of restricted one-sided L\'evy processes encountered, for example, in finance. In particular we identify a new nonlocal mass conserving boundary condition by showing it corresponds to fast-forwarding, i.e.  removing the time the process spends outside the domain. We treat all combinations of killing, reflecting and fast-forwarding boundary conditions.  

In Part I we show wellposedness of the backward and forward Cauchy problems with a one-sided pseudo-differential operator with boundary conditions as generator. We do so by showing convergence  of Feller semigroups based on grid point approximations of the modified L\'evy process.

In Part II we show that the limiting Feller semigroup is indeed the semigroup associated with the modified L\'evy process by showing continuity of the modifications with respect to the Skorokhod topology. \end{abstract}

\begin{keyword}  nonlocal operator, nonlocal differential equation, spectrally positive L\'evy process, Feller process 

\MSC[2020]{35S15; 60J50; 60J35; 60G51; 60J27}\end{keyword} 
\end{frontmatter}

\tableofcontents

\section{Introduction}
   The last three decades have seen a surge in the application and theoretical development of fractional \cite{MR2218073,MR1658022, MR2090004,MR2884383,Benson2000b,Schumer2009} and nonlocal \cite{T20,MR2343205,MR3156646,MR2780345,MR3987876,D19} integro-differential  equations. An important reason is the discovery of clear probabilistic explanations for solutions and boundary conditions involving  processes with jumps (hence the nonlocality) \cite{MR2884383}, which is the foundation of particle tracking/Monte Carlo numerical methods, e.g. \cite{Zhang2006a}. In contrast to the 
 complete understanding of boundary conditions for one-dimensional  diffusion  (continuous) processes \cite{MR0345224,Peskir15},  one-dimensional L\'evy jump (discontinuous) processes are yet to be fully classified with respect to their   pseudo-differential operators/boundary conditions for the associated backward and forward Kolmogorov equations. Indeed this is an active field of research \cite{T20, MR3467345,  MR2006232, MR3217703,MR3413862,MR3323906,MR4112713}. Identifying boundary conditions is especially complex when one wants to impose a mass conserving boundary condition modeling jumps across the boundary of the domain, as a variety of natural  modifications of the trajectories appear. Four examples are: censoring and its maximal extension \cite{MR2006232}, stochastic reflection  \cite{MR3467345} and fast-forwarding \cite{MR3582209}.  
A particularly well-understood case is the one of the spectrally positive $\alpha$-stable  L\'evy processes $Y$, for $\alpha\in(1,2)$. In fact, the work \cite{MR3720847}, which originates  in \cite{MR3413862}, gives a detailed description of the backward and forward equations identified by restricting trajectories of $Y$ to $[-1,1]$ either by killing, stochastically reflecting or fast-forwarding. In particular, fast-forwarding $Y$ by removing the time $Y$ spends outside $[-1,1]$ results in the nonlocal (or ``reinsertion in the interior'') boundary condition
\begin{equation}\label{eq:ff_intro}
\partial^{\alpha-1}_- f(-1 )=  \int_{0}^{2}f'(y-1) \frac{y^{1-\alpha}}{\Gamma(2-\alpha)} \,\dd y=0
\end{equation}
for the backward equation, where  $\partial^{\alpha-1}_- $ is a (right) Caputo derivative of order $\alpha-1\in(0,1)$ on $(-\infty,1]$. The methods of  \cite{MR3720847} are based on a finite difference/Gr\"unwald approximation of the Feller generator of $Y$, which provides a rigorous and yet intuitive explanation of \eqref{eq:ff_intro}. Briefly, the nonlocal boundary condition \eqref{eq:ff_intro} describes the following conservation of mass: when $Y$    leaves $(-1,\infty)$ by a drift, its mass is redistributed  \textit{inside the domain} (hence the nonlocality) according to the location of $Y$ at its first jump back inside $(-1,\infty)$. This is entirely different from stochastic reflection, where exiting particles are forced to stay on $\{-1\}$, which results in the standard Neumann boundary condition $f'(-1)=0$.  \\

The main purpose of this work is to extend the methods  and results for stable processes in \cite{MR3720847} to  recurrent one-sided L\'evy processes without the aid of scaling properties.  This means characterising the backward and forward Cauchy problems of the restrictions to an interval of $Y$   via a finite difference approximation. For a detailed discussion of these Cauchy problems we refer to the introduction of Part I \cite{BKT20}.\\

We now introduce our one-sided process on $\mathbb R$ (before we restrict it to an interval) and then discuss the results presented in this article.   Spectrally positive L\'evy processes posses a rich and well developed theory   \cite{MR1406564,MR2250061,MR3014147} along with several applications, for example in finance, hydrology, and queues \cite{MR2023021,CG00,CW03,MR1919609,MR2343205,MR0391297,MR1492990,MR2577834,MR3342453}. Importantly, their fluctuation theory features several explicit and semi-explicit formulae not available for most L\'evy processes. Such formulae are   expressed in terms of the scale function $k_0$,   defined by its  Laplace transform $1/\psi$, for $\psi$ being the Laplace exponent of the process   \cite[Chapter VII]{MR1406564}. We follow this tradition obtaining  a full description in terms of scale functions of backward and forward generators of our Feller processes on an interval.  We denote by $Y$ any recurrent  spectrally positive L\'evy process with paths of unbounded variation and no diffusion component.  Our main contribution is Theorem \ref{thm:main_II}, which we reworded below for convenience.
\begin{theorem}\label{thm:main_II_intro}
Restrict $Y$ to a process $\Y{LR}$ on $[-1,1]$ by imposing two boundary conditions at $\{-1,1\}$ for any combination of killing, reflecting and fast-forwarding. Then $\Y{LR}$ is a Feller process with backward and forward  generators given by the Caputo/Riemann–Liouville type operators with boundary conditions in  Table \ref{explicitProcesses_II}.
\end{theorem}
This is obtained by combining the $J_1$-Skorokhod continuity theory developed in 
Section \ref{sec:skor_maps} with the strong convergence of the approximating (backward) semigroups in \cite[Theorem 5.1]{BKT20}. Let us give here a worded description   of the processes $\Y{LR}$.

\begin{table}
\centering
\vline
\begin{tabular}{l|c|c}
  \hline
\hspace{1cm} Process $\Y{LR}$ &
\hspace{.1cm}	Forward generator 
\hspace{.1cm} & \hspace{.1cm} Backward generator
\hspace{.1cm}\\
	\hline
	1.  $\Y{DD}_t=\Dl (\Dr(Y))_t$& $(\Cl , \mathrm{DD})$  & $(\Cr , \mathrm{DD})$   \\
  \hline
2.	$\Y{DN}_t=\Dl  (\Nrone(Y))_t $ & $(\Cl , \mathrm{DN})$ & $(\Cr , \mathrm{DN})$\\
  \hline
3.	 $\Y{ND}_t=\Dr(\Nlone(Y))_t$ & $(\Cl , \mathrm{ND})$ &$(\Cr , \mathrm{ND})$\\
  \hline
4.	$\Y{NN}_t=\Nrone(\Nlone(Y))_t$ & $(\Cl , \mathrm{NN})$ & $(\Cr , \mathrm{NN})$\\
  \hline
5.	$\Y{N^*D}_t= \Dr(\Nalone (Y))_t$ &$(\RLl , \mathrm{N^*D})$ & $(\Cr , \mathrm{N^*D})$\\
  \hline
6.	 $\Y{N^*N}_t= \Nrone(\Nalone (Y))_t$ & $(\RLl , \mathrm{N^*N})$ & $(\Cr , \mathrm{N^*N})$ \\
  \hline 
\end{tabular}\vline
\caption{\label{explicitProcesses_II} This is the same as \cite[Table 1]{BKT20}. It lists the restrictions of the spectrally positive process $Y$ to $[-1,1]$ and the associated forward and backward generators of  strongly continuous contraction semigroups  on    $X=L^1[-1,1]$ and $X=C_0(\Omega)$, respectively. The maps to construct $\Y{LR}$   are defined in Section \ref{sec:skor_maps}. The generators $(\Gen,\BC)$ are defined in \cite[Definition 2.8]{BKT20} and the explicit representation for the domains  of the generators $(\Gen,\BC)$ can be found in \cite[Table 2]{BKT20}.}
\end{table}  

\begin{enumerate}
\item $\Y{DD}$: $Y$  is killed as soon as it leaves $(-1,1)$. 
\item $\Y{DN}$: $Y$ is killed if it drifts across the left boundary. If it jumps across the right boundary we make a time change deleting the time for which $Y$ is to the right of the right boundary. (By  \cite[Lemma 2]{MR1175272}, $\Y{DN}$ equals in law  $\Y{DN^*}$, reflecting $Y$ at the right boundary and then killing it at the left boundary.)
\item $\Y{ND}$: we make a time change deleting the time for which $Y$ is to the left of the left boundary.  This process is then killed if it jumps across the right boundary. 
\item $\Y{NN}$: we make a time change  deleting the time for which $Y$ is outside the domain $(-1,1)$. (This process equals in law $\Y{NN^*}$, i.e. reflecting $Y$ at the right boundary and then fast-forwarding the paths at the left boundary.)
\item $\Y{N^*D}$: $Y$ is reflected at the left boundary and then killed if it jumps across the right boundary.
\item $\Y{N^*N}$: $Y$ is reflected at the left boundary and then fast-forwarded at the right boundary. (This process equals in law the two sided reflection of $Y$ as defined in \cite{MR3467345}, which we prove in Corollary \ref{cor:N*NeqNN}.)
\end{enumerate}

From the descriptions above we see that we treat all possible combinations of $\BC$ for a recurrent $Y$, and in Table \ref{tab:YLR} we list possible alternative definitions for this processes. \\

We obtain as a corollary of Table \ref{explicitProcesses_II} several new and known results concerning exit problems and   resolvent measures for one-sided L\'evy processes. 
The new results   are the representation of the resolvent measures for  $\Y{NN}$ and $\Y{ND}$, the identification of Lebesgue measure as an invariant measure for $\Y{NN}$, and the solution for the  exit problem for  $\Y{ND}$. Note that this exit problem describes a natural quantity, namely the time spent by a spectrally positive L\'evy process  in an interval $[a,b]$ before its first jump above $b$. As for the known results, we provide a new proof of the representation of the resolvent measures for $\Y{DN}$, $\Y{N^*D}$ and  $\Y{N^*N}$, first proved in \cite{MR2054585} and \cite{MR1995924}. Unfortunately, in the  cases $\Y{N^*D}$ and  $\Y{N^*N}$ our work provides a new proof  only by introducing the regularity assumption \cite[(H1)]{BKT20} on the Laplace exponent to tame the singularities arising from the interpolated schemes in  \cite[Section 4]{BKT20} (as discussed in   \cite[Section 1.3]{BKT20}).    (We also recall that we treat only recurrent $Y$ with paths of unbounded variation and no diffusion component.) However,  our proof is based on strong convergence of finite difference schemes and continuity of maps on Skorokhod spaces, which appears very different from  the several existing proofs. This is because none of the existing proofs employ  compound Poisson approximations of $Y$. See for example    \cite{MR2126964} which uses martingale arguments,  \cite{MR2126963} using It\^o excursion theory and \cite{MR2126965} using potential theory, and we refer to \cite{MR3014147,MR4158667, MR2946445,MR3301294} for further discussions of these known results and proofs.  Also worth highlighting that several key steps of our proof are not tied to the one-sidedness of our processes. Thus we believe it can be extended to the study of other processes, such as the symmetric stable process,  which is part of our current investigation. \\


The main technical challenge in this article is proving that the paths of $Y$ are points of continuity of the fast-forwarding map in the $J_1$ topology. In Theorem \ref{thm:ff} (Corollary \ref{cor:ff_twosid}) we give  simple conditions that characterise the points  of continuity of fast-forwarding maps on half-lines (intervals). These conditions cover a wide range of processes with jumps and therefore the presentation in Section \ref{subsec:ff} is given independently of its application in the rest of this article. To illustrate the difficulties behind this result, consider  fast-forwarding below 0, defined for a path $\dpath$  as the composition $\dpath\circ \AF_f^{-1}$, where $\AF_f^{-1}$ is the right inverse of the additive functional $\AF_f(t)=\int_0^t\mathbf 1_{\{\dpath(z)>0\}}\, \dd z$. First of all, note that one must have that the limit path does not spend (Lebesgue) positive time at 0 (see for example Remark \ref{rmk:ff_disc}-(ii)). Once this is assumed, the natural strategy would be to give conditions such that $f_n\to f$ implies $(\dpath_n, \AF_{f_n}^{-1})\to (\dpath, \AF_f^{-1})$ and then use known results on $J_1$ continuity of composition of functions (see, e.g., \cite[Theorem 13.2.2]{MR1876437}). However, to the best of our knowledge, existing results   do not cover one-sided  processes, essentially because,   almost surely,  if  $t$  is a point of discontinuity of $\AF_f^{-1}$, then $\dpath$ is discontinuous at $\AF_f^{-1}(t)$, as we illustrate in Remark \ref{rmk:ff_disc}-(iii). Another issue is that proving this joint convergence to $(\dpath, \AF_f^{-1})$  is a large part of our (constructive) proof, which would not be significantly shortened even if we were to assume that $\dpath$ is continuous (to apply, say, \cite[Theorem 13.2.2]{MR1876437}).  \\
Another important results is Theorem \ref{thm:ffwaitYh}, where we identify the transition rate matrix of fast-forwarding the discrete  Gr\"unwald type approximating process. This allows us to identify  the fast-forwarding boundary condition (in the limit) as discussed in Part I \cite[Section 1.3]{BKT20}. Also note that this work gives the details for the Skorokhod continuity part  of \cite{MR3720847}. Let us finally remark that the articles  \cite{MR3582209, MR1175272, MR3896857} applied fast-forwarding to a one-sided  L\'evy  process, and \cite{MR3582209}  is  the only one that applied fast-forwarding to the boundary where the (stable) process exits by a drift.\\

This work structured as follows: Section \ref{sec:prelim} introduces basic results on stochastic processes and Skorokhod spaces; Section \ref{sec:skor_maps} studies the continuity properties of our maps on the Skorokhod spaces; Section \ref{sec:YandYh} applies these maps to $Y$ and its approximation and connects the latter to the approximating processes constructed in \cite[Section 3.3]{BKT20}; Section \ref{sec:main_II} applies the continuity results to these processes to derive our main results.

\section{Preliminaries}\label{sec:prelim}

We introduce basic notation and results  stochastic processes in Section \ref{subsec:SP} and Skorokhod spaces \ref{subsec:skor}. We will use the standard notation $\mathbb R$, $\mathbb N$, $\mathbb Z$,  $C_c^\infty(\mathbb R)$,  $L^1(E)$, $a\wedge b$, $a\vee b$ and $\mathbf1_A$, which we defined in \cite{BKT20}. Also, we denote by $\|\cdot\|_X$ the norm of a Banach space $X$,  and by $C_0(E)$   the Banach space of real-valued continuous functions on $E$ vanishing at infinity with the supremum norm for $E$ a locally compact metric space \cite[Page 1]{MR3156646}. For a complex valued function $f$ with domain containing $E$, we write $\supnorm{f}{E}=\sup\{|f(x)|:x\in E\}$. We recall the convention that $\Omega$ refers to the interval \[\Omega\in\Big\{(-1,1),[-1,1),(-1,1],[-1,1]\Big\}\] with an endpoint excluded if the problem has a Dirichlet boundary condition (D) there. The symbol $\Leb$ is reserved for the Lebesgue measure (contrarily to Part I \cite{BKT20}).

\subsection{Stochastic processes}\label{subsec:SP} If $E $ is a metric space we denote by $\mathcal B(E)$ its Borel $\sigma$-algebra, and if $\{X_n\}_{n\in \mathbb N},X$ are random variables on $E $, we say  that $\{X_n\}_{n\in \mathbb N}$ converges weakly to $X$, denoted by  $X_n\Rightarrow X$ as $n\to \infty$, if $\mathbb E[g(X_n)]\to \mathbb E[g(X)]$ for all real-valued bounded  continuous functions on $E $. We recall some basic facts about spaces of c\`adl\`ag (right continuous with left limits) functions that can be found, for example, in \cite[Chapter  VI.1]{MR1943877}.
Let $D([0,\infty), \skorss )$ be the  space of c\`adl\`ag functions $\dpath:[0,\infty)\to \skorss$   for $E$ being a locally compact separable metric space. We always    equip  $D([0,\infty), \skorss )$  with the $J_1$ Skorokhod topology and the induced Borel $\sigma$-algebra $\mathcal D_{J_1}$, and recall that  $D([0,\infty), \skorss )$ is a Polish space. For each $t\ge0$ we define the projection map $\pi_t: D([0,\infty), \skorss )\to \skorss  $ by $\pi_t(\dpath)=\dpath(t)$ and recall that $\pi_t$ is measurable. By a c\`adl\`ag   process (with state-space $E$) we mean a probability measure $\mathbb P$ on $D([0,\infty), \skorss )$. So we construct the random variables $(X_t)_{t\ge 0}$ such that $\mathbb P[X_t\in B]=\mathbb P\circ \pi_t^{-1}(B)$,   $B\in \mathcal B(\skorss )$, $t\ge0$. Then by \cite[Problem 1, Chapter 2]{MR838085} and standard relabelling,  we can identify any  c\`adl\`ag   process with a progressive   process $X:D([0,\infty), \skorss )\times [0,\infty)\to \skorss $   with respect to the $\mathbb P$-completion of the natural filtration of $X$, in the terminology of \cite[Chapter 2]{MR838085}. Vice versa, we identify any progressive   process with state-space $E$ and c\`adl\`ag paths with a c\`adl\`ag process on  $D([0,\infty), \skorss )$, by Kolmogorov's Extension Theorem. We use the standard terminology defined  in \cite[Section 1, Chapter 2]{MR838085} for stopping times, their filtrations and related properties. For any $B\in \mathcal B(\skorss )$ we define the first exit time from $B\in\mathcal B(\skorss )$ for     $X$  by
\[\tau^X_B=\inf\big\{t>0: X_t\not\in B\big\},\]
with the convention that $\inf\emptyset=\infty$.  Note that by \cite[Theorem 1.6, Chapter 2]{MR838085}, assuming that $X$  starts at $x\in B$ and it is progressive with respect to a complete right continuous filtration, if $B$ open or closed, then $\tau^X_B$ is a stopping time with respect to this filtration. We remark that in order to simplify notation within certain proofs, the symbols $\tau$ and $\tau_n$ might be defined differently in different proofs. By convention, for a bounded measurable function $g:D([0,\infty), \skorss )\to \skorss $  and denoting by $\mu$ the   distribution of $X_0$, we write $\mathbb E_\mu[g(X)]=\mathbb E[g(X)]$ and $\mathbb P_\mu[X\in B]=\mathbb P[X\in B]$ for $B\in \mathcal D_{J_1}$, and substitute $\mu$ with $x\in \skorss $ if $\mu$ is the Dirac delta measure at $x$.  
We say that $X$ is a (time-homogeneous) Markov or  strong Markov process (with respect to its natural filtration) according to the standard definitions in \cite[pages 156 and 158]{MR838085} , respectively.
 We say that $X$ is a Feller process if $X$ is a strong Markov process and for each    $g\in C_0
  (E)$, $x\mapsto\mathbb E_x[g(X_t)]\in C_0
  (E)$ for every $t>0$ and $\|\mathbb E_{\cdot}[g(X_t)]-g\|_{C_0
  (E)}\to 0$ as $t\downarrow0$. We recall that the completion of the natural filtration of a Feller process is right continuous \cite[Theorem 2.7, Chapter 2]{MR838085}, and that Feller processes with state space $E$ are in one-to-one correspondence with Feller semigroups on $C_0
  (E)$ as defined in \cite[Chapter 1]{MR3156646}.  
For any c\`adl\`ag process $X$ (or a c\`adl\`ag path)  and $B\in \mathcal B(E)$ we define the  continuous non-decreasing   functional \[ t\mapsto A^X_B(t)=\int_0^t\mathbf 1_{\{X_s\in B\}}\,\dd s\] 
and its right continuous inverse taking values in $[0,\infty]$
$$(\AF^X_B)^{-1}(t)=\inf\big\{s>0: \AF^X_B(s)>t\big\},\quad t\in[0,\infty).
$$ 
We will often ease notation by writing $\AF^X_B=\AF_X$ and $(\AF^X_B)^{-1}=(\AF_X)^{-1}$ when the dependence on $B$ is clear. For any Markov process $X$ in this work $\AF^X_B$ will be a continuous additive functional for $X$ in the sense of \cite[Definition IV.1.15]{MR0264757}. 
We say that $X$ is a L\'evy process if $E=\mathbb R$ and $X$ is a   c\`adl\`ag  process with independent and stationary increments, and we recall that if $X$ has paths of unbounded variation, then, for any $x\in(a,b)$, 
\begin{equation}\label{prop:exit_open_closed}
 \tau^X_{(-\infty,b)} = \tau^X_{(-\infty,b]} \quad\text{and}\quad \tau^X_{(a,\infty)} = \tau^X_{[a,\infty)}, \quad \text{$\mathbb P_x$-a.s},  
\end{equation}
which is a consequence of the strong Markov property and regularity of $(-\infty,0)$ and $(0,\infty)$ for 0 \cite[Theorem 6.5]{MR2250061}.\\
A continuous time Markov process $X=(X_t)_{t\ge 0}$ for a collection of initial distributions $X_0\in S_0\subset \skorss$ is a \textit{discrete-space continuous-time Markov process}
  if for each $X_0$ the range $S$ of $X$ is countable and its transition rate matrix $Q=R(P-I)$  has bounded rates, i.e. $\sup_{i\in S}R_i<\infty$, where $P=(P_{i,j})
_{i,j\in S}\ge 0$ is such that $\sum_j P_{i,j}=1$, $R=(R_{i})_{i\in S}$ is a diagonal matrix with finite non-negative entries, and $I$ is the identity matrix.   Then, for a given initial distribution any $X$ is characterised by its transition rate matrix $Q$ \cite[Chapter 4.2]{Stroock13} and we deduce from  \cite[Eq. (4.2.1)]{Stroock13} that for all $i\in S_0$ and $j\in S$
\[ 
e^{-t R_i} = \mathbb P_i[J_1>t]= \mathbb P_i\big[X_s=i \,\,\rm{for\, all}\,\,s\in[0,t]\big] , 
\]
where $J_1$ is the time of the first jump of $X$, and
\[ P_{i,j}=\lim_{t\downarrow 0}\mathbb P_i\big[t<J_1,\, X_{J_1}=j\big],
\]
and note that the above implies that $P_{i,i}=0$ for all $i$.
Recall that by the ergodic result in \cite[Eq. (4.4.11)]{Stroock13},  the $j$-th entry of the vector $ \lim_{\LTp\downarrow 0} \LTp(\LTp I-Q^T)^{-1} \vec e_{i}$ is given by \begin{equation}\label{eq:ergodic_stroock}
  \lim_{\LTp\downarrow 0} \LTp \int_0^\infty e^{-\LTp t} \mathbb P_i\big[X_t=j\big]\,\dd t =\mathbb P_i\left[\sigma_j^{X}<\infty \right] ,\quad \text{if }i\neq j\text{ and }R_j=0,
\end{equation}
 where   $\sigma_j^X:=\inf\{t\ge J_1: X_t= j\}$, $Q^T$ is the transpose of $Q$ and $\vec e_{i}$ is the vector with $i$-th entry equal to 1 and 0 otherwise.

\subsection{Skorokhod spaces and the Continuous Mapping Theorem}\label{subsec:skor}
For $T\in(0,\infty)$, let $D([0,T],\mathbb R)$ and $D([0,\infty),\mathbb R)$ be the spaces real-valued c\`{a}dl\`{a}g functions on $[0,T]$ and $[0,\infty)$, respectively, and we equip the first space with the Skorokhod metric
\[
d_{J_1,T}(\dpath ,g):= \inf_{\gamma\in \Gamma}\Big\{ \| \gamma -I \|_{C[0,T]} \vee \supnorm{\dpath (\gamma)-g }{[0,T]}\Big\},
\]
where $\Gamma $ is the space of strictly increasing continuous bijections of $[0,T]$ to itself and $I\in \Gamma$ is the identity map, and the second space with the Skorokhod metric
\[
d_{J_1}(\dpath ,g):= \sum_{m=1}^\infty 2^{-m}(1\wedge d_{J_1,m}(\dpath |_m,g|_m))
\]
with $\dpath|_m$ and $g|_m$ respectively being the smoothed restrictions  to $D([0,m],\mathbb R)$ (as defined in \cite[Eq. (16.3)]{MR1406564})  of $\dpath $ and $g$. These metrics induce the Skorokhod $J_1$ topology on the respective spaces \cite[Sections 12 and 16]{MR1406564}. We will use the following simple property
\begin{equation}\label{eq:J1scaleout}
d_{J_1}(\dpath ,c\dpath )\le |1-c|,\quad \text{for any }c\in\mathbb R,\,\, \dpath\in D([0,\infty),\mathbb R),
\end{equation} 
which is follows from  $
d_{J_1,m}(\dpath |_m,(c\dpath )|_m)\le |1-c|$. 
We recall that  $\dpath_n\to f$ as $n\to\infty$ in $D([0,\infty),\mathbb R)$ if and only if there exists a sequence $\{T_m\}_{m\in \mathbb N}$ of continuity points of $\dpath$ such that $T_m\to \infty$ and  for each $m\in\mathbb N$  the restrictions of $f_n$ and $f$ to $[0,T_m]$ are such that $\dpath_n\to f$ in $D([0,T_m],\mathbb R)$ as $n\to \infty$ \cite[Section 12.9]{MR1876437}. We denote one dimensional open balls by $B_r(x)=\{y\in\mathbb R:|y-x|<r\}$ for any $x\in \mathbb R$ and $r>0$. We recall that the Borel $\sigma$-algebra on $D([0,T],\mathbb R)$ (respectively $D([0,\infty),\mathbb R)$) equals the smallest $\sigma$-algebra  containing  $\{\pi^{-1}_t(B_{r}(x)):x\in\mathbb R,r>0,t\in \mathcal T)\}$ for $\mathcal T$ a dense set in $[0,T]$ containing $T$ (respectively  a dense set  in $[0,\infty)$) \cite[Theorems 12.5.iii and 16.6.iii]{MR1700749}. Then, for any closed set $C\subset\mathbb R$ and reals $a<b$
  the sets\
  \begin{equation}\label{eq:tunnels_measurable}
     \big\{f: f(t)\in C\,\text{for all }t\in [a,b]     \big\}\quad\text{and}\quad      \big\{f:   f(t)\in C^c \text{ for some }t\in [a,b]\big\}
  \end{equation}
 are measurable in $D([0,\infty),\mathbb R)$. Also,  by the argument in \cite[page 232]{MR0233396}, for any $B\in \mathcal B(\mathbb R)$ the map from $D([0,\infty),\mathbb R)$ to $[0,\infty]$ 
 \begin{equation}\label{eq:timeintunnel_measurable}
 f\mapsto \Leb(t\in [0,\infty):\,f(t)\in B),
 \end{equation}
 is Borel measurable, and for completeness we give the proof in Appendix \ref{app:additional}.
  To ensure that the fast-forwarding maps of Section \ref{subsec:ff} are well-defined we sometimes restrict our analysis to  paths that spend infinite amount of time in the interval $C$ (see Remark \ref{rmk:ff_disc}-(i)).  Namely, for any $a,b$ such that $-\infty \le a<b\le\infty$ we define the sub-space of $D([0,\infty),\mathbb R)$ by
\[
D_{a,b}([0,\infty),\mathbb R)= D([0,\infty),\mathbb R)\cap \setinftime{a}{b} , 
\]
where 
\[
 \setinftime{a}{b} := \big\{\dpath  \in D([0,\infty),\mathbb R): \Leb (t\in[0,\infty ): \dpath (t)\in (a,b))=\infty \big\}.
\]
 We recall two classical theorems and a standard proposition. Below, for a function $g$ between two metric spaces, we denote by ${\rm Disc}(g)$ the set of points of discontinuity of $g$.
 
\begin{theorem}\label{thm:CMT}(Continuous Mapping Theorem \cite[Theorem 3.4.3]{MR1876437})
 If $X_n\Rightarrow X$ on $E $ and  $g:E \to E' $ is a measurable map between two separable  metric spaces satisfying $\mathbb P[ X\in {\rm Disc}(g)]=0$, then  $g(X_n)\Rightarrow g(X)$ on $E' $.
\end{theorem}

\begin{theorem}\label{thm:skor_repr}(Skorokhod Representation Theorem \cite[Theorem 3.2.2]{MR1876437})
 If $X_n\Rightarrow X$ on a separable metric space $E $, then there exist $\widetilde X_n, \widetilde X$ on a common probability space such that
 $\widetilde X_n= X_n$ and   $\widetilde X= X$ in law, and $\mathbb P[\lim_{n\to \infty}\widetilde X_n = \widetilde X]=1$
\end{theorem}

\begin{proposition}\label{prop:asthenweak}
If $E $ is a metric space and $\{X_n\}_{n\in \mathbb N},X$ are random variables on the same probability space on  $E $ satisfying   $\mathbb P[\lim_{n\to \infty}X_n = X]=1$, then $X_n\Rightarrow X$ as $n\to \infty$.  Moreover, if $E =D([0,\infty),\mathbb R)$ and if $\mathbb P[\pi_{t} (X)=\pi_{t-} (X)]=1$ for some $t>0$, then $\pi_{t}$ is continuous at $X$ $\mathbb P$-a.s. and  $\pi_{t} (X_n)\Rightarrow \pi_{t} (X)$  as $n\to \infty$ on $\R$.   
\end{proposition}
\begin{proof} The first statement is standard. The second statement follows by \cite[Theorem 16.6.i]{MR1700749} and  Theorem \ref{thm:CMT}.\end{proof}

The next corollary is a simple weakening of the Continuous Mapping Theorem (CMT), which we need when dealing with the killing maps in Section \ref{sec:skor}.

\begin{corollary}\label{cor:sparated_CMT}
 Suppose that $X_n\Rightarrow X$ as $n\to \infty$  on $E $, for each $n\in\mathbb N\cup\{0\}$  $g_n:E \to E' $ is a measurable map between two separable   metric spaces  and let $f:E \to E' $  satisfy $\mathbb P[ X\in {\rm Disc}(f)]=0$. If it also holds   for each $n\in\mathbb N$ that
 \begin{equation}\label{eq:measurforCMT}
     g_n(X_n)=f(X_n)\,\, \text{a.s.}\quad\text{and}\quad   g_0(X)=f(X)\,\, a.s.,
 \end{equation}
 then  $f(X_n)\Rightarrow f(X)$  as $n\to \infty$ on $E' $.
\end{corollary}
\begin{proof}
The proof is a simple adaptation of the standard argument in \cite[Theorem 3.2.2]{MR1876437}. Namely, construct the pushforward measures of $f(X_n)$ and $f(X)$  using  \eqref{eq:measurforCMT}, then use Theorem \ref{thm:skor_repr} to obtain versions of $X_n$ and $X$ such that  $\widetilde X_n\to \widetilde X$ a.s. in $E $, so that $g_n(\widetilde X_n)\to g_0( \widetilde X)$  a.s. in $E' $  by the continuity assumption on $f$. Then \eqref{eq:measurforCMT} and Proposition \ref{prop:asthenweak} conclude the proof.

\end{proof}

\section{Reflecting, killing  and fast-forwarding maps}\label{sec:skor_maps}

In this section we study the properties of several maps on (subspaces of) $D([0,\infty),\R)$  which impose either   reflecting, killing or fast-forwarding boundary conditions.   The key result is the characterisation of the continuity points of fast-forwarding maps (one and two-sided). These maps are interesting in their own right and therefore we present them with no reference to Part I \cite{BKT20}. 

\subsection{Reflecting ($\rm{N^*}$)}
\begin{definition}\label{def:refl}
For any   $a,b\in \mathbb R$ with $a<b$ we define the reflecting maps on $[a,\infty)$, $(-\infty,b]$ and $[a,b]$ from $D([0,\infty),\mathbb R)$ to itself respectively as  
\begin{align*}
 \Nal{a}(\dpath )(t)&= \dpath (t) +\inf_{s\le t}\{(\dpath (s)-a)\wedge 0\},\\
 \Nar{b}(\dpath )(t)&= \dpath (t) +\sup_{s\le t}\{(\dpath (s)-b)\vee 0\},\\
[\Nal{a}\Nar{b}](\dpath )(t)&= \dpath (t) + \eta_a(t)-\eta_b(t),
\end{align*}
for each $t\in[0,\infty)$, where the   non-decreasing functions $\eta_a,-\eta_b\in D([0,\infty),\mathbb R)$ are such that $[\Nal{a}\Nar{b}](\dpath )$ takes values in $[a,b]$, $\eta_a(0)=\eta_b(0)$ and they satisfy the minimal pushing conditions  
\[
\int_0^\infty \mathbf1_{\{[\Nal{a}\Nar{b}](\dpath )(t)>a\}}\,\dd \eta_a(t)=0,\quad \int_0^\infty \mathbf1_{\{[\Nal{a}\Nar{b}](\dpath )(t)<b\}}\,\dd \eta_b(t)=0.
\]
If $a=-1$ and $b=1$ we may simplify notation by writing $\Nal{a}=\Nalone$ and $\Nar{b}=\Narone$.
\end{definition}
\begin{remark}
Recall that  $\Nal{a} $, $\Nar{b}$ and $ [\rm{N^*_a N^*_b}]$ are all Lipschitz continuous on $D([0,\infty),\mathbb R)$  (see, e.g., \cite[Corollary  1.6]{MR2349573}), and thus they are all Lipschitz continuous as maps from $D_{c,d}([0,\infty),\mathbb R)$ to $D([0,\infty),\mathbb R)$  for any $c<d$.
\end{remark}

\subsection{Killing ($\rm{D}$)}
We model killing/Dirichlet boundary conditions by absorbing the path at the exit barrier the first time it touches or crosses the barrier (exit from an open set). These maps are not continuous in $D([0,\infty),\mathbb R)$, and we believe they are also not measurable. To overcome these issues with little technicality, we identify simple assumptions to show the paths of our spectrally positive process are continuity points of these maps. And to obtain measurability (in order to apply Corollary \ref{cor:sparated_CMT}), we show that the first exit from open intervals equals the first exit from certain closed intervals for our processes, as the first exit from closed intervals is measurable, thanks to \eqref{eq:tunnels_measurable}. 

\begin{definition}\label{def:DD}
We define the killing  maps on $(-1,\infty)$,  $(-\infty,1)$ and $(-1,1)$ from $D([0,\infty),\mathbb R)$ to itself respectively  as
\begin{equation*}
\Dl(f)(t)=\left\{\begin{split}
&f(t), & t<\Tau{f}{(-1,\infty)},\\
&-1, & t\ge\Tau{f}{(-1,\infty)}, 
\end{split} \right. \end{equation*}
\begin{equation*}
\Dr( f )(t)=\left\{\begin{split}
&f(t), & t<\Tau{f}{(-\infty,1)},\\
&1, & t\ge\Tau{f}{(-\infty,1)}, 
\end{split} \right.
\end{equation*}
for each $t\in[0,\infty)$, and $\Dl(\Dr)$, recalling that $ \tau^ f _{(-1,\infty)}=\inf\{s>0: f (s)\le -1\}$ and  $\tau^ f _{(-\infty,1)} =\inf\{s>0: f (s)\ge 1\}$. 

\end{definition}
Note that we have the commutative property $ \Dl(\Dr)=\Dr(\Dl)$. 
\begin{proposition}\label{prop:killmaps_measurable}
For all $h\ge0$, the  maps $\Dlh ,\,\Drh$ and $
\Dlh(\Drh)$ are measurable, where \begin{equation*}
\Dlh(f)(t):= \left\{\begin{split}
&f(t), & t<\Tau{f}{[h-1,\infty)},\\
&-1, & t\ge\Tau{f}{[h-1,\infty)}, 
\end{split} \right.
\end{equation*}
\begin{equation*}
\Drh( f )(t):= \left\{\begin{split}
&f(t), & t<\tau^f_{(- \infty,1-h]},\\
&1, & t\ge\tau^f_{(- \infty,1-h]}. 
\end{split} \right.
\end{equation*}
\end{proposition}

 \begin{proof} We show that $\Dlh$ is measurable, then so are $\Drh$ and $\Dlh(\Drh)$.  Recall it is enough to check  that $\Dlh^{-1}(A) $ is a Borel set in $D([0,\infty),\mathbb R)$  for any 
  $ A=  \pi_{t}^{-1}( B_{r}(x))$ with $x\in\mathbb R,\, r>0$ and $t>0$. Suppose first that $h>0$. If
 $B_{r}(x)\cap (-\infty,h-1)=\emptyset$, then $\Dlh^{-1}(A)=A\cap \{\dpath  \ge h-1 \,\rm{on}\,[0,t]\}$, which is measurable by \eqref{eq:tunnels_measurable}. Otherwise we have
  $B_{r}(x)\cap (-\infty,h-1)\not=\emptyset$. Then, if $B_{r}(x)\subset (-\infty,-1)$, then    $\Dl^{-1}(A)=\emptyset  $. If  $B_{r}(x)\subset (-1,\infty)$  then write $B_{r}(x)=B_1\cup B_2$  where $B_1\subset (-1,h-1)$ and $B_2\subset [h-1,\infty)$, so that
  $$
  \Dlh^{-1}(A) = \Dlh^{-1}(\pi_t^{-1}(B_2)) ,
  $$
  which is measurable by the same argument as in the first part of this proof. The only remaining possibility is that $-1\in B_{r}(x)$. In this latter case rewrite $B_{r}(x)=B_1\cup \{-1\}\cup B_2$ where $B_1\subset (- \infty,-1)$ and $B_2\subset (-1,\infty)$. Then 
\[
\Dlh^{-1}(A)=    \Dlh^{-1}\big( \pi_{t}^{-1}( \{-1\})\big)\cup \Dlh^{-1}(\pi_t^{-1}(B_2)),
\]
and the second set in the right hand side is measurable by the first part of the proof, and the first set is measurable because of \eqref{eq:tunnels_measurable} and
\[
\Dlh^{-1}\big(  \pi_{t}^{-1}( \{-1\})\big)= \{f: \, f(s)< h-1\text{ for some  }s \le t\}.
\] 
 We now treat the case $h=0$ in an analogous fashion. Suppose  that  
 $B_{r}(x)\cap (-\infty,-1)=\emptyset$, then $\Dlzero^{-1}(A)=A\cap \{\dpath  \ge-1 \,\rm{on}\,[0,t]\}$, which is measurable by \eqref{eq:tunnels_measurable}. Otherwise  $B_{r}(x)\cap (-\infty,-1)\not=\emptyset$. In this case, if  $B_{r}(x)\subset (-\infty,-1)$, then    $\Dlzero^{-1}(A)=\emptyset$. The remaining case is that $-1\in B_{r}(x)$ and thus we can write $B_{r}(x)=B_1\cup \{-1\}\cup B_2$ where $B_1\subset (- \infty,-1)$ and $B_2\subset (-1,\infty)$. 
    Then 
\[
\Dlzero^{-1}(A)=   \Dlzero^{-1}\big( \pi_{t}^{-1}( \{-1\})\big)\cup \Dlzero^{-1}( \pi_{t}^{-1}( B_2)),
\]
and the second set in the right hand side is measurable by the first part of the proof, and the first set is measurable because of \eqref{eq:tunnels_measurable} and
\[
\Dlzero^{-1}\big(  \pi_{t}^{-1}( \{-1\})\big)= \{f: \, f(s)< -1\text{ for some  }s \le t\}\cup  \pi_{t}^{-1}( \{-1\}).
\] 

\end{proof}

The proof of the following corollary is straightforward and we omit it. 
 \begin{corollary}\label{cor:for_measurability}
For any  $h>0$ we have the identities: $\Dlh(\dpath_h)=\Dl(\dpath_h)$  if $\dpath_h$ takes values in $-1+h\mathbb Z$; $\Drh(\dpath_h)=\Dr(\dpath_h)$  if $\dpath_h$ takes values in $1+h\mathbb Z$;     $\Dlzero(\dpath)=\Dl(\dpath)$ if   $\tau^{\dpath}_{[-1,\infty)}=\tau^{\dpath}_{(-1,\infty)}$; and $\Drzero(\dpath)=\Dr(\dpath)$ if   $\tau^{\dpath}_{(-\infty,1]}=\tau^{\dpath}_{(-\infty,1)}$.
\end{corollary}
\begin{proposition}\label{prop:SMkill1}
Suppose $\dpath \in D([0,\infty),\mathbb R)$ is such that $f(0)<1$ and for each $\delta>0$
\[
\sup_{0\le t< \tau-\delta} \dpath (t)<1,
\]
and if $\tau:=\Tau{\dpath}{(-\infty,1)}<\infty,$ then $ \dpath (\tau)>1.$
Then, $\dpath $ is a continuity point for $\Dr$.
\end{proposition} 
\begin{proof} Recall it is enough to check continuity   of the restriction of $\dpath$ to $D([0,T_n],\mathbb R)$ for a sequence of continuity points $T_n\to\infty$  of $\dpath$. Continuity is clear if $\tau=\infty$, as the first condition guarantees that $\Dr(\dpath )=\dpath $ and the same holds for any $\dpathtil$ close enough to $\dpath$. Otherwise $\tau<\infty$, and suppose first that $\dpath (\tau-)=1$. For any  continuity point $T>\tau$  of $\dpath $ we run the following argument. \
For an arbitrary $\epsilon>0$, denote by $\epsilon^*=\dpath (\tau)-1>0$,   let  $\delta>0$   such that $\sup_{t\in[\tau-\delta,\tau)}|\dpath (t)-1|\le \epsilon/2$ and let $\epsilon_*>0$ such that
\[
\sup_{0\le t< \tau-\delta} \dpath (t)\le 1-\epsilon_*.
\]
Let $\dpathtil $ such that 
\[
\inf_{\gamma\in\Gamma}\max \left\{\|I-\gamma\|_{C[0,T]},\supnorm{\dpathtil (\gamma)-\dpath }{[0,T]}\right\}< (\epsilon\wedge\epsilon^*\wedge\epsilon_*)/2,
\]
and denote by $\gamma$ any time change satisfying the above inequality and define $\widetilde\tau_\gamma =\inf\{t:\dpathtil (\gamma(t))\ge 1\}$. Then, directly from the definitions, $\Dr(\dpathtil )(\gamma)=\Dr(\dpathtil (\gamma))$ and
\begin{align*}
&\,\supnorm{\Dr(\dpathtil (\gamma))-\Dr(\dpath )}{[0,T]}\\
=&\, \supnorm{\dpathtil (\gamma)-\dpath }{[0,\widetilde\tau_\gamma \wedge\tau)}+\supnorm{\Dr(\dpathtil (\gamma))-\Dr(\dpath )}{[\widetilde\tau_\gamma \wedge\tau,(\widetilde\tau_\gamma \vee\tau)\wedge T]}\\
\le&\,\frac\epsilon 2 +\supnorm{\Dr(\dpathtil (\gamma))-\Dr(\dpath )}{[\widetilde\tau_\gamma \wedge\tau,(\widetilde\tau_\gamma \vee\tau)\wedge T]}.
\end{align*} 
 Suppose that $\tau<\widetilde\tau_\gamma $,  but because $\dpath (\tau)= \epsilon^*+1$ and 
$ \supnorm{\dpathtil (\gamma)-\dpath }{[0,T]}\le \epsilon^*/2, $
we obtain that
\[
\dpathtil (\gamma(\tau))=\dpathtil (\gamma(\tau))-\dpath (\tau)+ \dpath (\tau)\ge    -\frac{\epsilon^*}2 +\epsilon^*+1 =\frac{\epsilon^*}2 +1,
\]
which contradicts the definition of $\widetilde\tau_\gamma $. Suppose instead that $\tau-\delta\ge\widetilde\tau_\gamma $. Then, by right continuity of $\dpathtil (\gamma)$ we have  $\dpathtil (\gamma(\widetilde\tau_\gamma ))\ge 1$, but $\dpath (\widetilde\tau_\gamma )\le 1-\epsilon_*$, so that 
\[
\dpathtil (\gamma(\widetilde\tau_\gamma ))-\dpath (\widetilde\tau_\gamma )\ge\dpathtil (\gamma(\widetilde\tau_\gamma ))+ \epsilon_*-1\ge \epsilon_*>\epsilon_*/2,
\]
which  contradicts $
\supnorm{\dpathtil (\gamma)-\dpath }{[0,T]}\le \epsilon^*/2$. Therefore the only possibility is that $\tau-\delta < \widetilde\tau_\gamma \le \tau$, and so  
\begin{align*}
\supnorm{\Dr(\dpathtil (\gamma))-\Dr(\dpath  )}{[\widetilde\tau_\gamma \wedge\tau,(\widetilde\tau_\gamma \vee\tau)\wedge T]}&=\supnorm{\Dr(\dpathtil (\gamma))-\Dr(\dpath)}{[\widetilde\tau_\gamma ,\tau)}\\
&=\supnorm{1-\dpath }{[\widetilde\tau_\gamma ,\tau)}\\
&\le\supnorm{1-\dpath}{[\tau-\delta,\tau)} \le \epsilon/2.
\end{align*} 
We omit the simpler proof for the case  $\dpath (\tau-)<1$, as a simple   contradiction argument above proves that $\tau=\widetilde\tau_\gamma $ for $\widetilde \dpath$ close enough to $\dpath$.
\end{proof}

\begin{proposition}\label{prop:SMkill-1}
Suppose $\dpath \in D([0,\infty),\mathbb R)$ is such that $f(0)>-1$ and for each $s\ge0$, the function $t\mapsto\inf\{\dpath (z):\,z\in[s,t]\}$  is continuous for $t\ge s$. 
Then $\dpath $ is a continuity point for $\Dl$. 
\end{proposition} 
\begin{proof}
The proof is similar to the one of Proposition \ref{prop:SMkill1} and is omitted.
\end{proof}

 
\subsection{Fast-forwarding  (${\rm N}$)} 
We define the left fast-forwarding map for an arbitrary barrier $a\in \mathbb R$, prove its measurability and characterise its points of continuity. The corresponding results for the right version and   fast-forwarding outside an interval are then easily derived at the end of this section.

\begin{definition}\label{def:ff} For any $a\in\mathbb R$ and all  $\dpath \in D_{a,\infty}([0,\infty),\mathbb R)$ we define the \textit{left fast-forwarding map at $a$} by $\Nl{a}:D_{a,\infty}([0,\infty),\mathbb R)\to D([0,\infty),\mathbb R)$ by $   \Nl{a}(\dpath )=\dpath (\AF^{-1}_\dpath )$ where
\[
 A^{-1}_\dpath (t)=\inf \left\{s>0:\AF_\dpath (s)>t\right\},\quad \AF_\dpath (t)=\int_0^t \mathbf1_{\{\dpath (z)> a\}}\,{\rm d}z.
\] 
\end{definition}
Observe that  $\Nl{a}$   is well-defined as $\left\{s:\AF_\dpath (s)>t\right\}\neq \emptyset$ for all $t\ge 0$, and $\AF^{-1}_\dpath $ is right continuous and non-decreasing. 
 We recall that $\AF_\dpath ^{-1}(\AF_\dpath (z))\ge z$ for all $z\in [0,\infty)$. We also recall that if $z\in [0,\infty)$ is such that 
 \begin{equation}\label{eq:increasepoint}
    \text{$\AF_\dpath (z)<\AF_\dpath (z+\epsilon)$ for all $\epsilon>0$,}\quad\text{then}\quad \AF_\dpath ^{-1}(\AF_\dpath (z))=z, 
 \end{equation} and we say that $z$ is a \textit{point of increase for} $\AF_\dpath $.

 \begin{remark}\label{rmk:ff_disc} We discuss several issues with studying the continuity of $\Nl{a}$.
 \begin{enumerate}[(i)]
 \item The map $\Nl{a}$ cannot be extended to a map from $D([0,\infty), \mathbb R_\delta)$ to itself  by setting $\Nl{a}(f)(t)=\delta$ for all $t\ge \zeta:=\lim_{t\to\infty}\AF_\dpath(t)$, where $  \mathbb R_\delta $ is the one-point compactification of $\R$. This is because some paths  $\dpath\in D([0,\infty),\mathbb R)$  spend only a finite amount of time $\zeta$ above $a$, and there is no left limit at $\zeta$ for $\Nl{a}(\dpath)$. An example is   $\dpath(t)= \sum_{n=1}^\infty (2+(-1)^n)\mathbf1_{\{n\le t< n+n^{-2}\}} $ with $a=0$, then the life-time of $\Nl{a}(\dpath)$ is $\zeta = \sum_{n=1}^\infty n^{-2}$, but   $\lim_{t\uparrow \zeta}\Nl{a}(\dpath)(t)$ does not converge (oscillation between 3 and 1). 
 A natural space  $\mathbb D$ such that $\Nl{a}:\mathbb D\to \mathbb D$ is $\mathbb D=\cup_{T\in(0,\infty]}D_\delta([0,T),\mathbb R)$, where $D_\delta([0,T),\mathbb R):=\{\dpath\in D([0,T),\mathbb R)\, \text{and}\, \dpath=\delta \text{ on }[T,\infty)\}$. But we do not know if we can extend the $J_1$ topology in a natural way for our CMT, thus we   avoid this technicality by simply working with $D_{a,\infty}([0,\infty),\mathbb R)$. 
\item The map $\Nl{a}$ is not continuous. Indeed let $a=0$, $\dpath(t)  = \mathbf1_{\{t\ge 1\}}$ and $\dpath _n(t)= \frac1n \mathbf1_{\{t< 1\}} +\mathbf1_{\{t\ge 1\}} $. Then $\dpath _n\to \dpath $ uniformly (and thus in $D([0,\infty),\mathbb R)$), but  $\Nl{0}(\dpath )=1 $ and $\Nl{0}(\dpath _n)=\dpath _n$ for all $n$ and of course $\dpath_n\not\to 1$ in $D([0,\infty),\mathbb R)$. The issue here is that $
\Leb (t\in[0,\infty):\dpath (t)=0)=1$, implying that at times when $\dpath \le 0< \dpath_n $ the fast-forwarding map produces a large \textit{delay} (note that it is easy to give a similar example for $f$ continuous). By assuming  \eqref{eq:0} in Theorem \ref{thm:ff}, we are able to make this delay arbitrarily small for all $\dpath _n$ close enough to $\dpath $ (see Step (ii) of the proof). A small delay allows a small tilting of the Skorokhod time change we construct for  $\Nl{a}(\dpath _n)$. However requiring \eqref{eq:0} is not enough for continuity. For example, take  
$$
\dpath(t)  =(t-1)\mathbf1_{\{t<  1\}}  + \mathbf1_{\{1\le t\}}\quad\text{and}\quad
\dpath_n(t)  =(t-1+1/n)\mathbf1_{\{t<  1\}}  + \mathbf1_{\{1\le t\}}.$$
Clearly $\Leb(t\in[0,\infty):f(t)=0)=0$ and  $f_n\to f$ uniformly. But  $\Nl{0}(\dpath )=1 $ and  
$\Nl{0}(\dpath_n )=t\mathbf1_{\{t<  1/n\}} +  \mathbf1_{\{1/n\le t\}},$
and so the Skorokhod distance must be bounded  below by $1-1/n$. 
We rule out this behaviour with the following  assumption 
\begin{description}
\item[(Nl)\label{A1}]if there exists an $\epsilon>0$ and $t_1<t_2<t_3<t_4$ such that \[\text{$f(t)\ge a+\epsilon$ for all $t\in [t_1,t_2)\cup [t_3,t_4)$ and $f(t)\le a$ for all $t\in [t_2,t_3)$,}\] 
\, then $\sup\{f(t):t\in [t_2,t_3)\}<a$.
\end{description}
This condition is of course satisfied by continuous processes. It is also satisfied by    one-sided processes, as the these processes  jump only in one direction. 
\item We believe that Theorem \ref{thm:ff} cannot be obtained using existing results on composition of functions. For example, the standard result  \cite[Theorem 13.2.2]{MR1876437}  does not apply to our case as it requires either the limit path $\dpath$ or its time change $\AF^{-1}_{\dpath}$ to be continuous. The most general composition theorem we know of is \cite[Theorem 1.2]{MR2479479}, but it is easy to see that assumption \cite[A2]{MR2479479} fails almost surely even for our spectrally positive process $Y$, as for any discontinuity point $t$ of $\AF_Y^{-1}$, we have $\Nl{a}(Y)(t)=Y(\AF_Y^{-1}(t))>a= Y(\AF_Y^{-1}(t-))$. (This example also rules out the possibility of working with $M$ topologies and applying \cite[Theorem 13.2.4]{MR1876437}.) Moreover, even if we wanted to prove continuity of $\Nl{a}$   at $\dpath\in C[0,\infty)$   satisfying \eqref{eq:0}, most of our proof  would still be needed to prove the joint convergence to $(\dpath, \AF^{-1}_f)$ required by the composition theorems. 
 \end{enumerate}

 \end{remark}

\begin{proposition}\label{prop:ff_meas}
 The map  $\Nl{a}$ is measurable.
\end{proposition} 

 \begin{proof} It can be found in Appendix \ref{app:additional}.
  \end{proof}

\begin{theorem}\label{thm:ff}
Assume that $f(0)>a$. Then, the map  $\Nl{a}$ is continuous at $\dpath$ if and only if $\dpath$ satisfies   \ref{A1} and
\begin{equation}\label{eq:0}
\Leb\big(t\in[0,\infty):\dpath (t)=a\big)=0.
\end{equation}
\end{theorem} 
\begin{proof} (Conventions: to simplify notation  in this proof we write $\supnorm{\cdot}{E}=\|\cdot\|_{E}$ and we drop the brackets when composing additive functionals, their inverses and Skorokhod time changes, for example by writing $\AF^{-1}_f(\gamma(\AF_g))(t)=\AF^{-1}_f\gamma \AF_g(t)$ for $f,g\in D_{a,\infty}([0,\infty),\mathbb R)$, $t\in [0,\infty)$ and $\gamma\in \Gamma$.)\\

We first prove the `if' direction. It is enough to prove sequential continuity, and so we take a sequence $\{\dpath_m\}_{m\in\mathbb N}\subset D_{a,\infty}([0,\infty),\mathbb R)$  such that $\dpath_m\to\dpath$ as $m\to\infty$.   The  proof is divided into five steps. Namely:  (i)  choose   how close $\dpath_m$ has to be to $ \dpath $; (ii)   prove a key estimate stating that the time delay due to $\dpath_m\le a< \dpath $ or $\dpath \le a< \dpath_m$ is small; (iii)   construct the time change $\stctil $ for $\Nl{a}(\dpath _m)$; (iv)   show the time change is close to the identity; (v)    show $\Nl{a}(\dpath _m)(\stctil )$ is close to $\Nl{a}(\dpath )$. \\

[\textit{Step (i)}] For an arbitrary $\epsilon\in(0,1]$ satisfying $\epsilon<f(0)-a$, let $\bar\tau_0=0$    and define recursively $(\tau_n,\bar\tau_n)\in (0,\infty]^2$ for $n\in\mathbb N$ as
 \[
\tau_n=\inf\big\{t>\bar\tau_{n-1}: \dpath (t)\le a+\epsilon/2\big\},\quad  \bar\tau_n=\inf\big\{t>\tau_{n}: \dpath (t)\ge a+\epsilon\big\}.
\]
Note that we must have $\tau_n<\bar \tau_n<\tau_{n+1}$  for each  $n\in\mathbb N$ such that $\tau_{n+1}<\infty$. We first assume that $\tau_n <\infty$  for each  $n\in\mathbb N$, which implies that   $\tau_{n}\to\infty$ as $n\to\infty$ because $\dpath$ is c\`adl\`ag (the proof in the remaining cases is similar and explained at the end of the proof). 
 Fix  $N\in\mathbb N$ large and let $T$ be a point of continuity of $\dpath $ such that $\bar\tau_{N}<T<\tau_{N+1}$. Then $T$ is a point of increase of  $\AF_\dpath $, which implies that $\AF^{-1}_\dpath  \AF_\dpath (T)=T$, and so $\AF_\dpath (T)$ is a point of continuity of $\Nl{a}(\dpath )$. Thus, recalling Section \ref{subsec:skor}, because as $T$ goes to infinity $\AF_\dpath (T)\to\infty$, it is enough to prove that $\Nl{a}(\dpath_m )\to \Nl{a}(\dpath )$ as $m\to\infty$ in $D([0,\AF_\dpath (T)],\mathbb R)$ for the restrictions of  $\Nl{a}(\dpath_m ) $ and $ \Nl{a}(\dpath )$  to $[0,\AF_\dpath (T)]$.\\
 
  We now select four positive constants less than 1, denoted by $\delta_1,\delta_2,\delta_3$ and $\delta_4$, and we set  $\delta(\epsilon,T)=\delta=\min\{ \delta_1,\delta_2,\delta_3,\delta_4\}$.\\
   
   Let $\delta_1= \delta_{1,1}\wedge  \delta_{1,2}$, where, for each of the $n_j\le N$ such that $f>a$ at some point in $[\tau_{n_j},\bar\tau_{n_j})$ we fix a $t_{n_j}\in[\tau_{n_j},\bar\tau_{n_j})$  such that $ \dpath(t_{n_j})> a$ and then we set $\delta_{1,1}=\min_j (\dpath(t_{n_j})-a )/2$, then using right continuity of $\dpath$ we can select $\delta_4(\delta_{1,1})=\delta_4\in(0,1]$ such that  for all $n_j$
\begin{equation}\label{eq:delta11}
  t_{n_j}+\delta_4<\bar\tau_{n_j}\quad\text{and}\quad  \inf_{t\in[t_{n_j},t_{n_j}+\delta_4]}f(t)> f(t_{n_j})-\delta_{1,1}
\end{equation}
(and observe it implies $\AF_f(t_{n_j}+\delta_4)-\AF_f(t_{n_j})=\delta_4$ for all $n_j$); meanwhile  for all the $n_i$ such that $f\le a$ for all points in $[\tau_{n_i},\bar\tau_{n_i})$ we use \ref{A1} to find $\delta_{1,2}\in (0,1]$ such that 
\begin{equation}\label{eq:delta12}
\max_i\sup_{t\in [\tau_{n_i},\bar\tau_{n_i})} f(t)<  a-\delta_{1,2}.
\end{equation}
Let $\delta_2\in (0,1]$ be such that $T+\delta_2$ is a continuity point of $\dpath $ and
\begin{equation}\label{eq:delta2}
\|\dpath  -\dpath (T)\|_{[T-\delta_2,T+\delta_2]}\le \epsilon\wedge \frac{\dpath (T)-a}3.
\end{equation}
By  \eqref{eq:0}, we can now let $\delta_3(\epsilon,T,\delta_2,\delta_4)=\delta_3\in (0,1] $  such that
\begin{equation}\label{eq:lebpsi}
\Leb \big(t\in [0,T]:|\dpath (t)-a|\le  \delta_3\big)\le \frac{ \epsilon \delta_2 \delta_4 }6.
\end{equation}
By \cite[Lemma 1 Section 12]{MR1700749}  there exists $N'(\epsilon,T)=N'\in\mathbb N$ and $0=t_1<t_2<...<t_{N'+1}=T$ such that 
\begin{equation}\label{eq:modulus}
\max_{n\in\{1,...,N'\}}\sup_{s,t\in[t_n,t_{n+1})}|\dpath (s)-\dpath (t)|\le \frac{\delta}6.
\end{equation}
By our definitions  $\dpath (t)-a\ge \epsilon/2 $ for all $t\in \cup_{n=0}^{N}[\bar\tau_{n},\tau_{n+1})$.
 Let $M\in \mathbb N$ such that for all $m\ge M$,  $\dpath _m$ is
\[
\text{ $\frac{\delta^2\epsilon}{6N'}$ close to $ \dpath $ in $D([0,T],\mathbb R) $ and}
\]
\begin{equation}\label{eq:T+d}
\text{ $\epsilon\wedge \frac{ \dpath (T)-a}2$ close to $ \dpath $ in $D([T,T+\delta_2],\mathbb R)$.}
\end{equation}
 Denote by $g$ any such $\dpath _m$, and   
 let $\gamma$  be a Skorokhod time change 
 on $[0,T]$ such that 
\begin{equation}\label{eq:delta}
\|\gamma-I\|_{C[0,T]}\vee\|g(\gamma)-\dpath \|_{[0,T]} \le   \frac{\delta^2\epsilon}{6N'} .
\end{equation}
Then by  \eqref{eq:modulus} and \eqref{eq:delta}
\begin{align}\nonumber
&\,\max_{n\in\{1,...,N'\}}\sup_{s,t\in[t_n,t_{n+1})}| g(\gamma(s))- g(\gamma(t))|\\ \nonumber
\le&\,  \max_{n\in\{1,...,N'\}}\sup_{s,t\in[t_n,t_{n+1})} | g(\gamma(s))-\dpath (s)|+|\dpath (s)-\dpath (t)|+|\dpath (t)- g(\gamma(t))|\\		
\le&\, \frac{\delta}2. \label{eq:moduluspsi'}
\end{align}
 
In Step (iii) we  construct a Skorokhod time change $\stctil $ on $[0,\AF_\dpath (T)]$ to show that $g(\AF^{-1}_g)$ is (up to a constant depending only on $T$) $\epsilon$-close to $\dpath (\AF^{-1}_\dpath )$ in $D([0,\AF_\dpath (T)],\mathbb R)$.

[\textit{Step (ii)}]  But first we prove the key bound on time delay
\begin{equation}\label{eq:keyestimate}
\sup_{t\in[0,T]}| \AF_g (\gamma(t))-\AF_\dpath (t)|\le \epsilon \delta_2\delta_4.
\end{equation} To prove \eqref{eq:keyestimate}, observe that for any $t\in[0,T]$, using \eqref{eq:delta} in the last inequality,
\begin{align*}
&\,\left| \AF_g (\gamma(t))-\AF_\dpath (t)\right|\\
=&\,\left|\int_{0}^{\gamma(t)}\mathbf1_{\{ g(z)> a\}}\,{\rm d}z-\int_{0}^{t}\mathbf1_{\{\dpath (z)> a\}}\,{\rm d}z\right|\\
=&\,\left|\int_{0}^{\gamma(t)}\mathbf1_{\{ g(z)> a\}}\,{\rm d}z-\int_{0}^{t}\mathbf1_{\{ g(\gamma(z))> a\}}\,{\rm d}z +\int_{0}^{t}\mathbf1_{\{ g(\gamma(z))> a\}}-\mathbf1_{\{\dpath (z)> a\}}\,{\rm d}z\right| \\
&\le\|\gamma-I\|_{[0,T]}+\Leb \big(z\in[0,T]:  g(\gamma(z))\le a<  g(z)\,{\rm or}\, g(z)\le a<  g(\gamma(z)) \big)\\
&\quad +\Leb \big( z\in[0,T]:  g(\gamma(z))\le a< \dpath (z)\,{\rm or}\,\dpath (z)\le a<  g(\gamma(z)) \big)\\
&\le  \frac{\epsilon\delta_2\delta_4 }{5} +\Leb \big(z\in[0,T]:  g(\gamma(z))\le a<  g(z)\,{\rm or}\, g(z)\le a<  g(\gamma(z)) \big)\\
&\quad +\Leb \big(z\in[0,T]: |\dpath (z)-a|\le \delta\big),
\end{align*}
and from  \eqref{eq:lebpsi} we know the last term is bounded by $  \epsilon\delta_2\delta_4/5$. 
We now bound the second term by $  \epsilon\delta_2\delta_43/5$   using \eqref{eq:moduluspsi'}. Denote by 
$$\mathcal S(z)\quad \text{the condition}\quad g(z)\le a< g(\gamma(z))\quad\text{or}\quad g(\gamma(z))\le a< g(z), $$
and observe that whenever $z,\gamma^{-1}(z)\in [t_n,t_{n+1})$, $n=1,2,...,N' $, condition $\mathcal S(z)$ implies that $|g(\gamma(z))-a|\le \delta/2$, by \eqref{eq:moduluspsi'}. Also recall that $\gamma$ and $\gamma^{-1}$ are increasing, so that $z\le\gamma(z)$ if and only if $\gamma^{-1}(z)\le z$. Then 
\begin{align*}
&\quad\,\Leb \big(z\in[0,T]:  \mathcal S(z) \big)\\
&= \sum_{n=1}^{N' } \Leb \big(z\in[t_n,t_{n+1}):  \mathcal S(z)  \big)\\
&= \sum_{n=1}^{N'}\Bigg( \Leb \big(z\in[t_n,t_{n+1}):  \mathcal S(z)  \big)\mathbf1_{\{t_n\ge\gamma(t_n),\, \gamma(t_{n+1})\ge t_{n+1}\}}\\
&\quad+ \Leb \big(z\in[t_n,t_{n+1}):  \mathcal S(z)   \big)\mathbf1_{\{t_n<\gamma(t_n),\,  \gamma(t_{n+1})\ge t_{n+1}\}}\\
&\quad+ \Leb \big(z\in[t_n,t_{n+1}):  \mathcal S(z)   \big)\mathbf1_{\{t_n\ge \gamma(t_n),\,  \gamma(t_{n+1})< t_{n+1}\}}\\
&\quad+ \Leb \big(z\in[t_n,t_{n+1}):  \mathcal S(z)   \big)\mathbf1_{\{t_n<\gamma(t_n),\,  \gamma(t_{n+1})< t_{n+1}\}}\Bigg)\\
&\le\sum_{n=1}^{N' }\Bigg( \Leb \big(z\in[t_n,t_{n+1}): | g (\gamma(z)) -a|\le\delta/2\big)\mathbf1_{\{t_n\ge\gamma(t_n),\, \gamma(t_{n+1})\ge t_{n+1}\}}\\
&\quad+ \Leb \big([ t_n,\gamma(t_n))\cup\{z\in[\gamma (t_n),t_{n+1}):  \mathcal S(z)  \} \big)\mathbf1_{\{t_n<\gamma(t_n),\,  \gamma(t_{n+1})\ge t_{n+1}\}}\\
&\quad+ \Leb \big([\gamma (t_{n+1}), t_{n+1}]\cup\{z\in[t_n,\gamma (t_{n+1})):  \mathcal S(z)   \}\big)\mathbf1_{\{t_n\ge \gamma(t_n),\,  \gamma(t_{n+1})< t_{n+1}\}}\\
&\quad+ \Leb \big([t_n,\gamma(t_n)]\cup [\gamma(t_{n+1}),t_{n+1}]\big)\mathbf1_{\{t_n<\gamma(t_n),\,  \gamma(t_{n+1})< t_{n+1}\}}\\
&\quad +\Leb\big(z\in[\gamma (t_n),\gamma (t_{n+1})):  \mathcal S(z)   \big)\mathbf1_{\{t_n<\gamma(t_n),\,  \gamma(t_{n+1})< t_{n+1}\}} \Bigg)\\
&\le\sum_{n=1}^{N' }\Bigg( \Leb \big(z\in[t_n,t_{n+1}): | g (\gamma(z)) -a|\le\delta/2\big)\mathbf1_{\{t_n\ge\gamma(t_n),\, \gamma(t_{n+1})\ge t_{n+1}\}}\\
&\quad+ \Leb \big(z\in[\gamma (t_n),t_{n+1}): |  g (\gamma(z))-a|\le \delta/2 \big)\mathbf1_{\{t_n<\gamma(t_n),\,  \gamma(t_{n+1})\ge t_{n+1}\}}\\
&\quad+ \Leb \big(z\in[t_n,\gamma (t_{n+1})): |  g (\gamma(z))-a|\le\delta/2  \big) \mathbf1_{\{t_n\ge \gamma(t_n),\,  \gamma(t_{n+1})< t_{n+1}\}}\\
&\quad+ \Leb \big(z\in[\gamma (t_n),\gamma (t_{n+1})): |  g (\gamma(z))-a|\le\delta/2   \big) \mathbf1_{\{t_n<\gamma(t_n),\,  \gamma(t_{n+1})< t_{n+1}\}}\Bigg)\\
&\quad +2N'\|\gamma-I\|_{C[0,T]}\\
&\le  \Leb \big(z\in[0,T]: |  g (\gamma(z))-a|\le\delta/2   \big) +2N'\|\gamma-I\|_{C[0,T]}
\\
&\le  \Leb \big(z\in[0,T]: |  \dpath (z)-a|\le\delta   \big) +2  \frac{\epsilon\delta_2\delta_4}{5}\\
&\le \frac{\epsilon\delta_2\delta_4}{5}+2  \frac{\epsilon\delta_2\delta_4}{5}.
\end{align*}

[\textit{Step (iii)}]
We partition $[0, \AF_\dpath (T)]$ as 
\[
\left(\bigcup_{n=1}^N \big[\AF_\dpath  (\bar\tau_{n-1}),\AF_\dpath  (\tau_{n} )\big)\cup \big[\AF_\dpath  (\tau_n ),\AF_\dpath  (\bar\tau_{n})\big)\right)\cup \big[\AF_\dpath  (\bar\tau_{N}),\AF_\dpath (T)\big],
\] and we recall that $[\AF_\dpath  (\bar\tau_{n-1}),\AF_\dpath  (\tau_{n} ))\neq\emptyset$ for all $n\le N$ as well as $[\AF_\dpath  (\bar\tau_{N}),\AF_\dpath (T)]$, meanwhile $[\AF_\dpath  (\tau_n ),\AF_\dpath  (\bar\tau_{n}))$ could be empty or not.
Note that 
  by the definition of $\bar\tau_{n-1}$ and $\tau_n$, for $ t \in [\AF_\dpath  (\bar\tau_{n-1}),\AF_\dpath  (\tau_{n} ))$,   $\bar\tau_{n-1}+s$ is a point of increase for  $\AF_\dpath$, where $ s=t-\AF_\dpath  (\bar\tau_{n-1})$,
 because $0\le s\le  \AF_\dpath  (\tau_{n} )-\AF_\dpath  (\bar\tau_{n-1})=\tau_{n} -\bar\tau_{n-1}$
 and thus
 \[
 t = \AF_\dpath  (\bar\tau_{n-1})+s= \AF_\dpath  (\bar\tau_{n-1}+s).
 \]
Then we just proved that for any $n=1,...,N$ and $ t \in [\AF_\dpath  (\bar\tau_{n-1}),\AF_\dpath  (\tau_{n} ))$
\begin{equation}\label{eq:s=t-}
\AF^{-1}_\dpath (t)
= \bar\tau_{n-1}+s,
\quad\text{where}\quad s=t-\AF_\dpath  (\bar\tau_{n-1}),
\end{equation}
which immediately implies
\begin{equation}\label{eq:psiabove}
\dpath \big(\AF^{-1}_\dpath (t)\big)=\dpath ( \bar\tau_{n-1}+s),\quad \text{for}\quad t\in \big[\AF_\dpath  (\bar\tau_{n-1}),\AF_\dpath  (\tau_{n} )\big).
\end{equation} 
Also for each $n\in\mathbb N$ such that $[\AF_\dpath (\tau_n ),\AF_\dpath (\bar\tau_{n}))\neq \emptyset$, 
\begin{equation}\label{eq:psibelow}
a\le\dpath \big(\AF^{-1}_\dpath (t)\big)\le a+\epsilon,\quad t\in \big[\AF_\dpath (\tau_n ),\AF_\dpath (\bar\tau_{n})\big).
\end{equation}

We now construct a Skorokhod time change $\stctil $ on $[0,\AF_\dpath  (T)]$ so that   $\Nl{a}(g)(\stctil )=g( \AF^{-1}_g \stctil (t))=g(\gamma(  \bar\tau_{n-1}+s))$   for $t\in [\AF_\dpath  (\bar\tau_{n-1}),\AF_\dpath  (\tau_{n}))$, to \textit{synchronise}  $\Nl{a}(g)$ with  \eqref{eq:psiabove} and apply \eqref{eq:delta}; meanwhile we \textit{linearly rescale} the time change on the remaining intervals exploiting  bounds for  $\Nl{a}(g )(\stctil )$  similar to \eqref{eq:psibelow}      (due to   $g( \gamma)\le a+ 2\epsilon$  on $[ \tau_n , \bar\tau_{n})$). This linear rescaling is small thanks to the key estimate \eqref{eq:keyestimate}.   For the very last interval we need an extra rescaling to get $\stctil \AF_\dpath (T)=\AF_\dpath (T)$, and there we use the $\delta_2$ ``control'' from \eqref{eq:delta2} and \eqref{eq:T+d}. \\
For $n=1,2,...,N$ we define 
\begin{equation*}
\stctil (t)=\left\{\begin{split}
& \AF_g \gamma \AF^{-1}_\dpath (t),&t\in \big[\AF_\dpath  (\bar\tau_{n-1}),\AF_\dpath  (\tau_{n} )\big),\\
&c_n\big(t-\AF_\dpath  (\tau_{n} )\big)+ \AF_g  \gamma(\tau_{n} ),&t\in \big[\AF_\dpath  (\tau_{n} ),\AF_\dpath  (\bar\tau_{n})\big),\\
& \AF_g  \gamma \AF^{-1}_\dpath (t),&t\in \big[\AF_\dpath  (\bar\tau_{N}),\AF_\dpath  (T-\delta_2)\big),\\
&  c_{N+1}\big(t-\AF_\dpath  (T-\delta_2)\big)+ \AF_g  \gamma(T-\delta_2),&t\in \big[\AF_\dpath  (T-\delta_2 ),\AF_\dpath  (T)\big],
\end{split}\right.
\end{equation*}
where  
\[
c_n=\frac{ \AF_g  \gamma(\bar\tau_n)- \AF_g  \gamma(\tau_{n} )}{\AF_\dpath  (\bar\tau_{n})-\AF_\dpath  (\tau_{n} )},
\]
when $[\AF_\dpath  (\tau_{n} ),\AF_\dpath  (\bar\tau_{n}))\neq\emptyset$, and  
\[
c_{N+1}=\frac{ \AF_\dpath  (T)-  \AF_g  \gamma(T-\delta_2)}{\AF_\dpath  (T)-\AF_\dpath  (T-\delta_2 ) }.
\] 
It is immediate to verify that $\stctil (0)=\gamma(0)=0$ (recalling that $\AF_f^{-1} (0)=0$  as $f(0)>a$) and  $\stctil  (\AF_\dpath (T))= \AF_\dpath (T)$. Moreover,  $\stctil $  is increasing and continuous as:
\begin{itemize}
  \item $ \AF_\dpath^{-1} $ is increasing and continuous on $[\AF_\dpath (\bar\tau_{n-1}),\AF_\dpath  (\tau_{n} ))$, $n=1,...,N+1$;
  \item $\gamma $ is increasing and continuous;
  
 \item    $ \AF_g $ is increasing and continuous on $[\gamma(\bar\tau_{n-1}),\gamma(\tau_{n}))$ by \eqref{eq:delta} and the definition of $\bar\tau_{n-1}$ and $\tau_{n}$, $n=1,...,N+1$;
 \item if $[\AF_\dpath (\tau_{n} ),\AF_\dpath  (\bar\tau_{n}))\neq\emptyset$    we are just linearly interpolating with a positive slope because by \eqref{eq:delta11}
\begin{equation}\label{eq:gedelta_1}
\AF_\dpath  (\bar\tau_{n})-\AF_\dpath  (\tau_{n} )>\delta_4 \quad \text{for all}	\quad n=1,...,N,
\end{equation}
and by the $ \delta$-bound in \eqref{eq:delta} and \eqref{eq:delta11}  
\[
\AF_g  \gamma(\bar\tau_n)- \AF_g  \gamma(\tau_{n} )> \delta_4  \quad \text{for all}	\quad n=1,...,N,
\] 
and similar observations yield $c_{N+1}>0$;
\item if $[\AF_\dpath (\tau_{n} ),\AF_\dpath  (\bar\tau_{n}))=\emptyset$ we only need to check that  $\AF_g \gamma \AF^{-1}_\dpath \AF_f(\bar \tau_{n})=\AF_g \gamma (\bar \tau_{n})$ equals  $$\lim_{t\uparrow \AF_f(\tau_{n})}\AF_g \gamma \AF^{-1}_\dpath (t)=\lim_{t\uparrow  \tau_{n}}\AF_g \gamma \AF^{-1}_\dpath \AF_f(t)=\AF_g \gamma (\tau_{n}),$$
where in the identities above we used that $\bar\tau_n$ and   every point in $[\bar\tau_{n-1},\tau_n)$  are points of increase for $\AF_\dpath$,  and in the last equality we used continuity of both $\AF_g$ and $\gamma$. To check this, observe that   \eqref{eq:delta12} holds, and so, by the $\delta$-bound in \eqref{eq:delta},   $g(\gamma(t))\le a$ for all $t\in [\tau_{n} ,\bar\tau_{n})$ and thus  $\AF_g \gamma(\tau_{n} )=\AF_g \gamma (\bar\tau_{n})$.
  \end{itemize}   So we proved that $\stctil $ is a Skorokhod time change on $[0,\AF_\dpath (T)]$.\\
  
[\textit{Step (iv)}]  We now verify that $\|\stctil -I\|_{C[0,\AF_\dpath (T)]}\le C \epsilon$, where $C>0$ depends only on $T$. First note that  by \eqref{eq:keyestimate} and \eqref{eq:gedelta_1}
\begin{equation}\label{eq:cn}
\begin{split}
&|c_n-1|\le \frac{2\delta_4\epsilon}{  \delta_4 }=2\epsilon,\\
&|c_{N+1}-1|\le \frac{|\AF_g  \gamma(T-\delta_2 )-\AF_\dpath ( T-\delta_2 )|}{\delta_2  }\le \epsilon.
\end{split}
\end{equation} 
If $t\in [\AF_\dpath  (\tau_{n}),\AF_\dpath  (\bar\tau_{n}))\neq\emptyset$, $n=1,...,N$,  then, using \eqref{eq:cn} and \eqref{eq:keyestimate}, 
\begin{align*}
|\stctil (t)-t|&=|c_n\left(t-\AF_\dpath  (\tau_{n})\right)+ \AF_g \gamma(\tau_{n})-t|\\
&\le (t-\AF_\dpath  (\tau_{n}))|c_n-1|+| \AF_g \gamma(\tau_{n})-\AF_\dpath  (\tau_{n})|\\
&\le \AF_\dpath (T)2\epsilon +\epsilon ,
\end{align*}
and similarly for $t\in [\AF_\dpath  (T-\delta_2),\AF_\dpath  (T)]$. Meanwhile, if $t\in[\AF_\dpath  (\bar\tau_{n-1}),\AF_\dpath  (\tau_{n}))$, $n=1,...,N$,  then, using \eqref{eq:s=t-} and \eqref{eq:keyestimate}, 
\begin{align*}
|\stctil (t)-t|
&= |\AF_g \gamma(\tau_{n-1}+s )-\AF_\dpath (\tau_{n-1}+s)|\le  \epsilon,
\end{align*}
and similarly for    $t\in[\AF_\dpath  (\bar\tau_{N}),\AF_\dpath  (T-\delta_2))$.\\

[\textit{Step (v)}]  It remains to prove that $\|g( \AF^{-1}_g \stctil )-\dpath ( \AF^{-1}_\dpath )\|_{[0,\AF_\dpath (T)]}\le C \epsilon$ for $C>0$ only dependent on $T$. For  $t\in  [\AF_\dpath  (\tau_{n} ),\AF_\dpath  (\bar\tau_{n}))\neq\emptyset$, $n=1,...,N$, the increasing function $
 \AF^{-1}_g\stctil (t)=\AF^{-1}_g\big(c_n\left(t-\AF_\dpath  (\tau_{n} )\right)+ \AF_g \gamma(\tau_{n} )\big)$
 is such that $$    \gamma(\tau_n)\le  \AF^{-1}_g\AF_g \gamma(\tau_n)\le \AF^{-1}_g\stctil (t)< \AF^{-1}_g\AF_g \gamma(\bar\tau_n)= \gamma(\bar\tau_n),$$
 so that by \eqref{eq:delta} and the definition of $\tau_n$ and $\bar\tau_n$
\begin{align*}
a\le g( \AF^{-1}_g\stctil (t))
& \le \sup_{z\in [\gamma(\tau_{n}),\gamma(\bar\tau_n) )}g(z)\le a+2\epsilon.
\end{align*}
Thus,  for   $n=1,...,N$, using \eqref{eq:psibelow},
\begin{align*}
&\,\|g(  \AF^{-1}_g\stctil )-\dpath ( \AF^{-1}_\dpath )\|_{ [\AF_\dpath  (\tau_{n} ),\AF_\dpath  (\bar\tau_{n}))}\\
\le &\, \|g( \AF^{-1}_g\stctil )-a\|_{ [\AF_\dpath  (\tau_{n} ),\AF_\dpath  (\bar\tau_{n}))}+\|a-\dpath ( \AF^{-1}_\dpath )\|_{ [\AF_\dpath  (\tau_{n} ),\AF_\dpath  (\bar\tau_{n}))}\le 3\epsilon.
\end{align*} 
For $t\in [\AF_\dpath  (\bar\tau_{n-1}),\AF_\dpath  (\tau_{n} ))$, $n=1,...,N$,  using    \eqref{eq:s=t-},
\[
g\big( \AF^{-1}_g\stctil (t)\big)=g\big( \AF^{-1}_g \AF_g \gamma \AF^{-1}_\dpath  \AF_\dpath (\bar\tau_{n-1}+s)\big) =g\big(\AF^{-1}_g \AF_g \gamma(\bar\tau_{n-1}+s)\big),
\]  and by   the $\epsilon/4$-bound from \eqref{eq:delta} we have that $\gamma(\bar\tau_{n-1}+s)\in [\gamma(\bar\tau_{n-1}),\gamma(\tau_{n}))$ is a point of increase for $\AF_g $,  where we used 
\[
\gamma\big(\bar\tau_{n-1}+\AF_\dpath (\tau_n )-\AF_\dpath (\bar\tau_{n-1})\big)=
\gamma(\bar\tau_{n-1}+\tau_n -\bar\tau_{n-1} ).
\]
So we proved that $
g( \AF^{-1}_g\stctil (t))  =g(\gamma(\bar\tau_{n-1}+s)),
$
  and by \eqref{eq:delta} and \eqref{eq:psiabove}
\begin{align*}
 \|g(  \AF^{-1}_g\stctil )-\dpath ( \AF^{-1}_\dpath )\|_{ [\AF_\dpath  (\bar\tau_{n-1}),\AF_\dpath  (\tau_{n} ))}\le \epsilon. 
\end{align*} 
The same argument holds for $ [\AF_\dpath  (\bar\tau_{N}),\AF_\dpath  (T-\delta_2))$. Finally, for  $t\in [\AF_\dpath  (T-\delta_2),\AF_\dpath  (T)]$,
\begin{align*}
\big|g\big(  \AF^{-1}_g \stctil (t)\big)-\dpath \big( \AF^{-1}_\dpath  (t)\big)\big|
&\le   \big|g\big(  \AF^{-1}_g \stctil (t)\big)-\dpath ( T)\big|+|\dpath (T)-\dpath  (T-\delta_2+s)|,
\end{align*}
 for $s= t-\AF_f(T-\delta_2)$ so that the second term is bounded by $\epsilon$ by \eqref{eq:delta2}, and thus it only remains to bound the first term uniformly for  $t\in [\AF_\dpath  (T-\delta_2),\AF_\dpath  (T)]$. 
Observe that 
\[
\AF^{-1}_g\stctil \AF_\dpath (T-\delta_2)=\AF^{-1}_g\AF_g  \gamma(T-\delta_2 )=\gamma(T-\delta_2) ,
\]
and 
\[
\AF^{-1}_g\stctil  \AF_\dpath (T)=\AF^{-1}_g \AF_\dpath (T)=\AF^{-1}_g(\AF_g  (T)+ \Delta)=\AF^{-1}_g\AF_g  (T+\Delta)=T+\Delta,
\]
where $|\Delta|=|\AF_\dpath (T)-\AF_g (T)|\le   \delta_2$ by \eqref{eq:keyestimate} and we used that all points in $[\gamma(T-\delta),T]$ are points of increase for  $\AF_g$, by the $\epsilon/4$-bound  in \eqref{eq:delta},   as well as all the points   $t\in[T,T+\delta_2]$, by
\begin{align*}
g(t)-a & =g(\overline\gamma(t))-\dpath (t) +(\dpath (T)-a)+(\dpath (t)-\dpath (T))\ge(\dpath (T)-a)/3 >0,
\end{align*}
using   \eqref{eq:delta2} and \eqref{eq:T+d}, with $\overline\gamma$ being any Skorokhod time change satisfying \eqref{eq:T+d}. So it remains to show that $g$ is close to $\dpath (T)$ on $[\gamma(T-\delta_2),T+\delta_2$].
Using \eqref{eq:delta2} and \eqref{eq:T+d} 
\begin{align*}
\|g- \dpath (T)\|_{[T,T+\delta_2]} &=\|g(\overline\gamma)- \dpath (T)\|_{[T,T+\delta_2]}\\
&\le \|g(\overline\gamma)- \dpath \|_{[T,T+\delta_2]}+\|\dpath - \dpath (T)\|_{[T,T+\delta_2]}
\le 2\epsilon .
\end{align*}
and using \eqref{eq:delta} and \eqref{eq:T+d}  
\begin{align*}
\|g- \dpath (T)\|_{[\gamma(T-\delta_2),T]} &= \|g(\gamma)- \dpath (T)\|_{[T-\delta_2,T]}\\
&\le \|g(\gamma)- \dpath \|_{[T-\delta_2,T]}+\|\dpath - \dpath (T)\|_{[T-\delta_2,T]}
\le 2\epsilon .
\end{align*}
 
 To complete the proof of the `if' direction, it remains to consider the case where $\tau_n =\infty$  for some  $n\in\mathbb N$. Let $m$ denote the smallest integer such that $\tau_n =\infty$. If $\bar\tau_{m-1}<\tau_m=\infty$ the same proof above applies easily as the limit process path $\dpath$ is eventually bounded below by $\epsilon/2>0$. Otherwise, $\tau_{m-2}<\bar\tau_{m-1}=\infty$, and then the same proof can be applied because for all $M\in\mathbb N$ there exists a continuity point $T>M$ of $f$  such that $f(T)>a$, as $\dpath$ spends an infinite amount of time above $a$. Then the last interval $[\tau_{m-2},T]$ is easy to treat as $\|\Nl{a}(f)-a\|_{[\AF_f(\tau_{m-2}),\AF_f(T)]}\le \epsilon$. And so we proved the `if' direction.\\
 
 The `only if' direction is simple and we only sketch the proof. If \ref{A1} fails, then let $\dpath_n$ equal to $f$ on $\R\backslash[t_2,t_3)$ and equal to  $\dpath+1/n$ on  $[t_2,t_3)$, then $f_n\to f$ uniformly but for any $T$ large and any Skorokhod time change $\gamma$, either $\|\Nl{a}(\dpath_n)(\gamma)-\Nl{a}(\dpath)\|_{[0,T]}\ge \epsilon/2$ for all $n$ large or $\|\gamma-I\|_{C[0,T]}\ge (t_2-t_1)\wedge(t_4-t_3)>0$. Alternatively, if $\Leb(t:\,\dpath(t)=a)>0$, we can find       $t_1<t_2\le t_3<t_4\in[0,\infty)$ such that $\Leb(t\in[t_1,t_2):\,\dpath(t)=a)=t_2-t_1$, $\dpath\le a$ on (the possibly empty interval) $[t_2, t_3)$ and  $\dpath> a$ on $(t_3,t_4)$. Define $f_n$ equal to $f$ everywhere but from $[t_1,t_2)$ where we let $f_n=f+1/n$ and a similar argument as above concludes.  
 
 \end{proof}

We   immediately have the version of Theorem \ref{thm:ff} for a right barrier $b\in\mathbb R$   and a two sided barrier.
\begin{definition}\label{def:ff_b_twosd}
Let $a,b\in\mathbb R$ with $a<b$. Define the \textit{right fast-forwarding map at $b\in\R$}, $\Nr{b}:  D_{-\infty ,b}([0,\infty),\mathbb R)\to   D ([0,\infty),\mathbb R)$,  by $\Nl{b}(f)=f(\AF^{-1}_f)$ where $ A^{-1}_\dpath (t)=\inf \left\{s>0:\AF_\dpath (s)>t\right\}$ and $\AF_\dpath (t)=\int_0^t \mathbf1_{\{\dpath (z)< b\}}\,{\rm d}z$; and define the \textit{two-sided fast-forwarding map at $a,b$ } as \[\Nl{a}(\Nr{b}):  D_{a,b}([0,\infty),\mathbb R)\to   D ([0,\infty),\mathbb R).\] 
\end{definition}
 Note that   $\Nr{b}$ is clearly measurable by using the same argument as in Proposition \ref{prop:ff_meas} and so is $\Nl{a}(\Nr{b})$.

\begin{corollary}\label{cor:ff_b}
Assume that $f(0)<b\in\R$. Then, the map  $\Nr{b}$ is continuous at $\dpath$ if and only if   $\dpath  $   satisfies   $\Leb(t\in[0,\infty):\dpath (t)=b)=0$ and
 \begin{description}
\item[(Nr)\label{A2}]:  if there exists an $\epsilon>0$ and $t_1<t_2<t_3<t_4$ such that $f(t)\le b-\epsilon$ for all $t\in [t_1,t_2)\cup [t_3,t_4)$ and $f(t)\ge b$ for all $t\in [t_2,t_3)$, then $\inf_{t\in [t_2,t_3)}f(t)>b$.
\end{description}
\end{corollary} 

\begin{corollary}\label{cor:ff_twosid}
Assume that $f(0)\in(a,b)$. Then, the map $\Nl{a}(\Nr{b})$ is continuous at $\dpath$ if and only if   $\dpath  $ satisfies    \ref{A1}, \ref{A2} and $\Leb(t\in[0,\infty):\dpath (t)\in\{a,b\})=0$.
\end{corollary} 

\begin{remark}\label{rmk:convffmaps}
If $a=-1$ and $b=1$ we write $\Nl{a}=\Nlone$ and $\Nr{b}=\Nrone$.
\end{remark}

We conclude this section with a simple statement showing that $\Nl{a}$ and $\Nr{b}$ commute and that their composition equals the natural definition of deletion  the time spent outside $(a,b)$. The proof can be found in Appendix \ref{app:additional}.

 \begin{proposition}\label{prop:FF_repr_ab} 
 For any $a <   b $,
$\Nl{a}(\Nr{b})=\Nr{b}(\Nl{a}) $  on $D_{a,b}([0,\infty),\mathbb R)$ 
  and  $\Nl{a}(\Nr{b})$ can be equivalently defined as $f\mapsto f((\AF_{(a,b)}^f)^{-1})$ where $\AF_{(a,b)}^f(t)=\int_0^t\indi{f(z)\in(a,b)}\,\dd z$. 
\end{proposition}


\section{One-sided and Gr\"unwald type processes}\label{sec:YandYh}
This section studies the pathwise construction of our Markov processes on $[-1,1]$, obtained by applying the reflecting, killing and fast-forwarding maps introduced in Section \ref{sec:skor_maps}. In Section \ref{sec:YLR} we study the restrictions of our spectrally positive process $Y$, meanwhile in Section \ref{sec:YLRh} we study the restrictions of our discrete-valued Gr\"unwald type approximation to $Y$. These last discrete-valued processes allow us to gain crucial insights on the boundary behaviour of the limit processes, along with providing a tool to   derive the boundary conditions identified in Theorem \ref{thm:main_II_intro}.
\subsection{One-sided processes and their boundary modifications}\label{sec:YLR}
We denote by $Y=(Y_t)_{t\ge 0}$ the c\`{a}dl\`{a}g modification  of the spectrally positive L\'evy process with Laplace exponent $\sym$, i.e. $\exp\{\sym(\xi)\}=\mathbb E_0[\exp\{ -\xi Y_1\}]$ for $\Re \xi\ge 0$ \cite[Page 188]{MR1406564}, with $\sym$   defined in \ref{H0}.
\begin{description}
\item[(H0)\label{H0}] For a nonnegative Borel measure $\phi$  on $(0,\infty)$ such that
\[
\int_{(0,\infty)} (y^2\wedge y)\,\phi(\dd y)<\infty\quad\text{and}\quad \int_{(0,1)} y\,\phi(\dd y)=\infty,
\]
  we define $\sym(\xi)=\int_{(0,\infty)}\left(e^{-\xi y}-1+\xi y\right)	\phi(\dd y),$ for $\Re \xi\ge0$.
\end{description}
Note that    $\sym(\xi) =\xi^2\int_0^\infty e^{-\xi x}\Phi(x)\,\dd x$ where  $\Phi(x)= \int_x^\infty \phi((y,\infty))\,\dd y$,   
and by  Dominated Convergence Theorem, as $|\xi|\downarrow 0$
\begin{equation}
\begin{split} 
\sym(\xi)\to 0 \quad\text{and}\quad  \sym'(\xi) \to 0.
\end{split}
\label{eq:convhomega}
\end{equation} 
 
\begin{remark}\label{rmk:Yproperties} We collect some basic facts about the process $Y$. For examples we refer to \cite[Example 2.4]{BKT20}. \begin{enumerate}[(i)]
\item The process $Y$ is a Feller process and its Feller semigroup on $C_0(\mathbb R)$ is generated by the closure of $(\Mr ,C_c^\infty(\mathbb R)) $ (see proof of Proposition \ref{prop:gruntolevy}), where 
\[
\Mr g(x):=\frac{\dd ^2}{\dd x^2}\int_0^\infty \Phi(y)g(x+y)\,\dd y.
\]
   
\item  Under assumption \ref{H0}, the Laplace exponent $\psi$ characterises spectrally positive, recurrent L\'evy processes with paths of unbounded variation and no diffusion component \cite[Remark 2.3.i]{BKT20}. In particular,  $\mathbb P_{Y_0}$-a.s., $\tau_{(-1,1)}^Y =\tau_{[-1,1]}^Y$,  $\tau_{(-1,\infty)}^Y =\tau_{[-1,\infty)}^Y$ and $\tau_{(-\infty,1)}^Y =\tau^Y_{(-\infty,1]}$, all immediate consequences of  \eqref{prop:exit_open_closed}. 
\item By \cite[Theorem 8.1.iii]{MR2250061} we know that  for all $\LTp \ge0$ 
\[
x\mapsto\mathbb E_x\left[e^{-\LTp \tau^Y_{(-\infty,1]}}\mathbf 1_{\{ \tau^Y_{(-\infty,1]}<\tau_{[-1,\infty)}^Y\}}  \right] \in C_0(-1,1]\]  
  and it equals 1 for $x=1$, and
\[
 x\mapsto \mathbb E_x\left[e^{-\LTp\tau_{[-1,\infty)}^Y}\mathbf 1_{\{ \tau^Y_{(-\infty,1]}>\tau_{[-1,\infty)}^Y\}} \right]  \in C_0[-1,1)
\] 
and it equals 1 for $x=-1$. Moreover,  for any $\LTp>0$ and $g\in C[-1,1]$,  
$$
x\mapsto R_\LTp^{{\rm DD}} g(x):=\int_0^\infty e^{-\LTp  t}\mathbb E_x\big[g(Y_t)\,\mathbf1_{\{t<\tau^Y_{(-1,1)}\}}\big]\,\dd t \in C_0(-1,1),
$$ which is a consequence of \cite[Theorem 8.7]{MR2250061}. 
\end{enumerate}

\end{remark}

\begin{table}[h]
\centering
\vline
\begin{tabular}{l|c|c}
  \hline
\hspace{1cm} Definition of $\Y{LR}$  
\hspace{.1cm} &   Equals a.s. &  Equals in law \\
	\hline
	1.  $Y^{{\rm DD}}=  \Dl (\Dr(Y)) $ & $\Dr(\Dl(Y))$   &  \\
  \hline 
2.	$  Y^{{\rm DN}}= \Dl  (\Nrone(Y)) $ & $ \Nrone (\Dl (Y))$ &  $ \Dl(\Narone(Y)) $  \\
  \hline
3.	 $  Y^{{\rm ND}}= \Dr(\Nlone(Y)) $  & $\Nlone(\Dr (Y))$ & \\
  \hline
4.	$  Y^{{\rm NN}}= \Nrone(\Nlone(Y))$ &  $\Nlone(\Nrone( (Y))$ &  $ \Nlone(\Narone(Y)) $ \\
  \hline
5.	$   Y^{{\rm N^*D}}= \Dr(\Nalone (Y))  $  \\
  \hline
6.	 $Y^{{\rm N^*N}}= \Nrone(\Nalone (Y))$ & & $ [\Nalone\Narone](Y)$  \\
  \hline 
\end{tabular}\vline
\caption{\label{tab:YLR}The definition of each process $\Y{LR}$, recalling that $Y$ is defined in Section \ref{sec:YLR} and the maps are defined in Section \ref{sec:skor_maps}.}
\end{table}  

To conclude the argument of our main result, Theorem \ref{thm:main_II_intro}, we need the resolvent operator of each $\Y{LR}$ to map $C_0(\Omega)$  to itself,  hence we prove this below.

\begin{proposition}\label{prop:cont_semigroup}
The c\`adl\`ag processes $\Y{LR}$ in Table \ref{tab:YLR} are strong Markov and the respective $\LTp$-resolvent operators  ($\LTp>0$) map $ C_0(\Omega)$ to itself.
\end{proposition}
\begin{proof} Recall that $Y$ is a Feller process and so is $\Nalone (Y)$ by \cite[Proposition VI.1]{MR1406564}. Then by Theorem \cite[Theorem 3.14, Chapter IV, page 104]{MR0193671} they are standard (or  Hunt) processes (see \cite[Definition 3.23, Chapter IV, page 104]{MR0193671}). Then, 
by \cite[Exercise V.2.11, page 212]{MR0264757} the processes $\Nlone(\Nrone(Y))$, $\Nlone    (Y)$, $\Nrone (Y)$ and $\Nrone(\Nalone (Y))$ 
   are strong Markov c\`adl\`ag processes and (the completion of) their natural filtration is right continuous. 
   We remark that for $\Nlone(\Nrone(Y))$ we used the ``two-sided fast-forwarding'' representation from Proposition \ref{prop:FF_repr_ab}. The remaining processes are obtained by killing the 
   processes $Y$, $\Nlone    (Y)$, $\Nrone (Y)$ and $\Nalone (Y)$ at the stopping time defined as their first exit from the open (in the respective state spaces) sets  $(-1,1)$, $[-1,1)$, $(-1,1]$ and $[-1,1)$, respectively. Then the resulting c\`ad\`ag processes are  strong Markov with right continuous filtration by  \cite[Theorem 10.1 and   Remarks 10.2.(1,3,4), Chapter X, page 301-302]{MR0193671}).\\  
We now simplify notation by writing $Z=\Y{LR}$.
 Then $\tau^Z_{(-1,1)}$ is a stopping time for each process $Z$ \cite[Theorem 1.6, Chapter 2]{MR838085} and  $\tau^Z_{(-1,1)}=\tau^Y_{(-1,1)}=\tau^{Y}_{[-1,1]}$ $\mathbb  P_x$-a.s. for any $x\in(-1,1)$ by Remark \ref{rmk:Yproperties}-(ii). Now we claim that for any $x\in(-1,1)$ it holds  $\mathbb P_x$-a.s. that 
 \begin{equation}\label{eq:exitZ}Z_{\tau_{(-1,1)}^Z}= 1\,\,\text{on}\,\, \big\{\tau_{[-1,1)}^Z<\tau_{(-1,1]}^Z\big\} \quad \text{and}\quad Z_{\tau_{(-1,1)}^Z}= -1\,\,\text{on}\,\, \big\{\tau_{[-1,1)}^Z>\tau_{(-1,1]}^Z\big\}
\end{equation}
for all boundary conditions (recalling that for killing we conventionally absorb at 1 or $-1$ in the natural way).   For any process and boundary that involves killing  (${\rm D}$) and reflecting  (${\rm N^*}$)  the claim is clear.   For the cases ${\rm  NR}$     on $\{\tau_{[-1,1)}^Z>\tau_{(-1,1]}^Z\}=\{\tau_{(-\infty,1)}^Y>\tau_{(-1,\infty)}^Y\}$ we have
\[
Z_{\tau_{(-1,1]}^Z}=   Y_{\AF^{-1}_Y(\tau_{(-1,\infty)}^Y)}  =-1,
\]
because $\tau_{(-1,\infty)}^Y=\AF_Y(\tau_{(-1,\infty)}^Y)$,  $\tau_{(-1,\infty)}^Y$ is a point of increase for $\AF_Y$  (by the strong Markov property and $-1$ being regular for $(-1,\infty)$) and $Y_{\tau_{(-1,\infty)}^Y}=-1$ because $Y$ is spectrally positive. For the cases ${\rm  LN}$,  on $\{\tau_{[-1,1)}^Z<\tau_{(-1,1]}^Z\}=\{\tau_{(-\infty,1)}^Y<\tau_{(-1,\infty)}^Y\}$ we have
\[
Z_{\tau_{[-1,1)}^Z} = Y_{\AF^{-1}_Y(\tau_{(-\infty,1)}^Y)}  =1, 
\]
as $\AF^{-1}_Y(\tau_{(-\infty,1)}^Y)$ is the time of first re-entry in $(-\infty,1)$ which must be at $1$ because $Y$ is spectrally positive. 
We now can show that the resolvent maps $C_0(\Omega)$ to itself, i.e. $R_\LTp^{{\rm LR}}C_0(\Omega)\subset C_0(\Omega)$ where \[
R_\LTp^{{\rm LR}}g(x)=\mathbb E_x\left[ \int_0^\infty e^{-\LTp t} g(Z_t)
\,\dd t \right], \quad \LTp>0,\,x\in [-1,1],\,g\in C_0(\Omega).
\]

Indeed by the strong Markov property and \eqref{eq:exitZ}
\begin{align*}
R_\LTp^{{\rm LR}}g(x)&=\mathbb E_x\left[ \int_0^{\tau_{(-1,1)}^Z} e^{-\LTp t} g(Z_t)
\,\dd t \right]+\mathbb E_x\left[ \int_{\tau_{(-1,1)}^Z}^\infty e^{-\LTp t} g(Z_t)
\,\dd t \right]\\
   &=R_\LTp^{{\rm DD}} g(x)+ \mathbb E_x\left[  e^{-\LTp  \tau_{(-1,1)}^Z } R_\LTp^{{\rm LR}}g(Z_{\tau_{(-1,1)}^Z}) \right]\\
      &=R_\LTp^{{\rm DD}} g(x)+R_\LTp^{{\rm LR}}g(1) \mathbb E_x\left[  e^{-\LTp  \tau_{[-1,1)}^Z }  \mathbf1_{\{\tau_{[-1,1)}^Z<\tau_{(-1,1]}^Z\}} \right]\\
      &\quad+R_\LTp^{{\rm LR}} g(-1)\mathbb E_x\left[  e^{-\LTp  \tau_{(-1,1]}^Z }  \mathbf1_{\{\tau_{[-1,1)}^Z>\tau_{(-1,1]}^Z\}} \right]\\
                &=R_\LTp^{{\rm DD}} g(x)+  R_\LTp^{{\rm LR}} g(1)\mathbb E_x\left[  e^{-\LTp  \tau^{Y}_{(-\infty,1)} }  \mathbf1_{\{\tau^{Y}_{(-\infty,1)}<\tau^{Y}_{(-1,\infty)}\}} \right]\\
                &\quad+   R_\LTp^{{\rm LR}} g(-1)   \mathbb E_x\left[  e^{-\LTp  \tau^{Y}_{(-1,\infty)} }   \mathbf1_{\{\tau^{Y}_{(-\infty,1)}>\tau^{Y}_{(-1,\infty)}\}} \right],
\end{align*}
which defines a continuous function on $C_0(\Omega)$ by Remark \ref{rmk:Yproperties}-(iii). 

\end{proof}
 
\begin{remark}\label{rmk:cont_semigroup}
By Proposition \ref{prop:cont_semigroup}, we can use the Markov property and right continuity of the paths of $\Y{LR}$ to conclude that  $t\mapsto \mathbb E_x[g(\Y{LR}_t)]$ is right continuous for each $x\in[-1,1]$ and $g\in C_0(\Omega)$.
\end{remark}

The following proposition will allow us to apply the CMT in Section \ref{sec:skor}.

 
\begin{proposition}\label{prop:Ypath_continuity_point} Assume \ref{H0} and let $x\in(-1,1)$. Then, $\mathbb P_x$-a.s. the paths of $X$ satisfy the conditions on $\dpath$ in the statement $S$, where $(X,S)$ can be:
\begin{flalign*}
&(Y,\,\text{Proposition \ref{prop:SMkill1}}), &&(\Dr(Y),\,\text{Proposition \ref{prop:SMkill-1}}), \\ & (\Nrone(Y),\,\text{Proposition \ref{prop:SMkill-1}}), && (\Nlone(Y),\,\text{Proposition \ref{prop:SMkill1}}),\\
&(\Nalone(Y),\,\text{Proposition \ref{prop:SMkill1}}), &&(\Nalone(Y),\,\text{Corollary \ref{cor:ff_b}  with }b=1),\\
&(Y,\,\text{Corollary \ref{cor:ff_b} with }b=1), &&(Y,\,\text{Corollary \ref{cor:ff_twosid} with }\{a,b\}=\{-1,1\}).
\end{flalign*}
\end{proposition}
\begin{proof} Let $x<1$ be the starting point of $Y$ and simplify notation by writing $\tau=\tau^Y_{(-\infty,1)}$. Concerning $(Y,\,\text{Proposition \ref{prop:SMkill1}})$, combining  \eqref{prop:exit_open_closed} with the fact that $Y$ cannot creep upward \cite[Section 8.1]{MR2250061}, we obtain that $\mathbb P_x$-a.s.
 $Y_{\tau}>1$. Now suppose there exists a $\delta>0$ such that $\sup_{t\in[0,\tau-\delta)}Y(t)=1$ with positive probability, then there exists $\tau'<\tau$ such that $\dpath (\tau'-)=1$ and $\dpath (\tau')<1$, but this implies that $Y$ has a negative jump, which is a contradiction. The claim from the other two cases concerning Proposition \ref{prop:SMkill1} follow similarly.  Concerning Proposition \ref{prop:SMkill-1} the result follows as $Y$ moves downward continuously  and this property is preserved by $\Dr$ and $\Nrone$. The remaining cases follow by Remark \ref{rmk:recur}, the fact that the occupation measure of $Y$ allows a density ($Y$ moves downward by a drift and 0 is regular for $(-\infty,0)$, so  its range has positive Lebesgue measure on any compact time interval \cite[Section 0]{MR958195}) which implies \eqref{eq:0}, and that  \ref{A1} and \ref{A2} are always satisfied by one-sided processes. 


\end{proof}

\subsection{Gr\"unwald type processes and their boundary modifications}\label{sec:YLRh}
 In this section we identify the pathwise description of the processes with generators $\Gen^{{\rm LR}}_{-h}$ from \cite[Lemma 3.14]{BKT20} when started on a gridpoint $x\in \Gridh$, where for any $n\in\mathbb N$, we define
 \[
 \Gridh:=\{-1+hj:j=0,1,...,2/h\},\quad\text{where}\quad h=2/(n+1).
 \]
 We first recall from \cite[Proposition 3.11]{BKT20} that  $Y^h$ denotes the  compound Poisson process generated  by 
 \begin{equation}\label{eq:Gh}
\Crh f(x)=\Pi^{-1}G_h\Pi f(x)=\sum_{j=0}^{\infty} f(x+(j-1)h)\Gru_{j,h},
\end{equation}
with transition rate matrix
\begin{equation*}
G_h=
  \begin{pmatrix} 
     \ddots  & \ddots & \ddots & \ddots  &  \ddots& \ddots & \ddots  \\
     \ddots &  \Gru_{1,h} & \Gru_{2,h} & \ddots & \Gru_{n-1,h}  & \Gru_{n,h}  & \ddots   \\
     \ddots & \Gru_{0,h} & \Gru_{1,h} &    \ddots &\Gru_{n-2,h} & \Gru_{n-1,h}&  \ddots \\
     \ddots & 0&\Gru_{0,h}&\ddots    & \ddots &\ddots & \ddots   \\
    \ddots  & \ddots&\ddots&\ddots&\Gru_{1,h}&\Gru_{2,h}& \ddots  \\
    \ddots & 0 &\ddots&0&\Gru_{0,h}&\Gru_{1,h}&  \ddots \\
    \ddots  & \ddots  & \ddots & \ddots  &\ddots & \ddots& \ddots 
  \end{pmatrix},
\end{equation*}
where $\{\Gru_{j,h}:j\in\mathbb N\}$ is the set of coefficients determined by $\sym((1-\xi)/h)=\sum_{j=0}^\infty \Gru_{j,h} \xi^j$ for $|\xi|<1$. We recall from  \cite[Lemma 3.3]{BKT20} that
\begin{equation}\label{eq:26}
\text{$-\Gru_{1,h},\Gru_{0,h},\Gru_{j,h}>0$, for $j\ge 2$, and $-\Gru_{1,h}=\sum_{j\neq 1}\Gru_{j,h}$,}
\end{equation}
and that we set  $\Gru_{-j,h}=0$ for $j\in\mathbb N$. We refer to \cite[Sections 3.1 and 3.3]{BKT20} for an in depth discussion of these coefficients and their properties. 
\begin{remark}[Convention]
From now on, to simplify notation, we drop the $h$ subscript by writing $\Gru_j=\Gru_{j,h}$.  
\end{remark} 
We will apply the maps of Section \ref{sec:skor_maps} to $Y^h$ with $Y^h_0\in \Gridh$ and show that the transition matrix in each case is 
\begin{equation}\label{Glrn+2}
G_{n+2}^{\mathrm{LR}}=
 \begin{pmatrix} 
      0& 0&0&0&0&0&0\\
        {d_0^l}& b_1^l & b_2^l & \cdots & b_{n-1}^l & b_n  &  {d_{n+1}^r}  \\
   0& \Gru_0 & \Gru_1 &    \cdots &\Gru_{n-2} & b_{n-1}^r  & {d_{n}^r} \\
   0& 0& \Gru_0&\ddots &\vdots    & \vdots &\vdots   \\
   \vdots&  \vdots&\ddots&\ddots&\Gru_1&b_2^r&  {d_{3}^r} \\
  0&  0 &\cdots&0&\Gru_0& b_1^r& d_{2}^r\\
    0& 0&0&0&0&0&0\\
  \end{pmatrix}.
\end{equation}
with the coefficients $\{b_j^l,b_j^r, b_n: j=1,2,...,n-1\}$  determined for each boundary condition in Table \ref{mainTable}. Note that \eqref{Glrn+2} equals \cite[Eq. (40)]{BKT20} with the addition of absorbing   Dirichlet boundary conditions at $-1$ and $1$.

\begin{table}[h]
\centering
\begin{tabular}{|c|c|c|c|}
  \hline
  Case  & Rates $b^{l,r},\, d^{l,r} $   \\ 
  \hline 
  $\mathrm{DR}$ & $b_i^l=\Gru_i$,\,\, $d_{0}^l= \Gru_0$ \\
  $\mathrm{NR}$ & $b_i^l=-\sum_{j=0}^{i-1}\Gru_{j}$,\,\, $d_{0}^l= 0$ \\
  $\mathrm{N^*R}$ & $b_1^l=\sum_{j=0}^1\Gru_{j},\,\, b_i^l=\Gru_i,i\ge2$,\,\,  $d_{0}^l= 0$\\
  $\mathrm{LD}$ & $b_i^r=\Gru_i$,\,\, $b_n=b_n^l$,\,\, $d_{i}^r= \sum_{j=i}^\infty \Gru_{j}$ \\
  $\mathrm{LN}$ & $b_i^r=-\sum_{j=0}^{i-1}\Gru_{j},\,\, b_n=-\sum_{j=0}^{n-1}b_j^l$,\,\, $d_{i}^r= 0$   \\
\hline
\end{tabular}
\caption{\label{mainTable}Table of boundary weights used to build the transition rate matrix \eqref{Glrn+2}, assuming  $b_0^l=0$ for $\mathrm{NR}$ and $\mathrm{N^*R}$, and with the exception that in the case $\mathrm{ND}$ we set $d^r_{n+1} = \sum_{j=n+1}^\infty b^l_j$.}
\end{table}

We will use the next proposition and the remark afterward to guarantee the fast-forwarding maps are well-defined.
 \begin{proposition}\label{prop:recur} Assuming \ref{H0}, 
 the process  $Y$ is recurrent, and so is $Y^h$ for every $h>0$.
 \end{proposition}
 \begin{proof}
By \cite[Theorem 25.3,  page 159]{MR1739520} and  \ref{H0} the process $Y$ has finite first moment, and so does $Y^h$ (due to its finite L\'evy measure). By \cite[Theorem 36.7,   page 248]{MR1739520} it is enough to prove that $\mathbb E_0[Y_1]=0$ and $\mathbb E_0[Y^h_1]=0$. Recalling that the Laplace exponents of $Y$ and $Y^h$ are respectively $\sym(ix)$ and $e^{ixh}\sym((1-e^{-ixh})/h)$  (see \cite[Remark 3.2.ii]{BKT20}),  we obtain
\[
\mathbb E_0[Y_1]=\lim_{x\to 0}\frac{\dd}{\dd x}i \mathbb E_0\big[e^{-ixY_1}\big]=\lim_{x\to 0}\frac{\dd}{\dd x}i e^{\sym(ix)}=\lim_{x\to 0} i^2\sym'(ix) e^{\sym(ix)} =0,
\] 
 by \eqref{eq:convhomega}, and similarly
 \[
\mathbb E_0[Y_1^h]=\lim_{x\to 0} \big(he^{ixh}\sym((1-e^{-ixh})/h)+\sym'((1-e^{-ixh})/h) \big)i^2\mathbb E_0\big[e^{-ixY^h_1}\big] =0.
\] 

  \end{proof} 
 \begin{remark}\label{rmk:recur}
 Proposition \ref{prop:recur} with \cite[Theorem 35.4.iii, page 239]{MR1739520} implies that for any $x\in \mathbb R$, $\mathbb P_x$-a.s., $Y$, $Y^h\in D_{a,b}([0,\infty),\mathbb R)$ for any $-\infty\le a<b\le\infty$ (for all $h$ small),   and using $\lim\inf_t Y_t=\lim\inf_t Y_t^h=-\infty$ $\mathbb P_x$-a.s. by \cite[Proposition 37.10.3, page 255]{MR1739520}, we obtain that, $\mathbb P_x$-a.s., $\Nalone(Y)$, $\Nal{h-1}(Y^h)\in D_{-\infty,1}([0,\infty),\mathbb R)$. 
 \end{remark}
 \begin{table}[h]
\centering
\vline
\begin{tabular}{l|c|c}
  \hline
\hspace{1cm} Process $\Yh{LR}$  
\hspace{.1cm} & \hspace{.1cm}  Rate matrix $G_{n+2}^{{\rm LR}}$ 
\hspace{.1cm}\\
	\hline
	1.  $ Y^{{\rm DD},h}= \Dr(\Dl(Y^h))$   & $G_{n+2}^{{\rm DD}}$    \\
  \hline
2.	$  Y^{{\rm DN},h}=  \Dl (\Nrone(Y^h))$ & $G_{n+2}^{{\rm DN}}$   \\
  \hline
3.	 $  Y^{{\rm ND},h}=  \Dr(\Nlone(Y^h))$ &$G_{n+2}^{{\rm ND}}$   \\
  \hline
4.	$  Y^{{\rm NN},h}=  \Nrone(\Nlone(Y^h))$  & $G_{n+2}^{{\rm NN}}$  \\
  \hline
5.	$  Y^{{\rm N^*D},h}=  \Dr(\Nal{h-1}(Y^h))$  & $G_{n+2}^{{\rm N^*D}}$   \\
  \hline
6.	 $  Y^{{\rm N^*N},h}=  \Nrone(\Nal{h-1}(Y^h))$  &$G_{n+2}^{{\rm N^*N}}$    \\
  \hline 
\end{tabular}\vline
\caption{\label{tab:cmt_h}Pathwise representation  $\Yh{LR}$ of the processes generated by the interpolated  matrices $\Gen^{\BC}_{-h}$ on gridpoints $\Gridh$ from \cite[Lemma 3.14]{BKT20}. In the second column we listed the associated transition rate matrix as defined in \eqref{Glrn+2}.}
\end{table}  
 
With the help of Section \ref{subsec:ff} we prove below the desired pathwise representation on gridpoints $\Gridh$ to the Feller processes constructed in Part I in \cite[Lemma 3.14]{BKT20}.
\begin{proposition}\label{prop:Yh_eq_Gh} Assuming \ref{H0},  for any $n\in\mathbb N$ and letting $h=2/(n+1)$, the Markov process taking values in $ \Gridh$ with transition rate matrix $G_{n+2}^{\BC}$  is $\Yh{LR}$ given  in Table \ref{tab:cmt_h}. 
\end{proposition}
\begin{proof} Recall that all processes are absorbed once they visit 1 or $-1$, which justifies the first and the last row. It is immediate to observe that for all cases the rate of leaving a point  in $\Gridh\backslash\{-1,h-1,1-h,1\}$  remains $-\Gru_1$, and this holds also for  $\Yh{DR}$ leaving $h-1$ and  $\Yh{LD}$ leaving $1-h$. Also the rates of landing on $\Gridh\backslash\{-1,1-h,1\}$ from $\Gridh\backslash\{-1,h-1,1-h,1\}$ are clear. The rates for  $\Yh{NR}$ and $\Yh{N^*R}$ leaving $h-1$ and  $\Yh{LN}$ leaving $1-h$ are proved in Theorem \ref{thm:ffwaitYh} and Proposition \ref{prop:NawaitYh}.  The first (last) column in $\Gen^{{\rm LR}}_{n+2}$  is clear as   $\Dl$ ($\Dr$)
allows to reach $-1$ (1) and does so by accumulating the intensities of all the jumps less (greater) or equal than $-1$ (1), meanwhile the other maps never allow to reach $\{-1,1\}$. Note that the map $\Nrone$  accumulates the intensities of jumps of $Y^h$ to a state greater or equal to $1-h$ on $1-h$, by Proposition \ref{prop:recur}. 
The coefficients for the second row of $\Gen^{{\rm NR}}_{n+2}$ and the second last row of  $\Gen^{{\rm LN}}_{n+2}$ are obtained immediately from  Theorem \ref{thm:ffwaitYh}. The last $n$ coefficients for the second row of $\Gen^{{\rm N^*R}}_{n+2}$ are clear and we are done.



\end{proof}

\begin{proposition}\label{prop:NawaitYh}
For any $h>0$, if \ref{H0} holds and  $Y^h_0=0$, then the first jump of  $\Nal{0}(Y^h)$ and of  $\Nar{0}(Y^h)$  are  exponentially distributed with rate $-(\Gru_1+\Gru_0)$ and $\Gru_0$, respectively. 
\end{proposition}
\begin{proof} 
 Denote by $J_m$ the time of the $m$-th jump of $Y^h$ with the convention $J_0=J_{-1}=0$ and compute using standard arguments    and \eqref{eq:26},
\begin{align*}
&\,\mathbb P_0[\Nar{0}(Y^h)_s =0\,\, \text{for all}\, s\le t]\\
=&\,e^{\Gru_1 t} \sum_{n=0}^\infty \frac{(-\Gru_1 t)^n}{n!} \mathbb P_0[Y^h_{J_m}- Y^h_{J_{m-1}}\ge 0\, \,\text{for all}\, 0\le m\le n ]\\
=&\,e^{\Gru_1 t} \sum_{n=0}^\infty \frac{(-\Gru_1 t)^n}{n!}\left( \frac{\sum_{m=2}^\infty \Gru_m}{-\Gru_1}\right)^n= e^{-\Gru_0 t}.
\end{align*}
A similar computation using $\mathbb P_0[Y^h_{J_m}- Y^h_{J_{m-1}}\le 0\,\, \text{for all}\, 0\le m\le n ]=(-\Gru_0/\Gru_1)^n$ yields
$\mathbb P_0[\Nal{0} (Y^h)_s =0\,\, \text{for all}\, s\le t] = e^{(\Gru_1 +\Gru_0) t}.$
\end{proof}

\subsubsection{Details for fast-forwarding}\label{subsec:ff}

We derive the waiting times and the distribution of the jumps of fast-forwarding  $Y^h$ at a boundary. We use the auxiliary process $Y^{h,{\rm stop}}$, which is defined as $Y^h$ stopped the first time it visits $h\mathbb N$. Then, for each $h>0$, the transition rate matrix of $Y^{h,{\rm stop}}$ is 
 \begin{equation}\label{eq:Gstop}
 G_{\mathrm{stop}}=(g_{i,j})_{i,j\in\mathbb Z}\quad\mbox{ with }\quad g_{i,j}=\begin{cases}\Gru_{j-i+1},&i\le 0,\\0,&i> 0.\end{cases}
\end{equation}   

\begin{theorem}\label{thm:ffwaitYh} Assume \ref{H0}. Then the process $\Nlone(Y^h)$ with   $Y_0^h=-1+h$ jumps the first time according to an exponential waiting time with rate $\Gru_0$, and it lands on $-1+(1+j)h$ with probability $-(\sum_{i=0}^j\Gru_j)/\Gru_0$ for each $j\ge1$.  The process $\Nrone(Y^h)$ with   $Y_0^h=1-h$ jumps the first time according to an exponential waiting time with rate $\Gru_0$, and it lands on $1-2h$. 
\end{theorem}
\begin{proof}  Without loss of generality 
we look at the process $\Nl{-h}(Y^h)$ with  $Y_0^h=0$.  
We first justify the identities for $j\in\mathbb N$ 
\[
\mathbb P_0\left[ \Nl{-h}(Y^h)_{J_1}=jh \right]=\mathbb P_0\left[ Y^{h}_{\sigma_{\mathbb N }^{Y^h}}=jh \right]=\mathbb P_0\left[\sigma_j^{Y^{h,{\rm stop}}}<\infty  \right] = -\frac1{\Gru_0 } \sum_{i=0}^j\Gru_j,
\]
where $J_1$ is the first jump time of $\Nl{-h}(Y^h)$, ${\sigma_{\mathbb N }^{Y^h}=\inf\{t>0: Y^{h}_t \ge h } \}$ 
and $\sigma_j^{Y^{h,{\rm stop}}}=\inf \{t>0: Y^{h,{\rm stop}}_t=hj \}$. The first identity is a direct consequence of the pathwise definition $\Nl{-h}$ (Definition \ref{def:ff}) along with the piecewise constant paths of $Y^h$ (or simply put, $\Nl{-h}(Y^h)$ visits $jh$ the first time it moves if and only if $Y^h$ visits $jh $ the first time it enters $h\mathbb N$); the second identity is clear; and the third identity is due   Corollary \ref{cor:lambdaRlambda*} and \eqref{eq:ergodic_stroock}.
Also by Proposition \ref{prop:waitFF} we know that $J_1$ at 0 is an exponential waiting time  with rate $\Gru_0 $, which concludes the proof for $\Nl{-h}(Y^h)$. For   $\Nr{h}(Y^h)$ the proof is immediate from Propositions \ref{prop:waitFF} and \ref{prop:recur}.
\end{proof}


\begin{proposition}\label{prop:waitFF}
For any $h>0$, assuming \ref{H0} and $Y^h_0=0$, the first jump of $\Nl{-h}(Y^h)$ and of $\Nr{h}(Y^h)$ for  are both exponentially distributed with rate $\Gru_0$. 
\end{proposition}
\begin{proof} 
Denote by $E^{Y^h}_m$ the waiting time between the $(m-1)$-th and the $m$-th jump of  $Y^h$ for $m\in\mathbb N$, and recall that $\{E^{Y^h}_m:m\in\mathbb N\}$ is a collection of iid  exponential random variables with rate  $\mu=-\Gru_1$ and      $\mathbb P[\sum_{m=1}^j E^{Y^h}_m>t]=e^{-\mu t}\sum_{m=0}^{j-1}\frac{(\mu t)^m}{m!}$ (Erlang distribution).
Let $J_m$ denote the $m$-th jump time of the stopped process $Y^{h,{\rm stop}}$ ($Y^h$ stopped the first time it visits $h\mathbb N$) and define $\sigma^{(m)}$ to be the $m$-th return time at 0  of $Y^{\rm{stop}}$ with the convention that $\sigma^{(0)}=0$, and define the function $j(\sigma^{(m)}):\{\sigma^{(m)}<\infty\}\to \mathbb N\cup\{0\}$ such that $\sigma^{(m)} =J_{j(\sigma^{(m)})}$.    We now justify the following identities
\begin{align*}
&\quad\,\,\mathbb P_0[\Nl{-h}(Y^h)_s=0\,\,\text{for all }s\le t] \\
&= \mathbb P_0\left[\int_0^\infty \mathbf 1_{\{Y^{h,\rm{stop}}_z=0\}}\,\dd z>t\right]\\
&=\sum_{n=0}^\infty \mathbb P_0\left[\left\{\int_0^\infty \mathbf 1_{\{Y^{h,\rm{stop}}_z=0\}}\,\dd z>t\right\}\cap\left\{  \sigma^{(n)}<\infty, \sigma^{(n+1)}=\infty \right\}\right]\\
&=\sum_{n=0}^\infty \mathbb P_0\left[\left\{\sum_{m=0}^n\int_{J_{j(\sigma^{(m)})}}^{J_{j(\sigma^{(m)})+1}} \mathbf 1_{\{Y^{h,{\rm stop}}_z=0\}}\,\dd z>t\right\}\cap\left\{  \sigma^{(n)}<\infty, \sigma^{(n+1)}=\infty \right\}\right]\\
&=\sum_{n=0}^\infty \mathbb P_0\left[\left\{\sum_{m=0}^n J_{j(\sigma^{(m)})+1}-J_{j(\sigma^{(m)})} >t\right\}\cap \left\{  \sigma^{(n)}<\infty, \sigma^{(n+1)}=\infty \right\}\right]\\
&=\sum_{n=0}^\infty \left(e^{-\mu t}\sum_{m=0}^{n}\frac{(\mu t)^m}{m!}\right)\mathbb P_0\left[ \sigma^{(n)}<\infty, \sigma^{(n+1)}=\infty \right]\\
&=e^{-\mu t}\sum_{m=0}^\infty \frac{(\mu t)^m}{m!} \left(\sum_{n\ge m}\mathbb P_0\left[ \sigma^{(n)}<\infty, \sigma^{(n+1)}=\infty \right]\right)\\
&=e^{-\mu t}\sum_{m=0}^\infty \frac{(\mu t)^m}{m!} \mathbb P_0\left[\sigma^{(m)}<\infty\right]\\
&=e^{(-\mu +\mu \mathbb P_0\left[\sigma^{(1)}<\infty\right] )t} .
\end{align*}
The first identity is immediate from the definition of $\Nl{-h}(Y^h)$, as this process does not move for $t$-units of time if and only if  the free process $Y^h$ spend at least $t$-units of time at 0 before visiting $h\mathbb N$; the second identity uses  $ \{ \sigma^{(m)}<\infty, \sigma^{(m+1)}=\infty \} \cap \{ \sigma^{(n)}<\infty, \sigma^{(n+1)}=\infty \}=\emptyset  $ for all $n\neq m$; the third and fourth identities are clear observing that, on $\{ \sigma^{(n)}<\infty, \sigma^{(n+1)}=\infty \}$, $Y^{h,\rm{stop}}$ visits 0 only at times $ \sigma^{(m)}$ for $m\le n$;  the sixth identity is an application of Tonelli's Theorem: the second last identity is clear: the last identity uses  $ \mathbb P_0\left[\sigma^{(m)}<\infty\right] =\mathbb P_0\left[\sigma^{(1)}<\infty\right]^m$ \cite[page 96]{Stroock13}; the fifth identity follows by observing that 
$$\left(J_{j(\sigma^{(m)})+1}-J_{j(\sigma^{(m)})}\right) =E^{Y^h}_{j(\sigma^{(m)})} \quad\text{on}\quad \{\sigma^{(m)}<\infty\},$$ and computing 
\begin{align*}
 &\quad\,\, \mathbb P_0\left[\left\{\sum_{m=0}^{n} E^{Y^h}_{j(\sigma^{(m)})} >t\right\}\cap \{\sigma^{(n)}<\infty, \sigma^{(n+1)}=\infty\} \right]\\
  &= \sum_{p\in P_n}\mathbb P_0\left[  \left\{\sum_{m=0}^{n} E^{Y^h}_{p(m)} >t\right\}\cap \{\sigma^{(n)}<\infty, \sigma^{(n+1)}=\infty\}\cap\left\{ j(\sigma^{(m)})=p(m):m\le n\right\}\right]\\
  &= \sum_{p\in P_n}\mathbb P_0\left[  \sum_{m=0}^{n} E^{Y^h}_{p(m)} >t \right]\mathbb P_0\left[ \{\sigma^{(n)}<\infty, \sigma^{(n+1)}=\infty\}\cap\left\{ j(\sigma^{(m)})=p(m):m\le n\right\}\right]\\
    &=\mathbb P_0\left[  \sum_{m=1}^{n+1} E^{Y^h}_{m} >t \right] \sum_{p\in P_n}\mathbb P_0\left[ \{\sigma^{(n)}<\infty, \sigma^{(n+1)}=\infty\}\cap\left\{ j(\sigma^{(m)})=p(m):m\le n\right\}\right] ,
\end{align*}
where $P_n=\{p:\{0,1,...,n\}\to\mathbb N;\, p(m)<p(m+1)\}$
  is countable as $P_n=\cup_{l=n}^\infty \{p:\{0,1,...,n\}\to\{1,...,l+1\}, p(m)<p(m+1)\}$, and note that we used independence of the event $\{\sigma^{(n)}<\infty, \sigma^{(n+1)}=\infty , j(\sigma^{(m)})=p(m), m\le n \}  $ and the random variables  $\{E^{Y^h}_l:l\ge 0\}$ (the event depends only on the jump magnitudes of $Y^h$, not on the length of the waiting times of $Y^h$ \cite[Page 81]{Stroock13}). 

Hence we showed that the waiting time is exponential, and it remains to compute $\mathbb P_0\left[\sigma^{(1)}<\infty\right]$. We directly compute
\begin{align*}\mathbb P_0\left[\sigma^{(1)}<\infty\right]&=\mathbb P_0\left[Y^{h,\text{stop}}_{J_1}=-1,\, s\mapsto Y^{h,\text{stop}}_{J_1+s}\text{ hits }0  \right]\\
& =\mathbb P_0\left[Y^{h,\text{stop}}_{J_1}=-1\right]\mathbb P_{-1}\left[Y^{h,\text{stop}} \text{ hits }0  \right]   \\ 
& =\mathbb P_0\left[Y^{h}_{J^{Y^h}_1}=-1\right]\mathbb P_{0}\left[Y^{h,\text{stop}} \text{ hits }1 \right] \\
& =\frac{\Gru_0}{-\Gru_1}\frac{-\Gru_0-\Gru_1}{\Gru_0},
\end{align*}
where $J_1^{Y^h}$ is the first jump time of $Y^{h}$, and we used the strong Markov property in the second identity, the third identity is clear, and we used Corollary \ref{cor:lambdaRlambda*} with \eqref{eq:ergodic_stroock} in the last identity. To prove the statement for   $ \Nr{h}(Y^h)$ observe that the waiting time is exponential by the same proof above, but using $Y^{h,\text{stop}-}$ instead of  $Y^{h,\text{stop}}$, where $Y^{h,\text{stop}-}$ is the process $Y^{h}$ stopped on its first visit of $-h\mathbb N$. To calculate the average of the exponential time it remains to compute $\mathbb P_0\left[\sigma^{(1)-}<\infty\right]$, where $\sigma^{(1)-}$ is the time of first return at 0 of  $Y^{h,\text{stop}-}$. Because $Y^h$ moves to the left by single step and it is recurrent (Proposition \ref{prop:recur}) we easily obtain, using \eqref{eq:26},
\[
\mathbb P_{0}\big[\sigma^{(1)-}<\infty\big]=\mathbb P_{0}\left[ Y^{h}_{J_1^{Y^h}}=1\right]=\frac{\sum_{m=2}^\infty\Gru_m}{-\Gru_1}=\frac{- \Gru_0- \Gru_1}{-\Gru_1}.
\]

\end{proof}

The following is an easy consequence of Propositions  \ref{prop:NawaitYh} and \ref{prop:waitFF}, and we omit the proof.
\begin{corollary}
For any $h>0$, assuming \ref{H0} and $Y^h_0\in \{1-jh:\,j\in\mathbb N\}$, it holds that $\Nrone(Y^h)=   \Nar{1-h}(Y^h) $ in law. 
\end{corollary}

In Proposition \ref{prop:waitFF} we found the exponential waiting time of $\Nlone(Y^h)$ at $h-1$. To find the distribution of the where $\Nlone(Y^h)$ moves the first time it jumps we prove the generalised version of \cite[Theorem 16]{MR3720847}, with the proof essentially unchanged. We denote by $G_{\mathrm{stop}}^T$ the transpose of $G_{\mathrm{stop}}$, so that 
\begin{align*}
\big(G_{\mathrm{stop}}^T\vec x\big)_n=\sum_{k\in\mathbb Z}g_{k,n}\vec x_k 
&=\sum_{k\ge n+1}\Gru_{k} \vec x_{n+1-k}.
\end{align*}


\begin{lemma} Assume \ref{H0}. Then the resolvent of $G^T_{\mathrm{stop}}$ for $\LTp>0$ evaluated at $\vec e_{0}$ is given by $\left((\LTp I-G^T_{\mathrm{stop}})^{-1}\vec e_{0}\right)=\vec{y}/\Gru_{0} $, where
\begin{equation}\label{resy*}
\vec y_n=\begin{cases}e^{h(n-1)\varphi^{-1}(\LTp)},&n\le 0,\\
\frac1\LTp \sum_{k=n}^\infty \Gru_{k+1}  e^{h(n-1-k)\varphi^{-1}(\LTp)}, &n> 0,\end{cases}
\end{equation}
 and $\varphi(\LTp):=e^{ h \LTp }\sym((1-e^{- h \LTp})/h)=\sum_{k=0}^\infty \Gru_{k} e^{h(1-k)\LTp}.$
\end{lemma}

\begin{proof}
Note that $\varphi$ is invertible on $[0,\infty)$, as its derivative is positive  on $(0,\infty)$ and bounded below away from 0, which is immediate from $\sym',\,\sym>0$ on $(0,\infty)$ and
\[\varphi'(\LTp)=he^{h\LTp}\sym((1-e^{- h \LTp})/h)+\sym'((1-e^{- h \LTp})/h)>0.\]
 Next we apply  $\LTp I-G^T_{\mathrm{stop}}$ to our candidate resolvent  and show that the result is indeed $\vec e_{0}$. Recall that $\vec{y}$ is given by \eqref{resy*}. For $n\le -1$,
\begin{equation*}\begin{split}\left((\LTp I-G_{\mathrm{stop}}^T)\vec y\right)_n&=\LTp e^{ h (n-1)\varphi^{-1}(\LTp)}-\sum_{k=0}^\infty \Gru_{k}  e^{ h (n-k)\varphi^{-1}(\LTp)}\\
&=\LTp e^{ h (n-1)\varphi^{-1}(\LTp)}-e^{ h (n-1)\varphi^{-1}(\LTp)}\sum_{k=0}^\infty \Gru_{k}  e^{ h (1-k)\varphi^{-1}(\LTp)}\\
&=\LTp e^{ h (n-1)\varphi^{-1}(\LTp)}-e^{ h (n-1)\varphi^{-1}(\LTp)}\varphi(\varphi^{-1}(\LTp))=0.
\end{split}\end{equation*}
For $n=0$, 
\begin{equation*}\begin{split}\left((\LTp I-G^T_{\mathrm{stop}})\vec y\right)_{0}&=\LTp e^{- h \varphi^{-1}(\LTp)}-\sum_{k=1}^\infty \Gru_{k}  e^{- h k\varphi^{-1}(\LTp)}\\
&=\LTp e^{- h \varphi^{-1}(\LTp)}-e^{- h \varphi^{-1}(\LTp)}\sum_{k=0}^\infty \Gru_{k}  e^{ h (1-k)\varphi^{-1}(\LTp)}+ \Gru_{0}  \\
&=\LTp e^{- h \varphi^{-1}(\LTp)}-e^{- h \varphi^{-1}(\LTp)}\varphi(\varphi^{-1}(\LTp))+\Gru_{0}  =\Gru_{0}  .
\end{split}\end{equation*}
For $n\ge 1$, $\left((\LTp I-G^T_{\mathrm{stop}})\vec{y}\right)_{n}=0$ by the  definition of $G^T_{\mathrm{stop}}$ and $\vec{y}$. Hence $(\LTp I-G^T_{\mathrm{stop}})\vec{y}=\Gru_{0} \vec e_{0}$ and therefore $(\LTp I-G^T_{\mathrm{stop}})^{-1}\vec e_{0}=\vec{y}/\Gru_{0} $.
\end{proof}

\begin{corollary}\label{cor:lambdaRlambda*}  Assuming \ref{H0}, we have that
$\lim_{\LTp\downarrow 0} \LTp(\LTp I-G^T_{\mathrm{stop}}) ^{-1} \vec{e}_{0}=\vec z,$
where $\vec z_j=0$ for $j\le 0$ and $\vec z_j=\frac{-1}{\Gru_{0} } \sum_{k=0}^j \Gru_{k} $ for $j> 0$.
\end{corollary}
\begin{proof} As $\LTp\downarrow 0$, $\varphi^{-1}(\LTp)\to 0$ and hence 
$$\left(\LTp(\LTp I-G^T_{\mathrm{stop}}) ^{-1}\vec e_{0}\right)_n\to 0$$ for all $n\le 0$. For $n> 0$,
$$\left(\LTp(\LTp I-G^T_{\mathrm{stop}}) ^{-1}\vec e_{0}\right)_n\to \frac{1}{\Gru_{0} } \sum_{k=n}^\infty \Gru_{k+1} ,$$
and we conclude with \eqref{eq:26}.

\end{proof}


\subsection{Convergence of the  Gr\"unwald type process}\label{sec:convfreegru}
\begin{proposition}\label{prop:gruntolevy} Assume \ref{H0} and let $-\infty\le a<b\le \infty$. Then, if $Y^h_0\Rightarrow Y_0$ on $\R$ as $h\to 0$, then   $Y^h\Rightarrow Y$    on $D_{c,d}([0,\infty),\mathbb R)$ as $h\to 0$.
\end{proposition}
\begin{proof} Recall that  on $C^\infty_c(\mathbb R)$, two integration by parts prove 
\begin{align*}
\Mr   g(x)&=\int_0^\infty \Phi(y)g''(x+y)\,\dd y=\int_0^\infty\big( g(x+y)-g(x)-yg'(x)\big)\,\phi({\rm d}y),
\end{align*}
so that  the closure of $(\Mr,C^\infty_c(\mathbb R) ) $ in $C_0(\mathbb R) $ generates the Feller process $Y$ \cite[Corollary 2.10]{MR3156646}.  Let $0\neq g\in C^\infty_c(\mathbb R)$ and $a,b\in\mathbb R$ such that the support of $g$ is contained in $[a,b]$. Let $\epsilon>0$ be arbitrary. Then, for any $a'$ small enough with $a'<a$ and all $h>0$ small,  
\begin{equation*}
\left|\Crh  g(x)- \Mr g(x)\right|=\left\{\begin{split}
&0, & x\in (b,\infty),\\
&\left|\Crh g(x)- \Mr g(x)\right|, & x\in [a',b],\\
&|e_h(x)|, & x\in (-\infty,a'),
\end{split}\right.
\end{equation*}
where 
\[
e_h(x)=\sum_{j=m(x,h)}^\infty \Gru_{j} g(x+(j- 1)h)-\Mr g(x),
\]
with $ m(x,h)\ge2$ being the largest integer so that  $ (m(x,h)-1)h\le  a-x$. Thus we can choose  $a'$ such that for all $x\le a'$
\begin{align*}
|e_h(x)| & \le \sum_{j=m(x,h)}^\infty  \Gru_{j} |g(x+(j- 1)h)|+\epsilon \le \|g\|_{C_0(\mathbb R)}\Big(\sum_{j=m(a',h)}^\infty\Gru_{j} \Big) +\epsilon,
\end{align*}
because  $ \Mr g\in C_0(\mathbb R)$ and for each $h>0$, $m(a',h)\le m(x,h)$ for all $x\le a'$. By \eqref{eq:26} and an easy adaptation of the proof of the third limit in \cite[Eq. (32)]{BKT20},
\[
\limsup_{h\to 0}\sum_{j=m(a',h)}^\infty\Gru_{j} =-\limsup_{h\to 0}\sum_{j=0}^{m(a',h)-1}\Gru_{j} \le  3e\, \big(\phi([a-a',\infty))+\phi((a-a',\infty)) \big),
\]
and because $\lim_{y\to\infty }\phi((y,\infty))=0$, we can choose a possibly smaller $a'$ such that 
\[ 3e\, \big(\phi([a-a',\infty))+\phi((a-a',\infty)) \big)\le \frac\epsilon{2\|g\|_{C_0(\mathbb R)}},\]
so that for all small $h>0$ we can use the bound 
\begin{align*}
\sup_{x\in(-\infty,a']}|e_h(x)| & \le 2\epsilon.
\end{align*}
On the other hand, by \cite[Corollary 3.8]{BKT20}, for all $h>0$ small
\[
\left\|\Crh g- \Mr g\right\|_{C[a',b]}\le\epsilon.
\] 
As we proved strong convergence of the generators on the core $C_c^\infty(\mathbb R)$, by \cite[Theorem 17.25]{MR1464694} we obtain the weak convergence of the respective stochastic processes on $D([0,\infty),\mathbb R)$. The convergence on  $D_{c,d}([0,\infty),\mathbb R)$ follows immediately from Remark \ref{rmk:recur}.
 \end{proof}

\section{Convergence of processes and semigroups}\label{sec:main_II}
We are finally ready to prove our main results, by combining the Skorokhod continuity results of Section \ref{sec:skor_maps} with the Trotter–Kato convergence proved in Part I in \cite[Theorem 5.1]{BKT20}.  To do so, we first show that we can find a sequence of grids from the approximation scheme of Part I so that $\Yh{LR}_t\Rightarrow \Y{LR}_t$ for almost every $t>0$ and $\Y{LR}_0=\Yh{LR}_0=x$ in a dense subset of $(-1,1)$. Then we  combine this weak convergence with the Trotter–Kato convergence of the (interpolated) Gr\"unwald type semigroups of Part I. As we showed that on gridpoints these semigroups are given by $\Yh{LR}$ (Proposition \ref{prop:Yh_eq_Gh}), we can characterise pathwise the limit semigroups of Part I, and the proof of Theorem \ref{thm:main_II_intro} is complete. Finally, we apply our results to derive new resolvent measures for the processes involving a left fast-forwarding boundary condition.

\subsection{Skorokhod convergence}\label{sec:skor}
\begin{lemma}[Skorokhod convergence]\label{lem:skor} Assume \ref{H0} and recall the definitions of $Y,\,\Y{LR},\,Y^h$ and $\Yh{LR}$ from Tables \ref{tab:YLR} and \ref{tab:cmt_h}. For any $n\in\mathbb N$, let $h=2/(n+1)$ and $Y^h_0=Y_0=x\in \Gridh\backslash\{-1,1\}$. Let $\{h_{j}(x)=h_{j}:j\in\mathbb N\}$ be a sequence such that $h_{j}\to 0$ as $j\to\infty$ and $x\in \Grid{h_j} $  for all $j\in\mathbb N$. Then
   \[
   Y^{{\rm LR},h_{j}}\Rightarrow Y^{{\rm LR}}\quad\text{as } j\to\infty \quad \text{on } D ([0,\infty),\mathbb R),
   \]
   and in particular,  for all $t>0$ outside of a countable set and $g\in C_0(\Omega)$,
   \begin{equation*}\label{eq:proj_convergence}
       \mathbb E_x\big[g\big( Y^{{\rm LR},h_j}_t\big)\big]\to \mathbb E_x\left[g\left(Y^{{\rm LR}}_t\right)\right].
   \end{equation*}
\end{lemma}
\begin{proof}  Note that we only need to prove the six weak convergences on $D([0,\infty),\mathbb R)$, as then the last convergence follows immediately from Proposition \ref{prop:asthenweak}, as 
$\mathbb P_x[\Y{LR}_{t-}\neq \Y{LR}_{t}]= 0 
$ for all $t>0$ outside of a countable set \cite[Section 16, page 174]{MR1700749}.  We recall the definitions of the killing maps in  Definition \ref{def:DD} and Proposition \ref{prop:killmaps_measurable}, of the (measurable) fast-forwarding maps in Definitions \ref{def:ff} and \ref{def:ff_b_twosd} and Remark \ref{rmk:convffmaps}, and the (continuous) reflecting maps in Definition \ref{def:refl}. Also, in each case we use, without mention, Proposition \ref{prop:gruntolevy} and every convergence is understood ``as $j\to\infty$''.
\begin{enumerate}
\item $\rm DD$: by Corollary \ref{cor:for_measurability} and Proposition \eqref{prop:exit_open_closed}   we have  $\Yh{DD}=\Dl(\Dr(Y^h))=\Dlp{h}(\Drp{h}(Y^h))$ and $\Y{DD}=\Dl(\Dr(Y))=\Dlp{0}(\Drp{0}(Y))$. Then the result follows by Corollary \ref{cor:sparated_CMT} and Propositions  \ref{prop:killmaps_measurable},  \ref{prop:SMkill1}, \ref{prop:SMkill-1} and \ref{prop:Ypath_continuity_point}.   

\item $\rm DN$: By Propositions \ref{prop:Ypath_continuity_point} and Corollary \ref{cor:for_measurability}, $\Nrone(Y)$ is $\mathbb P_x$-a.s. a continuity point of $\Dl$, $\Dl(\Nrone(Y))=\Dlp{0}(\Nrone(Y))$ and $\Dl(\Nrone(Y^h))=\Dlp{h}(\Nrone(Y^h))$. Moreover,   Corollary \ref{cor:ff_b} and  Proposition \ref{prop:Ypath_continuity_point} prove that  $\Nrone(Y^{h_j})\Rightarrow \Nrone(Y)$ on $D([0,\infty),\R)$. Then we conclude with Corollary \ref{cor:sparated_CMT}.
\item $\rm ND$: Similar as the above and omitted.
\item $\rm NN$: By Proposition \ref{prop:Ypath_continuity_point}, $Y$ is $\mathbb P_x$-a.s. a continuity point of $\Nlone(\Nrone)$ and by Proposition \ref{prop:recur} the paths of $Y^h$ belong to $D_{-1,1}([0,\infty),\mathbb R)$ for any $h>0$. Then the result follows from Theorem \ref{thm:CMT} and Corollary \ref{cor:ff_twosid}.  
\item $\rm N^*D$: If we show that
\begin{equation}\label{eq:Nal_conv}
\Nal{h_j-1}(Y^{h_j}) \Rightarrow \Nalone(Y) \quad\text{on}\quad D([0,\infty),\mathbb R),
\end{equation}
then the result follows from Proposition \ref{prop:Ypath_continuity_point} and Corollaries \ref{cor:for_measurability} 
and \ref{cor:sparated_CMT}. To do so, by Theorem \ref{thm:skor_repr} we denote by the same notation versions of $Y$ and $Y^{h_j}$ such that  $Y^{h_j}\to Y$ in  $D([0,\infty),\mathbb R)$  ($\mathbb P_x$-a.s.). Then, by continuity of $\Nalone$, we known that $\Nalone(Y^{h_j})\to \Nalone(Y)$ in $D([0,\infty),\mathbb R)$. Thus, by the triangle inequality, if we to show that $d_{J_1}(\Nal{h_j-1}(Y^{h_j}), \Nalone(Y^{h_j}))\to 0$, then \eqref{eq:Nal_conv} follows by Proposition \ref{prop:asthenweak}. And so, observing that $
c_j\Nal{h_j-1}(Y^{h_j})=  \Nalone(c_jY^{h_j})
$ pointwise if $c_j=(1-h_j)^{-1}$, then, using   \eqref{eq:J1scaleout} in the second inequality,
\begin{align*}
d_{J_1}\left( \Nal{h_j-1}(Y^{h_j}),\Nalone(Y^{h_j})\right)
&\le d_{J_1}\left(\Nal{h_j-1}(Y^{h_j}),c_j\Nal{h_j-1}(Y^{h_j})\right)\\
&\quad+d_{J_1}\left( \Nalone(c_jY^{h_j}),\Nalone( Y^{h_j} )\right)\\
&\le (c_j-1)+d_{J_1}\left( \Nalone(c_jY^{h_j}),\Nalone( Y^{h_j} )\right),
\end{align*}
and the second term  vanishes by continuity of $\Nalone$ and    $c_jY^{h_j}\to Y$in $D([0,\infty),\mathbb R)$ again by triangle inequality and \eqref{eq:J1scaleout}.

\item $\rm N^*N$: by noting that the convergence \eqref{eq:Nal_conv} holds on $D_{-\infty,1}([0,\infty),\mathbb R)$, this case follows by  Theorem \ref{thm:CMT} and Proposition \ref{prop:Ypath_continuity_point}.

 
\end{enumerate}
\end{proof}

 \begin{corollary}\label{cor:N*NeqNN}
Assume \ref{H0}. Then the process $\Y{N^*N}$ equals the process $[\Nalone\Narone](Y)$ in law.
\end{corollary}
\begin{proof} Let $x\in(-1,1)$ and a corresponding sequence $\{h_j:j\in\mathbb N\}$ as in Lemma \ref{lem:skor} and let $Y^{h_j}_0=x$ for all $j\in\mathbb N$ so that $Y^{{\rm N^*N},h_j}_0=[\Nal{{h_j}-1}\Nar{1-{h_j}}](Y^{h_j})_0=x$. Then, using Propositions \ref{prop:NawaitYh} and \ref{prop:waitFF} with the piecewise constant paths of $Y^h$, $Y^{{\rm N^*N},h_j}=[\Nal{h_j-1}\Nar{1-h_j}](Y^{h_j})$ $\mathbb P_x$-a.s. and by same trick with multiplication by $c_j=(1-h_j)^{-1}$ in step 5 of the proof of Lemma \ref{lem:skor}, we obtain that $[\Nal{{h_j}-1}\Nar{1-{h_j}}](Y^h)_t\Rightarrow [\Nalone\Narone](Y)_t$ as $j\to\infty$ on $[-1,1]$ for all  $t$ outside of a countable set. And so we obtain, by the Dominated Convergence Theorem,
\[
\int_0^\infty e^{-\LTp t}\mathbb E_x\big[g\big(\Y{N^*N}_t\big)\big]\, \dd t   = \int_0^\infty e^{-\LTp t} \mathbb E_x\big[g\big([\Nalone\Narone](Y)_t\big)\big] \, \dd t \]   for  and for any $\LTp>0$, $g\in C[-1,1]$ and $x$ in a dense subset of $(-1,1)$. Then, by Proposition \ref{prop:cont_semigroup} and \cite[Remark 2.21]{BKT20}\footnote{Note that this remark derives its conclusion for the two-sided reflection $[\Nalone\Narone](Y)$.},  the resolvents agree everywhere on $[-1,1]$. As    for each $x\in [-1,1]$, $t\mapsto \mathbb E_x[g(\Y{N^*N}_t)]$ is right continuous (Remark \ref{rmk:cont_semigroup}) and $t\mapsto \mathbb E_x[g([\Nalone\Narone](Y)_t)]$ is continuous (\cite[Remark 2.21]{BKT20}), we obtain that the semigroups agree on $C[-1,1]$ by \cite[Lemma 1.1, Chapter I]{MR0193671} and we are done. \end{proof}

\subsection{Pathwise characterisation of backward and forward equations}

\begin{theorem}\label{thm:main_II}[Pathwise characterisation] Under assumption \ref{H0}, the six processes in Table \ref{explicitProcesses_II} induce Feller semigroups on $C_0(\Omega)$ with backward and forward generators given in Table \ref{explicitProcesses_II}.  
\end{theorem} 
\begin{proof}
 By \cite[Corollaries 5.2 and 5.4]{BKT20}, the operators in the last column of Table \ref{explicitProcesses_II} generate Feller semigroups on $C_0(\Omega)$ and the operators in the second column equal the restriction to $L^1[-1,1]$ of the respective dual operators on $C_0(\Omega)^*$, the dual space of $C_0(\Omega)$.   For cases 1, 2, 5 and 6 in Table \ref{explicitProcesses_II} the conclusion follows by \cite[Remark 2.21]{BKT20} and Corollary \ref{cor:N*NeqNN}. For case 3 or 4 in Table \ref{explicitProcesses_II}, denote by $P$ the Feller semigroup on $C_0(\Omega)$ generated by $(\Cr,\BC)$, and let $P^h$ be the respective approximating semigroup as constructed in  \cite[Theorem 5.1]{BKT20}. Then, by \cite[Theorem 5.1]{BKT20}, $P_t^hg(x)\to P_t g(x) $  as $h\to 0$ for any $g\in C_0(\Omega)$, $x\in \Omega$ and $t>0$. By Proposition \ref{prop:Yh_eq_Gh},  $P_t^{h}g(x)=\mathbb E_x[g(Y_t^{{\rm LR},{h}} )]$ for any $x\in {\rm Grid}_h$, $t>0$ and $g\in C_0(\Omega)$. And so we choose $\{h_j\}_{j\in\mathbb N}$ as in Lemma \ref{lem:skor}, and the same lemma implies that   as $j\to\infty$  
 $$\mathbb E_x\big[g\big(Y_t^{{\rm LR},{h_j}}\big)\big]\to  \mathbb E_x\big[g\big(\Y{LR}_t\big)\big] \quad \text{for almost every }t>0.$$  Therefore, by the Dominated Convergence Theorem, we obtain that for any $\LTp>0$, $g\in C_0(\Omega)$ and $x$ in a dense subset of $\Omega$, \[
 \int_0^\infty e^{-\LTp t}P_t g(x)\, \dd t=\mathbb  \int_0^\infty e^{-\LTp t}\mathbb E_x\big[g\big(\Y{LR}_t\big)\big]\, \dd t,\]
 and by Proposition \ref{prop:cont_semigroup} these resolvents agree for all $x\in\Omega$. Then, by Remark \ref{rmk:cont_semigroup}, continuity of $t\mapsto P_t g(x)$ and  \cite[Lemma 1.1, Chapter I]{MR0193671}, the semigroups agree on $C_0(\Omega)$.   

  \end{proof}

In the following corollary $\Omega_\delta$ denotes the one-point compactification of $\Omega$ whenever $\Omega$ is not compact.
\begin{corollary} Under assumption \cite[(H1)]{BKT20}, if the initial conditions of the Feller processes constructed in \cite[Lemma 3.14]{BKT20} converge weakly on $\Omega_\delta$ to the initial condition of $\Y{LR}$, then these processes converge weakly to $\Y{LR}$ on $D([0,\infty),\Omega_\delta)$.  This is true under \ref{H0} for the cases ND, NN and DN.
\end{corollary}
\begin{proof}
Recalling the canonical extension of a Feller process on $\Omega$ to $\Omega_\delta$ (see, e.g., \cite[page 12]{MR3156646}), this is an immediate consequence of Theorem \ref{thm:main_II}, \cite[Theorem 5.1]{BKT20} and \cite[Theorem 17.25]{MR1464694}. 

\end{proof}

Recall that if $g\in \mathcal D$, then  $\Gen g = 0$ if and only if $P_t g=g$ for all $t>0$, where $P$ is a strongly continuous semigroup  with generator  $(\Gen, \mathcal D)$. Then, from the domain representation of $(\Cl,{\rm NN})$ in \cite[Table 4]{BKT20}    we see that Theorem \ref{thm:main_II} implies the following corollary.
\begin{corollary}
Assume \ref{H0}. Then $\mu(\dd x)= 2^{-1}\dd x$ on $[-1,1]$   is an invariant probability measure for $\Y{NN}$, i.e. $\mathbb P_{\mu}[\Y{NN}_t\in \dd x]=\mu(\dd x)$ for all $t>0$.
\end{corollary}

Let us also observe  that  from the domain representation of $(\RLl,{\rm N^*N})$ in \cite[Table 4]{BKT20} (and recalling \cite[Remark 2.15]{BKT20})  we see that $\dd W( y+1)/W(2)$ is an invariant probability measure for $\Y{N^*N}$, which was proved in \cite[Theorem 2.ii]{MR1995924} (here $W\in C[0,\infty)$ is the  function with Laplace transform $1/\sym$).

\subsection{Resolvent measures for one-sided processes}\label{sec:resolvnew}
We provide below a new representation for the resolvent measures and an exit problem that involve fast-forwarding. To ease the comparison with \cite[Theorem 8.11]{MR2250061}, we state it  for the case of fast-forwarding the spectrally negative L\'evy process $-Y$. We assume that any function $g\in C[0,a]$ is extended by 0 on $\mathbb R\backslash [0,a] $ upon being convolved,  we denote by $\star$ the convolution operator and we define $Ig(x)=g\star 1(x)=\int_0^xg(y)\,\dd y$ ($x\ge0$) for the constant function $ 1\in C[0,a]$.  Define for any $q\ge0$ the \textit{operator scale function} $Z^{(q)}$ acting on   $g\in C[0,a]$  as 
 \[
 Z^{(q)}[g](x) = g(x)+\sum_{n=1}^\infty q^n (W\star)^n g (x),\quad x\in[0,a],
  \]
   and we also let $Z^{(q)}(x):= Z^{(q)}[ 1](x)$, where $W$ is the positive non-decreasing function with Laplace transform $\int_0^\infty e^{-\xi x}W(x)\,\dd x=\sym(\xi) ^{-1}$, $\Re\xi>0$, and define 
   \[W^{(q)}(x)=W(x)+\sum_{n=1}^\infty q^n (W\star)^{n}W(x),\quad x>0,\]
   extended by 0 on $(-\infty,0)$ (cf. \cite[Theorem 2.1]{MR3014147}).
   Note that $Z^{(q)}[g]$ defines an absolutely uniformly convergent series, in the sense that for all $x\in[0,a]$
  $$\Supnorm{Z^{(q)}[g]}{[0,x]}\le \Supnorm{Z^{(q)}[|g|]}{[0,x]}\le \|g\|_{C[0,x]}\left(1+q W(x)\sum_{n=0}^\infty \frac{(qW(x)x)^n}{n!}\right).$$
 
 \begin{remark}\label{rmk:scalefn}
 Denoting $\dd Z^{(q) }/\dd x$  by $(Z^{(q) })'$, observe that 
  \[
q Z^{(q)}[W]= (Z^{(q) })'= q  W^{(q)},
\]
and for $ 1,g\in C[0,a]$, 
\[
  qZ^{(q)}[W\star 1]=qIZ^{(q)}[W] =I(Z^{(q)})'=Z^{(q)}-1, 
  \]
  and $IZ^{(q)}[g]=Z^{(q)}[g\star 1]=g\star Z^{(q)}$. Also, we recall that $Z^{(q)}= E_+^{\sym,q}$ and $W=k_0^+(\cdot-1)$ in the notation of Part I (see \cite[Remark 2.15]{BKT20}).
 \end{remark}

\begin{corollary}[Fast-forwarding resolvents]\label{cor:resolvnew} Let $-Y$ be any recurrent spectrally negative process with paths of unbounded variation and no diffusion component.
\begin{enumerate}[(i)]
    \item Let  and   $\tau_0 = \inf\{t> 0: \Nr{a}(-Y)_{t}\not\in (0,a] \}$. Then, for any $q\ge 0$, the $q$-resolvent measure on $(0,a]$ at $x\in(0,a]$ of $\Nr{a}(-Y)$ killed at $\tau_0$  is
\[U_{q,x}^{{\rm DN}}(\dd y)=\left(\frac{W^{(q)}(x)}{Z^{(q)}(a)} Z^{(q)}(a-y)  -W^{(q)}(x-y)\right)\dd y,\]
and the corresponding exit problem allows the solution \[
 \mathbb E_x[e^{-q\tau_0}]=Z^{(q)}  (x)-\frac{q \int_0^aZ^{(q)}(z)\,\dd z }{Z^{(q)}(a)}  W^{(q)}(x)   .
 \]
\item For any   $q> 0$, the $q$-resolvent measure on $[0,a]$ of $\Nl{0}(\Nr{a}(-Y))$  at $x\in[0,a]$ is
\[U_{q,x}^{{\rm NN}}(\dd y)=\left(\frac{Z^{(q)}(x)}{q\int_0^a Z^{(q)}(z)\,\dd z}Z^{(q)}(a-y) -W^{(q)}(x-y)\right)\dd y.\]

\end{enumerate}
\begin{proof} Recall that \ref{H0} characterises $Y$, and so the same strategy of Theorem \ref{thm:main_II} (cf. resolvent for $\mathcal D(\Cr,{\rm ND})$ in   \cite[Table 4]{BKT20}) 
proves that for any $g\in C_0(0,a]$
\begin{equation}\label{eq:DNresolv}
 \mathbb E_x\left[\int_0^{\tau_0} e^{-qt}g\big(\Nr{a}(-Y)_t\big)\,\dd t\right]=Z^{(q)}\left[\frac{IZ^{(q)}[g](a)}{1+qZ^{(q)}[IW](a)}W- W\star g\right] (x),    
\end{equation}
and by the identities in Remark \ref{rmk:scalefn} the   identity for $U_{q,x}^{{\rm DN}}$ is proved. The same proof holds for $U_{q,x}^{{\rm NN}}$ but using the resolvent for  $\mathcal D(\Cr,{\rm NN})$ in   \cite[Table 4]{BKT20} (observing that in this table, for the coefficient $d$,   $2=I_{\pm}(\pm1)$). For the exit problem, note that for any $g\in C_0(0,a]$ and $q>0$, \eqref{eq:DNresolv} holds. Then, by Monotone Convergence Theorem applied to $0\le g_n\uparrow 1 $ everywhere on $(0,a]$,
\begin{align*}
\frac1q- \frac{\mathbb E_x\left[ e^{-q\tau_0} \right]}q& = \frac{IZ^{(q)}[ 1](a)}{1+qZ^{(q)}[IW](a)} Z^{(q)}[W](x)- Z^{(q)}[ W\star  1] (x) \\
& = \frac{IZ^{(q)} (a)}{Z^{(q)}(a)} W^{(q)}(x)+\frac1q- \frac{Z^{(q)}  (x)}q.
\end{align*} 
\end{proof} 
\end{corollary}
 
 \begin{example}
In the stable/fractional case $\phi(\dd y)=y^{-1-\alpha}/\Gamma(-\alpha)$, $\alpha\in(1,2)$, we have $W(x)=x^{\alpha-1}/\Gamma(\alpha)$ and $Z^{(q)}(x)= E_{\alpha,1}(q x^\alpha )$, where  $E_{\gamma,\beta}(x)=\sum_{n=0}^\infty  x^{ n}/\Gamma(\gamma n+\beta)$ is the Mittag-Leffler function for two positive parameters $\gamma$ and $\beta$ \cite[Exercises 8.2.ii and 8.2.iii]{MR2250061}. Then Corollary  \ref{cor:resolvnew}-(i) implies that for all $x\in[0,a]$,
 \[
 \mathbb E_x\big[e^{-q\tau_0^{{\rm DN}}}\big]=E_{\alpha,1}(q x^\alpha)  -\frac{q\int_0^aE_{\alpha,1}(q z^\alpha)\,\dd z}{E_{\alpha,1}(q a^\alpha)} x^{\alpha-1} E_{\alpha,\alpha}(qx^\alpha),
 \]
 and 
 \[ \mathbb E_x\big[\tau_0^{{\rm DN}}\big]= aW(x)-\int_0^xW(z)\,\dd z=a\frac{x^{\alpha-1}}{\Gamma(\alpha)}- \frac{x^{\alpha}}{\Gamma(\alpha+1)}
 ,\] where $\tau_0^{{\rm DN}} =\Tau{\Nr{a}(-Y)}{(0,a]}$. Instead the average time spent in $(0,a]$ by $\Nar{a}(-Y)$ before exiting $(0,a]$ is, by \cite[Theorem 8.10.i]{MR2250061}, of the form
  \[ \mathbb E_x\big[\tau_0^{{\rm DN^*}}\big]= \frac{W(a)}{W'(a)}W(x)-\int_0^xW(z)\,\dd z=\frac{a}{\alpha-1}\frac{x^{\alpha-1}}{\Gamma(\alpha)}- \frac{x^{\alpha}}{\Gamma(\alpha+1)}
 ,\]
 where $\tau_0^{{\rm DN^*}} =\Tau{\Nar{a}(-Y)}{(0,a]}$.
 Note that for each $x\in (0,a]$, 
 \[
 \mathbb E_x\big[\tau_0^{{\rm DN^*}}\big]\to \infty\quad \text{meanwhile}\quad  \mathbb E_x\big[\tau_0^{{\rm DN}}\big]\to a-x\quad \text{as }\alpha\downarrow1,\]
 suggesting that, in applications, modelling fast-forwarding boundary conditions (particles free to move in and out of the domain) incorrectly with  Neumann/reflecting boundary conditions can lead to significant prediction errors. For completeness, let us note that by \cite[Theorem 8.10.ii]{MR2250061},
   \[ \mathbb E_x\big[\tau_a^{{\rm ND}}\big]= \int_0^aW(z)\,\dd z-\int_0^xW(z)\,\dd z=\frac{a^{\alpha}}{\Gamma(\alpha+1)}- \frac{x^{\alpha}}{\Gamma(\alpha+1)}
 ,\]
where $\tau_a^{{\rm N^*D}} =\Tau{\Nal{0}(-Y)}{[0,a)}$ equals $\Tau{\Nl{0}(-Y)}{[0,a)}$ in law.
\end{example}

\section*{Acknowledgments} We want to mention that an important role for this work was played by a significant amount of correspondence with many professors. Therefore we thank  Florin Avram,  Martin Barlow, Jean Bertoin, Krzysztof Burdzy, Bruce Henry, Adam Jakubowski, Arturo Kohatsu-Higa, Vassili Kolokoltsov, \L ukasz Kruk, Andreas Kyprianou, James Pitman, Ren\'e  Schilling,    Steven  Shreve,  Jason Swanson and Ger\'onimo Uribe Bravo.
\appendix
\section*{Appendix}

\section{Additional proofs}\label{app:additional}
\begin{proof}[of Proposition \ref{prop:ff_meas}]
Recall that the cylinder sets $\mathcal A=\{\pi_t^{-1}(B_r(x)): t,r>0, x\in\mathbb R\}$ generate the $\sigma$-algebra of $ D([0,\infty),\mathbb R)$. Let $A\in \mathcal A$.  If $B_r(x)\cap (a,\infty)=\emptyset$ then $\Nl{a}^{-1}(A)=\emptyset$. Otherwise $B_r(x)\cap (a,\infty)\not=\emptyset$ and we consider two sub-cases, namely $a\not\in B_r(x)$ and $a\in B_r(x)$. For the first case  we claim that $\Nl{a}^{-1}( A)$ equals the intersection with $\setinftime{a}{\infty}$ of
 $$
\bigcup_{\substack{r\in \mathbb Q^+,\,\tilde x\in \mathbb Q:\\  B_{\tilde r}(\tilde x )\subsetneq B_r(x)}}  \bigcap _{n\in\mathbb N} \left(\bigcup_{t\le q\in \mathbb Q} \pi_q^{-1}(B_{\tilde r}(\tilde x ))\cap \Big\{f:  \Leb(s\in [0,q]:\dpath (s)> a)\in (t,t+1/n)\Big\}\right),
 $$
 which is a measurable set in $D_{a,\infty}([0,\infty),\mathbb R)$ by \eqref{eq:timeintunnel_measurable}.
 To prove this first observe that for a function $\dpath\in D_{a,\infty}([0,\infty),\mathbb R)$ and any $t>0$
 \begin{align*}
 &    \Nl{a}(\dpath )(t)\in B_r(x)\\
 \iff &\,\exists t'\ge t:\dpath  (t')\in B_r(x) \,\, \& \,\, \lambda(s\in[0,t']:\dpath(s) > a )=t \\
  \iff&\, \exists B_{\tilde r}(\tilde x)\subsetneq B_{ r}( x),\, \tilde r\in \mathbb Q^+,\,\tilde x\in \mathbb Q: \\
  &  \forall\, n\in\mathbb N\,\, \exists\, q\ge t,\,q\in \mathbb Q^+:\,\dpath  ( q)\in B_{\tilde r}(\tilde x)\,\,\&\,\, \lambda(s\in[0,q]:\dpath(s) > a )\in (t,t+1/n).
\end{align*} 
The first `$\iff$' is clear from the definition of $ \Nl{a}$ and $a\notin B_r(x)$. To prove the second `$\Rightarrow$'  observe that by right continuity there exists $B_{\tilde r}(\tilde x)$ and a sequence of rationals $q_n\downarrow t'$ such that $\dpath (q_n)\in B_{\tilde r}(\tilde x)$ for all $n$ and then it has to hold that $\lambda(s\in[0,q_n]:\dpath \ge a )\downarrow t$. To prove the second second `$\Leftarrow$' observe that we can select a sequence $q_n\ge t$ such that $\dpath (q_n)\in B_{\tilde r}(\tilde x)$ for all $n$ and 
$$
t< \lambda(s\in[0,q_n]:\dpath  (s)>a )< t+1/n.
$$
As $\dpath  \in \setinftime{a}{\infty}$ and $q_n\ge t$, we known $\{q_n:n\in\mathbb N\}$ is contained in some compact interval $[t,b]$, so that  there exists a $t'\ge t$ such that $q_{n_j}\to \tilde t$ for some subsequence. If there exists a $q_{n_m}< t'$,   then we get a contraddiction because
$$
t< \lambda(s\in[0,q_{n_m}]:\dpath (s)> a )\le \lim_{j\to\infty}  \lambda(s\in[0,q_{n_j}]:\dpath(s) > a ) =t.
$$
Then we can choose a further subsequence $q_{m}\downarrow  t'$ and by right continuity $\dpath (t')\in \overline{B_{\tilde r}}(\tilde x)\subset B_{ r}(  x)$, and we are done.\\
 It remains to treat the case $a\in B_r(x)$.   In this case we can select $\tilde r,\tilde x$ so that $\Nl{a}(\dpath ) \in A$ if and only if either  $\Nl{a}(\dpath )(t)\in B_{\tilde r}(\tilde x)\subset (a,\infty)$ or $\Nl{a}(\dpath )(t)=a$. The set of functions $\dpath$ that satisfy the first condition are measurable by the same argument as above. We can now conclude if we prove that for each $t>0$ the set  $ \Nl{a}^{-1}( \pi_t^{-1}(\{a\}))$ equals the intersection of $\setinftime{a}{\infty}$ with
 $$
 \bigcap_{n,m\in\mathbb 
 N}    \left(\bigcup_{t\le q\in \mathbb Q} \pi_q^{-1}(B_{1/m}(a))\cap \Big\{f:  \Leb(s\in [0,q]:\dpath (s)> a)\in\left(t,t+1/n\right)\Big\}\right).
 $$
So we prove that
 \begin{align*}
 &\,\Nl{a}(\dpath )(t)=a\\
  \iff &\, \exists \, t'\ge t:\dpath  (t')=a \, \,\&\, \, \lambda(s\in[0,t']:\dpath (s)> a )=t \\
  &\,\,\&\,\, \forall\, n\in\mathbb N\, \, \Leb(s\in[t',t'+1/n]:\dpath(s)> a )>0 \\
  \iff &\,\forall\, n,m\in\mathbb N\,\, \exists\, q\in \mathbb Q^+:\dpath  ( q)\in B_{1/m}(a)\,\,\&\,\, \lambda(s\in[0, q]:\dpath (s)>a )\in(t,t+1/n) .
\end{align*} 
The first `$\iff$' is immediate due to
\begin{align*}
&\,\inf\left\{s:\int_0^s\mathbf 1_{\{\dpath (z)> a\}}\dd z>t\right\}=t' \\
\iff &\, \int_0^{t'}\mathbf 1_{\{\dpath (z)> a\}}\dd z=t\quad\&\quad \forall\, n\in\mathbb N\,\,\, \int_{t'}^{t'+1/n}\mathbf 1_{\{\dpath (z)> a\}}\dd z>0.
\end{align*} 
 For the second `$\Rightarrow$', by right continuity of $\dpath $ we can take a sequence of rationals $q_n\downarrow t'$ such that $\dpath  ( q_m)\in B_{1/m}(a)$, then, because $\lim_{m\to\infty}\lambda(s\in[0, q_m]:\dpath (s)> a )=t$, we can find $n_m\ge m$ such that 
 \[
 t=\lambda(s\in[0, t']:\dpath (s)>a )<\lambda(s\in[0, q_{n_m}]:\dpath(s)> a )< t+1/n.
 \]
 For the second `$\Leftarrow$', select  a sequence $\{q_{n}:n\in\mathbb N\}$ such that $\dpath  ( q_n)\in B_{1/n}(a)$ and $\lambda(s\in[0, q_n]:\dpath (s)> a )\in(t,t+n^{-1})$. Then, again by recurrence of $\dpath $, the $q_n$'s live in a compact set so take a subsequence $\{q_{n_j}\}$ converging to some $t'$. Note that again the existence of one $q_{n_i}<t'$ leads to a contradiction because 
 $$
 t<\Leb(s\in[0, q_{n_{i}}]:\dpath  (s)> a ) \le  \lim_{j\to\infty}\Leb(s\in[0, q_{n_j}]:\dpath  (s)> a )=t,
 $$
thus we can select a further subsequence $q_{i}\downarrow t'$ so that,  as $i\to\infty$, $\dpath (q_{i})$ converges to $a$ but also to $\dpath (t')$ by right continuity, and clearly  for any $\epsilon>0$  there exists a $q_{i}$ $\epsilon$-close to $t'$ implying
$$\Leb(s\in[t', t'+\epsilon]:\dpath  (s)> a )\ge \Leb(s\in[t', t'+q_{i}]:\dpath  (s)> a )  >0.$$

 \end{proof}
 
\begin{proof}[of Proposition \ref{prop:FF_repr_ab}]
The two maps are clearly well-defined and measurability is immediate  as they are the composition of measurable maps by Proposition \ref{prop:ff_meas}.   Now observe that for any $x\in (a,b)$  there exists $t>0$  with 
\[
\Nr{b}(\Nl{a}(\dpath ))(t)=f\left((\AF^\dpath_a)^{-1}\left(\left(\AF^{f\left((\AF^\dpath_a)^{-1}\right)}_b\right)^{-1}(t)\right) \right)=x
\]
if and only if there exist $t,t'\ge 0$ such that $\dpath (t')=x$ and
\begin{equation}
    \label{eq:lastcondff} (\AF^\dpath_a)^{-1}\left(\left(\AF^{f\left((\AF^\dpath_a)^{-1}\right)}_b\right)^{-1}(t)\right)=t',
\end{equation} where we simplified notation by writing $(\AF^\dpath_a)^{-1}$ for the right inverse of $s\mapsto \AF^\dpath_a(s):=\int_0^{s}\indi{f(z)>a}\,\dd z$ and $\big(\AF^{f\left((\AF^\dpath_a)^{-1}\right)}_b\big)^{-1}$ for the right inverse of $s\mapsto\int_0^{s}\indi{f((\AF^\dpath_a)^{-1}(z))<b}\,\dd z$. 
Equation \eqref{eq:lastcondff} holds if and only if there exist  $t,t'\ge 0$ such that $\dpath (t')=x$ and
\[
\left(\AF^{f\left((\AF^\dpath_a)^{-1}\right)}_b\right)^{-1}(t)=\AF^\dpath_a(t')  \quad\text{and}\quad \int_{t'}^{t'+1/n} \mathbf1_{\{\dpath (z)> a\} }\dd z>0 \,\,\,\, \forall n\in\mathbb N,
\]
which in turn  holds if and only if there exist $t,t'\ge 0$ such that $\dpath (t')=x$ and
\begin{equation}\label{eq:equff}
t=\int_0^{\AF^\dpath_a(t')}\indi{f((\AF^\dpath_a)^{-1}(z))<b}\,\dd z  \quad\text{and}\quad \int_{\AF^\dpath_a(t')}^{\AF^\dpath_a(t')+1/n} \indi{f((\AF^\dpath_a)^{-1}(z))<b}\dd z>0\,\,\,\, \forall n\in\mathbb N,
\end{equation}
where for the inequalities we used that $t'$ is a point of increase for $\AF^\dpath_a$. 
The first identity in \eqref{eq:equff} equals $t=\int_0^{t'}\indi{a<f(z)<b}\,\dd z $, meanwhile the inequality can be rewritten for all large $n$ as
\begin{align*}
0&< \int_{\AF^\dpath_a(t')}^{\AF^\dpath_a(t')+1/n} \indi{f((\AF^\dpath_a)^{-1}(z))<b}\dd z\\
&=\int_{t'}^{t'+1/n} \indi{f((\AF^\dpath_a)^{-1}(\AF^\dpath_a(t')+(z-t')))<b}\dd z\\
&=\int_{t'}^{t'+1/n} \indi{a<f(z)<b}\dd z,
\end{align*}
because for all small $\epsilon>0$ we have $f(z)>a$ for all $z\in[t', t'+\epsilon)$ which implies $(\AF^\dpath_a)^{-1}(\AF^\dpath_a(t')+\epsilon))=(\AF^\dpath_a)^{-1}(\AF^\dpath_a(t'+\epsilon)))=t'+\epsilon$. And so we proved that on $(a,b)$
\[
\Nr{b}(\Nl{a}(\dpath )) = f\left((\AF^\dpath_{(a,b)} )^{-1}\right).
\]
Clearly the same argument above proves that $\Nl{a}(\Nr{b}(\dpath))=\dpath((\AF^\dpath_{(a,b)} )^{-1})$ on $(a,b)$. Finally, to show the paths agree at times when they equal $a$ it is now enough to observe that $\Nr{b}(\Nl{a}(\dpath ))(t)=a $ implies that there exists a sequence of decreasing times $t_n\downarrow t$ such that $\Nr{b}(\Nl{a}(\dpath ))(t_n)=f((\AF^\dpath_{(a,b)} )^{-1})(t_n)>a$ and we can conclude by right continuity of the fast-forwarded paths. The exact same argument holds for $b$ and we are done.
 
\end{proof}

  \begin{proof}[Measurability of \eqref{eq:timeintunnel_measurable}]
We first prove that $(\dpath,t)\mapsto \pi_t(f)$ is $\sigma(\mathcal D_{J_1}\times \mathcal B(\R^+))\backslash \mathcal B(\mathbb R)$ measurable. To do we first show that for any $\epsilon>0$   the function
\[
(\dpath,t)\mapsto h_\epsilon(\dpath,t)=\frac1\epsilon \int_t^{t+\epsilon} \dpath(s)\,\dd s,
\]
is continuous. Let $(\dpath_n,t_n)\to (\dpath,t)$. Then 
\begin{align*}
    \epsilon|h_\epsilon(\dpath_n,t_n)- h_\epsilon(\dpath,t)| &=\left|\int_{t_n}^{t_n+\epsilon} \dpath_n(s)\,\dd s- \int_t^{t+\epsilon} \dpath(s)\,\dd s\right|\\
    &=\left|\int_{t_n}^{t_n+\epsilon} \dpath_n(s)\,\dd s-\int_{t}^{t+\epsilon} \dpath_n(s)\,\dd s  +\int_{t}^{t+\epsilon} \dpath_n(s)- \dpath(s)\,\dd s\right| \\
       &\le \sup_n\supnorm{\dpath_n}{[t-\epsilon,t+2\epsilon]}2|t_n-t| +\int_{t}^{t+\epsilon} \left| \dpath_n(s)- \dpath(s)\right|\,\dd s 
\end{align*}
which vanishes because convergence in $D([0,\infty),\mathbb R)$ imples convergence almost everywhere and the sequence must be uniformly bounded. By the right continuity of $\dpath$ $h_\epsilon(\dpath,t)\to \pi_t(\dpath)$ as $\epsilon\to 0$, and thus $(\dpath,t)\mapsto \pi_t(\dpath)$ is measurable. Now we write
\[
\dpath\mapsto\Leb\big(t\in[0,\infty): \dpath(t)\in B\big)=\int_{0}^\infty \indi{ \pi_t(\dpath)\in B}\,\dd t = \int_{0}^\infty \varphi(f,t)\,\dd t.
\]
Then the above function is measurable  by standard results in measure theory \cite[Chapter 7]{Halmos50} as  $\varphi:D([0,\infty),\mathbb R)\times [0,\infty)\to \mathbb R$ is a non-negative bounded measurable function.
\end{proof}


\begin{thebibliography}{10}
\expandafter\ifx\csname url\endcsname\relax
  \def\url#1{\texttt{#1}}\fi
\expandafter\ifx\csname urlprefix\endcsname\relax\def\urlprefix{URL }\fi
\expandafter\ifx\csname href\endcsname\relax
  \def\href#1#2{#2} \def\path#1{#1}\fi

\bibitem{MR2218073}
A.~A. Kilbas, H.~M. Srivastava, J.~J. Trujillo,
 Theory
  and applications of fractional differential equations, Vol. 204 of
  North-Holland Mathematics Studies, Elsevier Science B.V., Amsterdam, 2006.

\bibitem{MR1658022}
I.~Podlubny,
  Fractional
  differential equations, Vol. 198 of Mathematics in Science and Engineering,
  Academic Press, Inc., San Diego, CA, 1999.

\bibitem{MR2090004}
R.~Metzler, J.~Klafter,  The
  restaurant at the end of the random walk: recent developments in the
  description of anomalous transport by fractional dynamics, J. Phys. A
  37~(31) (2004) R161--R208.
\newblock \href {https://doi.org/10.1088/0305-4470/37/31/R01}
  {\path{doi:10.1088/0305-4470/37/31/R01}}.

\bibitem{MR2884383}
M.~M. Meerschaert, A.~Sikorskii,
   Stochastic models for fractional
  calculus, Vol.~43 of De Gruyter Studies in Mathematics, Walter de Gruyter \&
  Co., Berlin, 2012.
\newblock \href {https://doi.org/10.1515/9783110258165}
  {\path{doi:10.1515/9783110258165}}.

\bibitem{Benson2000b}
D.~A. Benson, S.~W. Wheatcraft, M.~M. Meerschaert,
   The fractional-order governing
  equation of {L}\'evy motion, Water Resour. Res. 36 (2000) 1413--1423.
\newblock \href {https://doi.org/10.1029/2000WR900032}
  {\path{doi:10.1029/2000WR900032}}.

\bibitem{Schumer2009}
R.~Schumer, M.~M. Meerschaert, B.~Baeumer,
  Fractional advection-dispersion
  equations for modeling transport at the {E}arth surface, J. Geophys. Res.
  114~(F4) (2009).
\newblock \href {https://doi.org/10.1029/2008JF001246}
  {\path{doi:10.1029/2008JF001246}}.

\bibitem{T20}
K.~Taira,  Boundary value problems
  and {M}arkov processes, 3rd Edition, Vol. 1499 of Lecture Notes in
  Mathematics, Springer, Cham, 2020.
\newblock \href {https://doi.org/10.1007/97830304587881}
  {\path{doi:10.1007/97830304587881}}.

\bibitem{MR2343205}
A.~E. Kyprianou, W.~Schoutens (Eds.),
  Exotic option pricing and advanced {L}\'{e}vy models, John Wiley \& Sons, Ltd.,
  Chichester, 2005.

\bibitem{MR3156646}
B.~B\"{o}ttcher, R.~Schilling, J.~Wang,
 L\'{e}vy matters {III},
  Vol. 2099 of Lecture Notes in Mathematics, Springer, Cham, 2013.
\newblock \href {https://doi.org/10.1007/978-3-319-02684-8}
  {\path{doi:10.1007/978-3-319-02684-8}}.

\bibitem{MR2780345}
V.~N. Kolokoltsov, Markov
  processes, semigroups and generators, Vol.~38 of De Gruyter Studies in
  Mathematics, Walter de Gruyter \& Co., Berlin, 2011.
\newblock \href {https://doi.org/10.1515/9783110250114}
  {\path{doi:10.1515/9783110250114}}.

\bibitem{MR3987876}
V.~N. Kolokoltsov, The
  probabilistic point of view on the generalized fractional partial
  differential equations, Fract. Calc. Appl. Anal. 22~(3) (2019) 543--600.
\newblock \href {https://doi.org/10.1515/fca-2019-0033}
  {\path{doi:10.1515/fca-2019-0033}}.

\bibitem{D19}
Q.~Du, Nonlocal Modeling,
  Analysis, and Computation, SIAM, 2019.
\newblock \href {https://doi.org/10.1137/1.9781611975628}
  {\path{doi:10.1137/1.9781611975628}}.

\bibitem{Zhang2006a}
Y.~Zhang, D.~A. Benson, M.~M. Meerschaert, E.~M. LaBolle, H.-P. Scheffler,
  Random walk approximation
  of fractional-order multiscaling anomalous diffusion, Phys. Rev. E 74~(2)
  (2006).
\newblock \href {https://doi.org/10.1103/PhysRevE.74.026706}
  {\path{doi:10.1103/PhysRevE.74.026706}}.

\bibitem{MR0345224}
K.~It\^{o}, H.~P. McKean, 
  Diffusion processes and
  their sample paths, Springer-Verlag, Berlin-New York, 1974.
\newblock \href {https://doi.org/10.1007/978-3-642-62025-6}
  {\path{doi:10.1007/978-3-642-62025-6}}.

\bibitem{Peskir15}
G.~Peskir, On boundary behaviour of one-dimensional diffusions: From {B}rown to
  {F}eller and beyond, William Feller-Selected Papers II 2187 (2015) 77--93.

\bibitem{MR3467345}
L.~N. Andersen, S.~Asmussen, P.~W. Glynn, M.~Pihlsg{\aa}rd,
 L\'{e}vy processes with
  two-sided reflection, in: L\'{e}vy matters {V}, Vol. 2149 of Lecture Notes
  in Math., Springer, Cham, 2015, pp. 67--182.
\newblock \href {https://doi.org/10.1007/978-3-319-23138-9_2}
  {\path{doi:10.1007/978-3-319-23138-9_2}}.

\bibitem{MR2006232}
K.~Bogdan, K.~Burdzy, Z.-Q. Chen,
  Censored stable processes,
  Probab. Theory Related Fields 127~(1) (2003) 89--152.
\newblock \href {https://doi.org/10.1007/s00440-003-0275-1}
  {\path{doi:10.1007/s00440-003-0275-1}}.

\bibitem{MR3217703}
G.~Barles, E.~Chasseigne, C.~Georgelin, E.~R. Jakobsen,
  On {N}eumann type
  problems for nonlocal equations set in a half space, Trans. Amer. Math. Soc.
  366~(9) (2014) 4873--4917.
\newblock \href {https://doi.org/10.1090/S0002-9947-2014-06181-3}
  {\path{doi:10.1090/S0002-9947-2014-06181-3}}.

\bibitem{MR3413862}
B.~Baeumer, M.~Kov\'{a}cs, M.~M. Meerschaert, R.~L. Schilling, P.~Straka,
  Reflected spectrally negative stable
  processes and their governing equations, Trans. Amer. Math. Soc. 368~(1)
  (2016) 227--248.
\newblock \href {https://doi.org/10.1090/tran/6360}
  {\path{doi:10.1090/tran/6360}}.

\bibitem{MR3323906}
O.~Defterli, M.~D'Elia, Q.~Du, M.~Gunzburger, R.~Lehoucq, M.~M. Meerschaert,
  Fractional diffusion on bounded
  domains, Fract. Calc. Appl. Anal. 18~(2) (2015) 342--360.
\newblock \href {https://doi.org/10.1515/fca-2015-0023}
  {\path{doi:10.1515/fca-2015-0023}}.

\bibitem{MR4112713}
L.~D\"{o}ring, A.~E. Kyprianou,
  Entrance and exit at infinity for
  stable jump diffusions, Ann. Probab. 48~(3) (2020) 1220--1265.
\newblock \href {https://doi.org/10.1214/19-AOP1389}
  {\path{doi:10.1214/19-AOP1389}}.

\bibitem{MR3582209}
P.~Patie, Y.~Zhao, Spectral
  decomposition of fractional operators and a reflected stable semigroup, J.
  Differential Equations 262~(3) (2017) 1690--1719.
\newblock \href {https://doi.org/10.1016/j.jde.2016.10.026}
  {\path{doi:10.1016/j.jde.2016.10.026}}.

\bibitem{MR3720847}
B.~Baeumer, M.~Kov\'{a}cs, H.~Sankaranarayanan,
 Fractional partial
  differential equations with boundary conditions, J. Differential Equations
  264~(2) (2018) 1377--1410.
\newblock \href {https://doi.org/10.1016/j.jde.2017.09.040}
  {\path{doi:10.1016/j.jde.2017.09.040}}.

\bibitem{BKT20}
B.~Baeumer, M.~Kov\'acs, L.~Toniazzi,
 Boundary conditions for nonlocal
  one-sided pseudo-differential operators and the associated stochastic
  processes {I}, arXiv preprint arXiv:2012.10864 (2020).

\bibitem{MR1406564}
J.~Bertoin,
  L\'{e}vy
  processes, Vol. 121 of Cambridge Tracts in Mathematics, Cambridge University
  Press, Cambridge, 1996.

\bibitem{MR2250061}
A.~E. Kyprianou, Introductory
  lectures on fluctuations of {L}\'{e}vy processes with applications,
  Universitext, Springer-Verlag, Berlin, 2006.
\newblock \href {https://doi.org/10.1007/978-3-540-31343-4}
  {\path{doi:10.1007/978-3-540-31343-4}}.

\bibitem{MR3014147}
A.~Kuznetsov, A.~E. Kyprianou, V.~Rivero,
   The theory of scale
  functions for spectrally negative {L}\'{e}vy processes, in: L\'{e}vy matters
  {II}, Vol. 2061 of Lecture Notes in Math., Springer, Heidelberg, 2012, pp.
  97--186.
\newblock \href {https://doi.org/10.1007/978-3-642-31407-0_2}
  {\path{doi:10.1007/978-3-642-31407-0_2}}.

\bibitem{MR2023021}
F.~Avram, A.~E. Kyprianou, M.~R. Pistorius,
  Exit problems for spectrally
  negative {L}\'{e}vy processes and applications to ({C}anadized) {R}ussian
  options, Ann. Appl. Probab. 14~(1) (2004) 215--238.
\newblock \href {https://doi.org/10.1214/aoap/1075828052}
  {\path{doi:10.1214/aoap/1075828052}}.

\bibitem{CG00}
J.~H. Cushman, T.~R. Ginn,
  Fractional advection‐dispersion
  equation: a classical mass balance with convolution‐{F}ickian flux, Water
  Resour. Res. 36~(12) (2000) 3763--3766.
\newblock \href {https://doi.org/10.1029/2000WR900261}
  {\path{doi:10.1029/2000WR900261}}.

\bibitem{CW03}
P.~Carr, L.~Wu, The finite
  moment log stable process and option pricing, J. Finance 58~(2) (2003)
  753--777.
\newblock \href {https://doi.org/10.1111/1540-6261.00544}
  {\path{doi:10.1111/1540-6261.00544}}.

\bibitem{MR1919609}
F.~Avram, T.~Chan, M.~Usabel,
  On the valuation of
  constant barrier options under spectrally one-sided exponential {L}\'{e}vy
  models and {C}arr's approximation for {A}merican puts, Stochastic Process.
  Appl. 100 (2002) 75--107.
\newblock \href {https://doi.org/10.1016/S0304-4149(02)00104-7}
  {\path{doi:10.1016/S0304-4149(02)00104-7}}.

\bibitem{MR0391297}
A.~A. Borovkov,  Stochastic
  processes in queueing theory, Springer-Verlag, New York-Berlin, 1976.
\newblock \href {https://doi.org/10.1007/978-1-4612-9866-3}
  {\path{doi:10.1007/978-1-4612-9866-3}}.

\bibitem{MR1492990}
N.~U. Prabhu, Stochastic
  storage processes, 2nd Edition, Vol.~15 of Applications of Mathematics (New
  York), Springer-Verlag, New York, 1998.
\newblock \href {https://doi.org/10.1007/978-1-4612-1742-8}
  {\path{doi:10.1007/978-1-4612-1742-8}}.

\bibitem{MR2577834}
B.~Baeumer, M.~M. Meerschaert,
  Tempered stable {L}\'{e}vy
  motion and transient super-diffusion, J. Comput. Appl. Math. 233~(10) (2010)
  2438--2448.
\newblock \href {https://doi.org/10.1016/j.cam.2009.10.027}
  {\path{doi:10.1016/j.cam.2009.10.027}}.

\bibitem{MR3342453}
F.~Sabzikar, M.~M. Meerschaert, J.~Chen,
  Tempered fractional
  calculus, J. Comput. Phys. 293 (2015) 14--28.
\newblock \href {https://doi.org/10.1016/j.jcp.2014.04.024}
  {\path{doi:10.1016/j.jcp.2014.04.024}}.

\bibitem{MR1175272}
J.~Bertoin, An extension of
  {P}itman's theorem for spectrally positive {L}\'{e}vy processes, Ann.
  Probab. 20~(3) (1992) 1464--1483.

\bibitem{MR2054585}
M.~R. Pistorius, On exit and ergodicity of the spectrally one-sided {L}\'{e}vy process reflected
  at its infimum, J. Theoret. Probab. 17~(1) (2004) 183--220.
\newblock \href {https://doi.org/10.1023/B:JOTP.0000020481.14371.37}
  {\path{doi:10.1023/B:JOTP.0000020481.14371.37}}.

\bibitem{MR1995924}
M.~R. Pistorius,  On doubly
  reflected completely asymmetric {L}\'{e}vy processes, Stochastic Process.
  Appl. 107~(1) (2003) 131--143.
\newblock \href {https://doi.org/10.1016/S0304-4149(03)00049-8}
  {\path{doi:10.1016/S0304-4149(03)00049-8}}.

\bibitem{MR2126964}
A.~E. Kyprianou, Z.~Palmowski,
  A martingale review of
  some fluctuation theory for spectrally negative {L}\'{e}vy processes, in:
  S\'{e}minaire de {P}robabilit\'{e}s {XXXVIII}, Vol. 1857 of Lecture Notes in
  Math., Springer, Berlin, 2005, pp. 16--29.
\newblock \href {https://doi.org/10.1007/978-3-540-31449-3_3}
  {\path{doi:10.1007/978-3-540-31449-3_3}}.

\bibitem{MR2126963}
R.~A. Doney, Some excursion
  calculations for spectrally one-sided {L}\'{e}vy processes, in:
  S\'{e}minaire de {P}robabilit\'{e}s {XXXVIII}, Vol. 1857 of Lecture Notes in
  Math., Springer, Berlin, 2005, pp. 5--15.
\newblock \href {https://doi.org/10.1007/978-3-540-31449-3_2}
  {\path{doi:10.1007/978-3-540-31449-3_2}}.

\bibitem{MR2126965}
M.~R. Pistorius, A
  potential-theoretical review of some exit problems of spectrally negative
  {L}\'{e}vy processes, in: S\'{e}minaire de {P}robabilit\'{e}s {XXXVIII},
  Vol. 1857 of Lecture Notes in Math., Springer, Berlin, 2005, pp. 30--41.
\newblock \href {https://doi.org/10.1007/978-3-540-31449-3_4}
  {\path{doi:10.1007/978-3-540-31449-3_4}}.

\bibitem{MR4158667}
F.~Avram, D.~Grahovac, C.~Vardar-Acar,
 The {$W$}, {$Z$} scale functions
  kit for first passage problems of spectrally negative {L}\'{e}vy processes,
  and applications to control problems, ESAIM Probab. Stat. 24 (2020)
  454--525.
\newblock \href {https://doi.org/10.1051/ps/2019022}
  {\path{doi:10.1051/ps/2019022}}.

\bibitem{MR2946445}
J.~Ivanovs, Z.~Palmowski,
 Occupation densities in
  solving exit problems for {M}arkov additive processes and their reflections,
  Stochastic Process. Appl. 122~(9) (2012) 3342--3360.
\newblock \href {https://doi.org/10.1016/j.spa.2012.05.016}
  {\path{doi:10.1016/j.spa.2012.05.016}}.

\bibitem{MR3301294}
J.~Ivanovs, Potential measures of
  one-sided {M}arkov additive processes with reflecting and terminating
  barriers, J. Appl. Probab. 51~(4) (2014) 1154--1170.
\newblock \href {https://doi.org/10.1239/jap/1421763333}
  {\path{doi:10.1239/jap/1421763333}}.

\bibitem{MR1876437}
W.~Whitt, Stochastic-process limits,
  Springer Series in Operations Research, Springer-Verlag, New York, 2002.
\newblock \href {https://doi.org/10.1007/b97479} {\path{doi:10.1007/b97479}}.

\bibitem{MR3896857}
A.~Lambert, G.~Uribe~Bravo,  Totally
  ordered measured trees and splitting trees with infinite variation,
  Electron. J. Probab. 23 (2018) Paper No. 120, 41.
\newblock \href {https://doi.org/10.1214/18-EJP251}
  {\path{doi:10.1214/18-EJP251}}.

\bibitem{MR1943877}
J.~Jacod, A.~N. Shiryaev,
  Limit theorems for
  stochastic processes, 2nd Edition, Vol. 288 of Grundlehren der
  Mathematischen Wissenschaften, Springer-Verlag, Berlin, 2003.
\newblock \href {https://doi.org/10.1007/978-3-662-05265-5}
  {\path{doi:10.1007/978-3-662-05265-5}}.

\bibitem{MR838085}
S.~N. Ethier, T.~G. Kurtz, Markov
  processes, Wiley Series in Probability and Mathematical Statistics, John
  Wiley \& Sons, Inc., New York, 1986.
\newblock \href {https://doi.org/10.1002/9780470316658}
  {\path{doi:10.1002/9780470316658}}.

\bibitem{MR0264757}
R.~M. Blumenthal, R.~K. Getoor, Markov processes and potential theory, Pure and
  Applied Mathematics, Vol. 29, Academic Press, New York-London, 1968.

\bibitem{Stroock13}
D.~W. Stroock, An introduction to {M}arkov processes, Vol. 230, Springer
  Science \& Business Media, 2013.

\bibitem{MR1700749}
P.~Billingsley, Convergence of
  probability measures, 2nd Edition, Wiley Series in Probability and
  Statistics, John Wiley \& Sons, Inc., New York, 1999, a Wiley-Interscience
  Publication.
\newblock \href {https://doi.org/10.1002/9780470316962}
  {\path{doi:10.1002/9780470316962}}.

\bibitem{MR0233396}
P.~Billingsley, Convergence of probability measures, John Wiley \& Sons, Inc.,
  New York-London-Sydney, 1968.

\bibitem{MR2349573}
L.~Kruk, J.~Lehoczky, K.~Ramanan, S.~Shreve,
 An explicit formula for the
  {S}korokhod map on {$[0,a]$}, Ann. Probab. 35~(5) (2007) 1740--1768.
\newblock \href {https://doi.org/10.1214/009117906000000890}
  {\path{doi:10.1214/009117906000000890}}.

\bibitem{MR2479479}
B.~Wu, On the weak convergence
  of subordinated systems, Statist. Probab. Lett. 78~(18) (2008) 3203--3211.
\newblock \href {https://doi.org/10.1016/j.spl.2008.06.006}
  {\path{doi:10.1016/j.spl.2008.06.006}}.

\bibitem{MR0193671}
E.~B. Dynkin, Markov processes. {V}ols. {I}, {II}, Vol. 122, Academic Press
  Inc., Publishers, New York; Springer-Verlag, Berlin-G\"{o}ttingen-Heidelberg,
  1965.

\bibitem{MR958195}
M.~T. Barlow,
  Necessary
  and sufficient conditions for the continuity of local time of {L}\'{e}vy
  processes, Ann. Probab. 16~(4) (1988) 1389--1427.

\bibitem{MR1739520}
K.-I. Sato, L\'{e}vy processes and infinitely divisible distributions, Vol.~68
  of Cambridge Studies in Advanced Mathematics, Cambridge University Press,
  Cambridge, 1999.

\bibitem{MR1464694}
O.~Kallenberg, Foundations of modern
  probability, Probability and its Applications (New York), Springer-Verlag,
  New York, 1997.
\newblock \href {https://doi.org/10.1007/b98838} {\path{doi:10.1007/b98838}}.

\bibitem{Halmos50}
P.~R. Halmos, Measure theory, University series in higher mathematics, New
  York, Van Nostrand, 1950.

\end{thebibliography}

\end{document}